\documentclass[11pt,amssymb]{amsart}
\usepackage{amssymb}
\DeclareFontFamily{OT1}{rsfs}{}

\DeclareFontShape{OT1}{rsfs}{n}{it}{<-> rsfs10}{}
\DeclareMathAlphabet{\mathscr}{OT1}{rsfs}{n}{it}

\newcommand{\Z}{{\mathbb Z}}

\newcommand{\C}{{\mathbb C}}

\newcommand{\Q}{{\mathbb Q}}

\newcommand{\cD}{\mathscr{D}}
\newcommand{\R}{{\mathbb R}}

\newcommand{\cF}{\mathscr{F}}
\newcommand{\mcK}{\mathcal{K}}
\newcommand{\Tr}{\mathrm{Tr}}
\newcommand{\cK}{\mathscr{K}}
\newcommand{\cA}{\mathscr{A}}

\newcommand{\cL}{\mathscr{L}}
\newcommand{\cE}{\mathscr{E}}
\newcommand{\mcG}{\mathcal{G}}

\newcommand{\cT}{\mathscr{T}}

\newcommand{\fX}{\mathfrak{X}}
\newcommand{\cO}{\mathscr{O}}

\newcommand{\Ga}{\mathrm{Gal}}
\newtheorem{thm}{Theorem}[section]
\newtheorem{lemma}[thm]{Lemma}
\newtheorem{prop}[thm]{Proposition}
\newtheorem{cor}[thm]{Corollary}

\newcommand{\cG}{\mathscr{G}}

\font\brus=wncyr10.240pk scaled 1200 .240pk

\begin{document}
\title[Lengths of closed geodesics]{On the fields generated by the lengths of closed geodesics in
locally symmetric spaces}

\author[Prasad]{Gopal Prasad}
\author[Rapinchuk]{Andrei S. Rapinchuk}

\address{Department of Mathematics, University of Michigan, Ann
Arbor, MI 48109}

\email{gprasad@umich.edu}

\address{Department of Mathematics, University of Virginia,
Charlottesville, VA 22904}

\email{asr3x@virginia.edu}

\maketitle

\section{Introduction}\label{S:I}

This paper is a sequel to our paper \cite{PR6} where we introduced
the notion of weak commensurability of Zariski-dense subgroups of
semi-simple algebraic groups and used our analysis of this
relationship to answer some differential-geometric questions about 
length-commensurable and isospectral locally symmetric spaces
that have received considerable amount of attention in recent
years (cf.\,\cite{CHLR}, \cite{Re1}; a detailed survey is given in
\cite{PR0}). More precisely, given a Riemannian manifold $M,$ 
the (weak) {\it length spectrum} $L(M)$ is the set of
lengths of all closed geodesics in $M,$ and  two Riemannian
manifolds $M_1$ and $M_2$ are said to be {\it iso-length} if $L(M_1) =
L(M_2),$ and {\it length-commensurable} if $\Q \cdot L(M_1) = \Q
\cdot L(M_2).$ It was shown in \cite{PR6} that
length-commensurability has strong consequences, one of which is
that length-commensurable arithmetically defined locally symmetric
spaces of certain types are necessarily commensurable, i.e. they have a
common finite-sheeted cover. In the current paper, we will study the following
two interrelated questions:
{\it Suppose that (locally symmetric spaces) $M_1$ and $M_2$ are
\emph{not} length-commensurable. Then

\vskip2mm

\noindent {\rm (1)} \parbox[t]{12cm}{How different are the sets
$L(M_1)$ and $L(M_2)$ $($or the sets  $\Q \cdot L(M_1)$ and
$\Q \cdot L(M_2)$$)$? }

\vskip2mm

\noindent {\rm (2)} \parbox[t]{13cm}{Can $L(M_1)$ and $L(M_2)$ be
related in \emph{any} reasonable way?}}

\vskip2mm

\noindent One can ask a variety of specific questions that fit the
general framework provided by (1) and (2): for example, can
$L(M_1)$ and $L(M_2)$ differ only in a finite number of elements, in
other words, can the symmetric difference \mbox{$L(M_1) \vartriangle
L(M_2)$} be finite? Regarding (2), the relationship between
$L(M_1)$ and $L(M_2)$ that makes most sense geometrically is that of
{\it similarity}, requiring that there be a real number $\alpha > 0$
such that $$L(M_2) = \alpha \cdot L(M_1) \ \ \text{(or} \ \Q \cdot
L(M_2) = \alpha \cdot \Q \cdot L(M_1)\ \text{),}$$ which
geometrically means that $M_1$ and $M_2$ can be made iso-length
(resp., length-commensurable) by scaling the metric on one
of them. At the same time, one can consider more general
relationships with less apparent geometric context like {\it
polynomial equivalence} which means that there exist polynomials
$p(x_1, \ldots , x_s)$ and $q(y_1, \ldots , y_t)$ with real
coefficients such that for any $\lambda \in L(M_1)$ one can find
$\mu_1, \ldots , \mu_s$ $\in L(M_2)$ so that $\lambda = p(\mu_1,
\ldots , \mu_s),$ and conversely, for any $\mu \in L(M_2)$ there
exist $\lambda_1, \ldots , \lambda_t \in L(M_1)$ such that $\mu =
q(\lambda_1, \ldots , \lambda_t).$ Our results show, in particular,
that for  most arithmetically defined locally symmetric spaces the
fact that they are not length-commensurable implies that the sets
$L(M_1)$ and $L(M_2)$ differ very significantly and in fact cannot
be related by any generalized form of polynomial equivalence (cf.\,\S
\ref{S:FGeod}).

%Actually, our techniques enable us to study possible relationships
%of the type described above between the weak-length spectra of not just
%two, but several, locally symmetric spaces. At present, results in
%differential geometry relating certain geometric invariants of more than two
%Riemannian manifolds simultaneously are rare.  An example of such a result
%is the following: If one fixes a
%number field $k$ then there exist a finite collection $M_1, \ldots ,
%M_r$ of arithmetic hyperbolic 3-manifolds of finite volume having $k$ as the
%invariant trace field such that for any finite-volume hyperbolic
%3-manifold $M$ with $k$ as the invariant trace field, the volume
%$v(M)$ is a rational linear combination of the volumes $v(M_1),
%\ldots , v(M_r)$ (cf.\,\cite{MR}, Theorem 12.7.1). In this paper, we will
%show that given a finite collection $M_1, \ldots
%, M_r$ of arithmetically defined locally symmetric spaces,  under
%very general assumptions $L(M_1), \ldots ,
%L(M_r)$ cannot possibly satisfy any nontrivial polynomial relation
%unless two of the spaces are length-commensurable (cf.\:\S
%\ref{S:FGeod}).

\vskip2mm

To formalize the idea of ``polynomial relations" between the weak
length spectra of Riemannian manifolds, we need to introduce some
additional notations and definitions. For a Riemannian manifold $M,$
we let $\cF(M)$ denote the subfield of $\R$ generated by the set
$L(M).$ Given two Riemannian manifolds $M_1$ and $M_2$, for $i \in
\{1 , 2\}$, we set $\cF_i = \cF(M_i)$ and consider the following
condition

\vskip3mm

\noindent $(T_i)$ \parbox[t]{13cm}{the compositum $\cF_1 \cF_2$
has infinite transcendence degree over the field $\cF_{3-i}$.}

\vskip3mm

\noindent In simple terms, the fact that condition $(T_i)$ holds
means that $L(M_i)$ contains ``many'' elements which are
algebraically independent of all the elements of $L(M_{3-i})$. The goal of  this paper is to prove 
that $(T_i)$ indeed holds for at least one $i \in \{1 , 2\}$
in various situations where $M_1$ and $M_2$ are pairwise
non-length-commensurable locally symmetric spaces. These results can
be used to prove a number of results on the nonexistence of
nontrivial dependence between the weak length spectra along the
lines indicated above - cf.\:\S \ref{S:FGeod}. Here we only mention
that $(T_i)$ implies the following condition

\vskip2mm

\noindent $(N_i)$ \parbox[t]{13cm}{$L(M_i)\not\subset A \cdot  \Q
\cdot L(M_{3-i})$ for any finite set $A$ of real numbers,}

\vskip2mm

\noindent which informally means that the weak length spectrum of
$M_i$ is ``very far''from being similar to the length spectrum of $M_{3-i}$.

\vskip2mm

To give the precise statements of our main results, we need to fix
some notations most of which will be used throughout the paper. Let
$G_1$ and $G_2$ be connected absolutely almost simple real algebraic groups such that $\mathcal{G}_i := G_i(\R)$ is noncompact for both $i = 1$ and $2$. (In \S\S 2-5 we will assume that both $G_1$ and $G_2$ are of adjoint type.) 
We fix a maximal compact subgroup $\mathcal{K}_i$ of   
$ \mathcal{G}_i,$ and let $\fX_i = \mathcal{K}_i \backslash
\mathcal{G}_i$ denote the associated symmetric space. Furthermore,
let $\Gamma_i \subset \mathcal{G}_i$ be a discrete torsion-free
Zariski-dense subgroup, and let $\fX_{\Gamma_i} := \fX_i/\Gamma_i$
be the corresponding locally symmetric space. Set $M_i =
\fX_{\Gamma_i}$ and $\cF_i = \cF(M_i)$. We also let $K_{\Gamma_i}$
denote the subfield of $\R$ generated by the traces $\mathrm{Tr}\:
\mathrm{Ad}\: \gamma$ for $\gamma \in \Gamma_i.$ Let $w_i$ be the
order of the (absolute) Weyl group of $G_i$.
%We set
%$$
%w = \max_{i \leqslant r} w_i \ \ \text{and} \ \ I = \{ i \leqslant r
%\: | \: w_i = w \}.
%$$
%Furthermore, we pick $i_0 \in I$ such that the field
%$K_{\Gamma_{i_0}}$ is maximal (under inclusion) among $K_{\Gamma_i}$
%for $i \in I,$ and set
%$$
%I(i_0) = \{ i \in I \: | \: K_{\Gamma_i} = K_{\Gamma_{i_0}} \}.
%$$

Before formulating our results, we need to emphasize that the proofs
{\it assume the validity of Schanuel's conjecture} in transcendental
number theory (cf.\,\S \ref{S:FGeod}), making the results {\it
conditional.}

\vskip2mm

%Set $M_i = \fX_{\Gamma_i}$ and $\cF_i = \cF(M_i).$
%
%
%\vskip1.5mm

\noindent {\bf Theorem 1.} {\it Assume that the subgroups $\Gamma_1$
and $\Gamma_2$ are finitely generated (which is automatically the
case if these subgroups are actually lattices).

\vskip2mm

\noindent {\rm (1)} If $w_1 > w_2$ then $(T_1)$ holds;

\noindent {\rm (2)} \parbox[t]{11.5cm}{If $w_1 = w_2$ but
$K_{\Gamma_1} \not\subset K_{\Gamma_2}$ then again $(T_1)$ holds.}

\vskip2mm

\noindent Thus, unless $w_1 = w_2$ and $K_{\Gamma_1} =
K_{\Gamma_2}$, condition $(T_i)$ holds for at least one $i \in
\{1 , 2\}$.}
%
%
%If there exists $i_0 \in I$ such that $I(i_0) = \{i_0\},$ then
%condition $(T_{i_0})$ holds. Thus, $(T_i),$ hence also $(N_i),$
%holds for some $i \in I$ unless there are two distinct $i_1 , i_2
%\in I$ with $K_{\Gamma_{i_1}} = K_{\Gamma_{i_2}}.$}

\vskip1mm

(We recall that $w_1 = w_2$ implies that either $G_1$ and $G_2$ are
of the same Killing-Cartan type, or one of them is of type $B_n$ and
the other of type $C_n$ for some $n \geqslant 3$.)

\vskip2mm

Much more precise results are available when the groups $\Gamma_1$
and $\Gamma_2$ are arithmetic (cf.\:\cite{PR6}, \S 1, and \S
\ref{S:AG} below regarding the notion of arithmeticity). As follows
from Theorem 1, we only need to consider the case where $w_1 = w_2$
which we will assume. Then it is convenient to divide our results
into three theorems, two of which treat the case where $G_1$ and
$G_2$ are of the same Killing-Cartan type, and the third one the
case where one of the groups is of type $B_n$ and the other of type
$C_n$ for some $n \geqslant 3$ (we note that the combination of
these three cases covers all possible situations where $w_1 = w_2$).
When $G_1$ and $G_2$ are of the same type, we consider separately
the cases where the common type is not one of the following: $A_n,$
$D_{2n+1}\, (n > 1)$ and $E_6$ and where it is one of these types.

\vskip2mm

\noindent {\bf Theorem 2.} {\it Notations as above, assume that
$G_1$ and $G_2$ are of the same Killing-Cartan type which is
different from $A_n,$ $D_{2n+1}$ $(n > 1)$ and $E_6$ and that the
subgroups $\Gamma_1$ and $\Gamma_2$ are arithmetic. Then either
$M_{1} = \fX_{\Gamma_{1}}$ and $M_{2} = \fX_{\Gamma_{2}}$ are
commensurable, hence $\Q \cdot L(M_{1}) = \Q \cdot L(M_{2})$ and
$\cF_{1} = \cF_{2},$ or conditions $(T_i)$ and $(N_i)$ hold for at
least one $i \in \{1 , 2\}$.}

\vskip2mm

(We note that $(T_i)$ and $(N_i)$ may not hold for {\it both} $i =
1$ and $2$; in fact it is possible that $L(M_1) \subset L(M_2),$ cf.\,Example 7.4.)

\vskip2mm

\noindent {\bf Theorem 3.} {\it Again, keep the above notations and
assume that the common Killing-Cartan type of $G_1$ and $G_2$ is one
of the following: $A_n,$ $D_{2n+1} (n > 1)$ or $E_6$ and that the
subgroups $\Gamma_1$ and $\Gamma_2$ are arithmetic. Assume in
addition that $K_{\Gamma_i} \neq \Q$ for at least one $i \in \{1 ,
2\}$. Then either $\Q \cdot L(M_{1}) = \Q \cdot L(M_{2}),$ hence
$\cF_{1} = \cF_{2}$ (although $M_{1}$ and $M_{2}$ may not be
commensurable), or conditions $(T_i)$ and $(N_i)$ hold for at least
one  $i \in \{1 , 2\}$.}

\vskip2mm

These results can be used in various geometric situations. To
illustrate the scope of possible applications, we will now give
explicit statements for real hyperbolic manifolds (similar results are
available for complex and quaternionic hyperbolic manifolds).

\vskip2mm

\noindent {\bf Corollary 1.} {\it
%Let $M_1, \ldots , M_r$ be a
%finite collection of hyperbolic manifolds of dimension $\neq 3$
%having finite volume, i.e.
Let $M_i$ $(i = 1, 2)$ be the quotient of the real hyperbolic space
$\mathbb{H}^{d_i}$ with $d_i \neq 3$ by a torsion-free Zariski-dense
discrete subgroup $\Gamma_i$ of  $G_i(\R)$ where $G_i =
\mathrm{PSO}(d_i , 1).$
%Let
%$$
%d = \max_{i \leqslant r} d_i \ \ \text{and} \ \ I = \{ i \leqslant r
%\: | \: d_i = d \}.
%$$
%Furthermore, pick any $i_0 \in I$ such that the trace field
%$K_{\Gamma_{i_0}}$ is maximal (under inclusion) among $K_{\Gamma_i}$
%for $i \in I,$ and set $I(i_0) = \{i \in I \: | \: K_{\Gamma_i} =
%K_{\Gamma_{i_0}} \}.$

\vskip2mm

\noindent \ (i) \parbox[t]{13cm}{If $d_1 > d_2$ then conditions
$(T_1)$ and $(N_1)$ hold.}

\vskip1mm

\noindent (ii) \parbox[t]{13cm}{If $d_1 = d_2$ but $K_{\Gamma_1}
\not\subset K_{\Gamma_2}$ then again conditions $(T_1)$ and $(N_1)$
hold.}

\vskip2mm

\noindent Thus, unless $d_1 = d_2$ and $K_{\Gamma_1} =
K_{\Gamma_2}$, conditions $(T_i)$ and $(N_i)$ hold for at least one
$i \in \{1 , 2\}$.

\vskip2mm

\noindent Assume now that $d_1 = d_2 =: d$ and the subgroups
$\Gamma_1$ and $\Gamma_2$ are arithmetic.

\vskip2mm

\noindent (iii) \parbox[t]{13cm}{If $d$ is either even or is
congruent to $3(\mathrm{mod}\: 4),$  then either  $M_{1}$ and
$M_{2}$ are commensurable, hence length-commensurable and $\cF_1 =
\cF_2$, or $(T_i)$ and $(N_i)$ hold for at least one $i \in \{1 ,
2\}$.}

\vskip2mm

\noindent \ (iv) \parbox[t]{13cm}{If $d \equiv 1(\mathrm{mod}\: 4)$
and in addition $K_{\Gamma_i} \neq \Q$ for at least one $i \in \{1 ,
2\}$ then either  $M_{1}$ and $M_{2}$ are length-commensurable
(although not necessarily commensurable), or conditions $(T_i)$ and
$(N_i)$ hold for at least one $i \in \{1 , 2\}$.}

}

\vskip2mm

The results of \cite{GR} enable us to consider the situation where
one of the groups is of type $B_n$ and the other is of type $C_n$.
\vskip2mm

\noindent {\bf Theorem 4.} {\it Notations as above, assume that
$G_1$ is of type $B_n$ and $G_2$ is of type $C_n$ for some $n
\geqslant 3$ and the subgroups $\Gamma_1$ and $\Gamma_2$ are
arithmetic. Then either $(T_i)$ and $(N_i)$ hold for at least one $i
\in \{1 , 2\}$, or $$\Q \cdot L(M_2) = \lambda \cdot \Q \cdot L(M_1)
\ \ \text{where} \ \ \lambda = \sqrt{\frac{2n+2}{2n-1}}.$$}

\vskip1mm

The following interesting result holds for all types.

\vskip2mm

\noindent {\bf Theorem 5.} {\it For $i =
1, 2$, let $M_i = \fX_{\Gamma_i}$  be an arithmetically defined locally symmetric space, and assume
that $w_1 = w_2$. If $M_2$ is compact and $M_1$ is not, then
conditions $(T_1)$ and $(N_1)$ hold.}

\vskip2mm

Finally, we have the following statement which shows that the notion
of ``similarity" (or more precisely, ``length-similarity") for
arithmetically defined locally symmetric spaces is redundant.

\vskip2mm

\noindent {\bf Corollary 2.} {\it Let $M_i = \fX_{\Gamma_i}$ for $i
=1, 2$ be arithmetically defined locally symmetric spaces. Assume
that there exists $\lambda \in \R_{> 0}$ such that
$$
\Q \cdot L(M_1) = \lambda \cdot \Q \cdot L(M_2).
$$
Then

\vskip1mm

\noindent (i) \parbox[t]{13cm}{if $G_1$ and $G_2$ are of the same
type which is different from $A_n,$ $D_{2n+1} (n > 1)$ and $E_6$,
then $M_1$ and $M_2$ are commensurable, hence length-commensurable;}

\vskip1mm

\noindent (ii) \parbox[t]{13cm}{if $G_1$ and $G_2$ are of the same
type which is one of the following: $A_n,$ $D_{2n+1} (n > 1)$ or
$E_6$ then, provided that $K_{\Gamma_i} \neq \Q$ for at least one $i
\in \{1 , 2\},$ the spaces $M_1$ and $M_2$ are length-commensurable
(although not necessarily commensurable).}}

\vskip2mm

(See Corollary 7.11 for a more detailed statement.)

\vskip2mm

While the geometric results in \cite{PR6} were derived from an 
analysis of the relationship between Zariski-dense subgroups of
semi-simple algebraic groups called {\it weak commensurability},
the results described above require a more general and technical
version of this notion which we call {\it weak containment.} We
recall that given two semi-simple groups $G_1$ and $G_2$ over a
field $F$ and Zariski-dense subgroups $\Gamma_i \subset G_i(F)$ for
$i = 1, 2,$ two semi-simple elements $\gamma_i \in \Gamma_i$ are
weakly commensurable if there exist maximal $F$-tori $T_i$ of $G_i$
such that $\gamma_i \in T_i(F),$ and for some characters $\chi_i$ of
$T_i$ (defined over an algebraic closure $\overline{F}$ of $F$), we
have
$$
\chi_1(\gamma_1) = \chi_2(\gamma_2) \neq 1.
$$
Furthermore, $\Gamma_1$ and $\Gamma_2$ are weakly commensurable if
every semi-simple element $\gamma_1 \in \Gamma_1$ of infinite order
is weakly commensurable to some semi-simple element $\gamma_2 \in
\Gamma_2$ of infinite order, and vice versa.
%Let now $G_1, \ldots , G_r$ be an arbitrary finite collection of
%semi-simple algebraic groups defined over a field $F$ of
%characteristic zero, and suppose that for each $i \leqslant r$ we
%are given a finitely generated Zariski-dense subgroup $\Gamma_i
%\subset G_i(F).$
The following definition provides a generalization
of the notion of weak commensurability which is adequate for our
purposes.

\vskip2mm

\noindent {\bf Definition 1.} Notations as above, semi-simple
elements $\gamma^{(1)}_1, \ldots , \gamma^{(1)}_{m_1} \in \Gamma_1$
are {\it weakly contained} in $\Gamma_2$ if there are semi-simple
elements $\gamma^{(2)}_1, \ldots , \gamma^{(2)}_{m_2} \in \Gamma_2$
such that
$$
\chi^{(1)}_1(\gamma^{(1)}_1) \cdots
\chi^{(1)}_{m_1}(\gamma^{(1)}_{m_1}) \ = \
\chi^{(2)}_1(\gamma^{(2)}_1) \cdots
\chi^{(2)}_{m_2}(\gamma^{(2)}_{m_2}) \ \neq \  1.
$$
for some maximal $F$-tori $T^{(j)}_k$ of $G_j$ containing
$\gamma^{(j)}_k$ and some characters $\chi^{(j)}_k$ of  
$T^{(j)}_k$ for $j \in \{1 , 2\}$ and $k \leqslant m_j.$

\vskip1mm

(It is easy to see that this property is independent of the choice
of the maximal tori containing the elements in question.)

\vskip2mm

We also need the following.

\vskip.7mm

\noindent {\bf Definition 2.} (a) Let $T_1, \ldots , T_m$ be a
finite collection of algebraic tori defined over a field $K,$ and
for each $i \leqslant m,$ let $\gamma_i \in T_i(K).$ The elements
$\gamma_1, \ldots , \gamma_m$ are called {\it multiplicatively
independent} if a relation of the form
$$
\chi_1(\gamma_1) \cdots \chi_m(\gamma_m) = 1,
$$
where $\chi_j \in X(T_j),$ implies that
$$
\chi_1(\gamma_1) = \cdots = \chi_m(\gamma_m) = 1.
$$

\vskip1mm

(b) Let $G$ be a semi-simple algebraic $F$-group. Semi-simple
elements $\gamma_1, \ldots , \gamma_m \in G(F)$ are called {\it
multiplicatively independent} if for some (equivalently, any) choice
of maximal $F$-tori $T_i$ of $G$ such that $\gamma_i \in T_i(F)$ for
$i \leqslant m,$ these elements are multiplicatively independent in the sense of 
part (a).

\vskip2mm

We are now in a position to give a definition that plays the central
role in the paper.

\vskip1mm

\noindent {\bf Definition 3.} We say that $\Gamma_1$ and $\Gamma_2$
as above satisfy {\it property $(C_i)$}, where $i=$ 1 or 2, if
for any $m \geqslant 1$ there exist  semi-simple elements $\gamma_1,
\ldots , \gamma_m \in \Gamma_i$ of infinite order  that are
multiplicatively independent and are {\it not} weakly contained in
$\Gamma_{3-i}.$

\vskip2mm

Our main effort is focused on developing a series of conditions that
guarantee the fact that $\Gamma_1$ and $\Gamma_2$ satisfies $(C_i)$
for at least one $i \in \{1 , 2\}$ (in fact, typically we are
able to pin down the $i$). Before formulating a sample result, we
would like to note that the notion of the trace subfield (field
of definition) $K_{\Gamma_i} \subset F$
%and of the subsets $I \subset \{1, \ldots , r\}$ and $I(i_0)$ (for
%$i_0 \in I$ such that $K_{\Gamma_{i_0}}$ is maximal under inclusion
%among $K_{\Gamma_i}$ for $i \in I$) introduced above
makes sense for {\it any} field $F$ and not only for $F = \R.$

\vskip1mm

\noindent {\bf Theorem \ref{T:F1}.} {\it Assume that $\Gamma_1$ and
$\Gamma_2$ are finitely generated.

\vskip2mm

\noindent \ (i) If $w_1 > w_2$ then condition $(C_1)$ holds;

\vskip1mm

\noindent (ii) If $w_1 = w_2$ but $K_{\Gamma_1} \not\subset
K_{\Gamma_2}$ then again $(C_1)$ holds.

\vskip2mm

\noindent Thus, unless $w_1 = w_2$ and $K_{\Gamma_1} =
K_{\Gamma_2}$, condition $(C_i)$ holds for at least one $i \in
\{1 , 2\}$.}

\vskip2mm

We prove much more precise results in the case where the $\Gamma_i$
are arithmetic. The statements however are somewhat  technical,
and we refer the reader to \S  \ref{S:AG} for their precise
formulations.

\vskip1.5mm

The reader may have already noticed similarities in the statements
of Theorem 1 and  Theorem \ref{T:F1}. The same similarities exist
also between the ``geometric'' Theorems 2-4 and the corresponding
``algebraic'' results in \S \ref{S:AG}. The precise connection
between ``algebra'' and ``geometry'' is given by  Proposition
\ref{P:LG1} which has the following consequence (Corollary \ref{C:LG1}):

\vskip2mm

\noindent {\it If $\fX_{\Gamma_1}$ and $\fX_{\Gamma_2}$ are locally
symmetric spaces as above with finitely generated fundamental groups
$\Gamma_1$ and $\Gamma_2$, then the fact that these groups satisfy
property $(C_i)$ for some $i \in \{1 , 2\}$ implies that the locally
symmetric spaces satisfy conditions $(T_i)$ and $(N_i)$ for the same
$i$.}

\vskip2mm

It should be noted that the
proof of Proposition \ref{P:LG1} assumes the truth of Schanuel's
conjecture, and in fact it is the only place in the paper where the
latter is used. In conjunction with the results of \S \ref{S:AG},
this facts provides a series of rather restrictive conditions on the
arithmetic groups $\Gamma_1$ and  $\Gamma_2$ in case $(T_i)$ fails
for both $i = 1$ and $2$. Eventually, these condition enable
us to prove that if $G_1$ and $G_2$ are of the same type which is
different from $A_n,$ $D_{2n+1} (n > 1)$ or $E_6$ then $G_1 \simeq
G_2$ over $K := K_{\Gamma_{1}} = K_{\Gamma_{2}}$ and hence the
subgroups $\Gamma_{1}$ and $\Gamma_{2}$ are commensurable in the
appropriate sense (viz., up to an isomorphism between $G_1$ and
$G_2$), yielding the commensurability of the locally symmetric
spaces $\fX_{\Gamma_{1}}$ and $\fX_{\Gamma_{2}}$ (cf.\:Theorem
2). If $G_1$ and $G_2$ are of the same type which is one of the
following $A_n,$ $D_{2n+1} (n > 1)$ or $E_6$, then $G_1$ and $G_2$
may not be $K$-isomorphic, but using the results from \cite{PR6}, \S
9, and \cite{PR7}, we show that (under some minor restrictions)
these groups necessarily have equivalent systems of maximal $K$-tori
(see \S \ref{S:ADE} for the precise definition) making the
corresponding locally symmetric spaces $\fX_{\Gamma_1}$ and
$\fX_{\Gamma_2}$ length-commensurable, and thereby proving
Theorem 3. To prove Theorem 4, we use the results of \cite{GR} that
describe when two absolutely almost simple $K$-groups, one of
type $B_n$ and the other of type $C_n$ $(n \geqslant 3)$, have the
same isomorphism classes of maximal $K$-tori.

\vskip2mm

\noindent {\bf Notations.} For a field $K$, $K_{\rm{sep}}$ will denote a separable closure.  
Given a (finitely generated) field $K$ of
characteristic zero, we let $V^K$ denote the set of (equivalence
classes) of nontrivial valuations $v$ of $K$ with locally compact
completion $K_v.$ If $v \in V^K$ is nonarchimedean, then $K_v$ is a
finite extension of the $p$-adic field $\Q_p$ for some $p$; in the
sequel this prime $p$ will be denoted by $p_v$.  
Given a subset $V$ of  $V^K$ consisting of nonarchimedean
valuations, we set $\Pi_V = \{ p_v \: \vert \: v \in V \}.$

\vskip2mm

\noindent {\bf Acknowledgements.} We thank Skip Garibaldi for
proving in \cite{Gar} Theorem \ref{T:D2n-G} which, in particular,
enabled us to include type $D_4$ in Theorem \ref{T:ArG2}. We also
thank Sai-Kee Yeung for a discussion of his paper \cite{SKY}. Both
authors were partially supported by the NSF (grants DMS-1001748 and
DMS-0965758) and the Humboldt Foundation. During the preparation of 
this paper, the second-named author was
visiting the Mathematics Department of the University of Michigan as
a Gehring Professor; the hospitality and generous support of this
institution are thankfully acknowledged.

\section{Weak containment}\label{S:WC}

The goal of this section is to derive several consequences of the
relation of weak containment (see Definition 1 of the Introduction)
that will be needed later. We begin with some definitions and
results for algebraic tori. Given a torus $T$ defined over a field
$K,$ we let $K_T$ denote its (minimal) splitting field over $K$
(contained in a fixed algebraic closure $\overline{K}$ of $K$). The
following definition goes back to \cite{PR3}.

\vskip2mm

\noindent {\bf Definition 4.} A $K$-torus $T$ is called $K$-{\it
irreducible} (or, {\it irreducible over} $K$) if it does not contain any proper $K$-subtori.

\vskip1mm

Recall that $T$ is $K$-irreducible if and only if $X(T) \otimes_{\Z} \Q$
is an irreducible $\Ga(K_T/K)$-module, cf.\,\cite{PR3}, Proposition
1. Now, let $G$ be an absolutely almost simple algebraic $K$-group.
For a maximal torus $T$ of $G,$ we let $\Phi = \Phi(G , T)$ denote
the corresponding root system, and let $\mathrm{Aut}(\Phi)$ be the
automorphism group of $\Phi$. As usual, the Weyl group $W(\Phi) \subset
\mathrm{Aut}(\Phi)$ will be identified with the Weyl group $W(G ,
T)$ of $G$ relative to $T.$ If $T$ is defined over a field extension
$L$ of $K,$ and $L_T$ is the splitting field of $T$ over $L$ in an
algebraic closure of the latter, then there is a natural
injective homomorphism
$$
\theta_T \colon \Ga(L_T/L) \to \mathrm{Aut}(\Phi).
$$
Since $W(\Phi)$ acts absolutely irreducibly on $X(T) \otimes_{\Z} \Q$,
we conclude that a maximal $L$-torus $T$ of $G$ such
that $\theta_T(\Ga(L_T/L)) \supset W(G , T)$ is automatically
$L$-irreducible. (We also recall for the convenience of further
reference that if $G$ is of inner type over $L$ then
$\theta_T(\Ga(L_T/L)) \subset W(G , T),$ cf.\,\cite{PR6}, Lemma 4.1.)

\vskip2mm

\noindent {\bf Definition 5.} Let $T_1, \ldots , T_m$ be $K$-tori.
We say that these tori are {\it independent} (over $K$) if their
splitting fields $K_{T_1}, \ldots, K_{T_m}$ are linearly disjoint
over $K,$ i.e. the natural map
$$
K_{T_1} \otimes_K \cdots \otimes_K K_{T_m} \longrightarrow K_{T_1}
\cdots K_{T_m}
$$
is an isomorphism.

\vskip2mm

\begin{lemma}\label{L:P1}
Let $T_1, \ldots , T_m$ be $K$-tori, and for $i\leqslant m$, let
$\gamma_i \in T_i(K)$ be an element of infinite order. Assume that
$T_1, \ldots , T_m$ are independent, irreducible and nonsplit over
some extension $L$ of $K.$ Then the elements $\gamma_1, \ldots ,
\gamma_m$ are multiplicatively independent (see Definition 2 in
\S\ref{S:I}).
\end{lemma}
\begin{proof}
Suppose there exist characters $\chi_i \in X(T_i)$ such that
$$
\chi_1(\gamma_1) \cdots \chi_m(\gamma_m) = 1.
$$
Since $\chi_i(\gamma_i) \in L^{\times}_{T_i}$ and the tori $T_1, \ldots
, T_m$ are independent over $L,$ it follows that actually
$\chi_i(\gamma_i) \in L^{\times}$ for all $i \leqslant m.$ Then for
any $\sigma \in \Ga(L_{T_i}/L)$ we have
\begin{equation}\label{E:P140}
(\sigma \chi_i - \chi_i)(\gamma_i) = 1.
\end{equation}
Being a $L$-rational element of infinite order in an $L$-irreducible
torus $T_i,$ the element $\gamma_i$ generates a Zariski-dense
subgroup of the latter, so (\ref{E:P140}) implies that $\sigma
\chi_i = \chi_i.$ But $X(T_i)$ does not have nonzero
$\Ga(L_{T_i}/L)$-fixed elements. Thus, $\chi_i = 0$ and
$\chi_i(\gamma_i) = 1.$
\end{proof}

\vskip2mm

The following lemma is crucial for unscrambling relations of weak
containment.
\begin{lemma}\label{L:unscramble}
Let $T^{(1)}_1, \ldots , T^{(1)}_{m_1}$ and $T^{(2)}_1, \ldots ,
T^{(2)}_{m_2}$ be two finite families of algebraic $K$-tori, and
suppose we are given a relation of the form
\begin{equation}\label{E:FF160}
\chi_1^{(1)}(\gamma_1^{(1)}) \cdots
\chi_{m_1}^{(1)}(\gamma_{m_1}^{(1)}) = \chi_1^{(2)}(\gamma_1^{(2)})
\cdots \chi_{m_2}^{(2)}(\gamma_{m_2}^{(2)}),
\end{equation}
where $\gamma_i^{(s)} \in T_i^{(s)}(K)$ and $\chi_i^{(s)} \in
X(T_i^{(s)}).$ Assume that  $T_1^{(1)}, \ldots , T_{m_1}^{(1)}$ are
independent, irreducible and nonsplit over $K.$ Then for every $i
\leqslant m_1$ such that the corresponding character $\chi_i^{(1)}$
in {\rm (\ref{E:FF160})} is nontrivial, there exists an integer $d_i
> 0$ with the following property:

\vskip2mm

\ \ \ \ \ \ \parbox[t]{12.5cm}{For any $\delta_i^{(1)} \in
d_iX(T_i^{(1)})$ there are characters $\delta_j^{(2)} \in
X(T_j^{(2)})$ for $j \leqslant {m_2}$ for which
\begin{equation}\label{E:unscramble}
\delta_i^{(1)}(\gamma^{(1)}_i) = \delta_1^{(2)}(\gamma^{(2)}_1)
\cdots \delta_{m_2}^{(2)}(\gamma^{(2)}_{m_2}).
\end{equation}}

\vskip2mm

\noindent In addition, if $\gamma_i^{(1)}$ has infinite order  and
$\delta_i^{(1)} \neq 0$ then the common value in {\rm
(\ref{E:unscramble})} is $\neq 1.$
\end{lemma}
\begin{proof}
As the tori $T^{(1)}_1, \ldots , T^{(1)}_{m_1}$ are independent over
$K,$  we have the natural isomorphism
\begin{equation}\label{E:FF162}
\Ga(K_{T_1^{(1)}} \cdots K_{T_{m_1}^{(1)}}/K) \simeq
\Ga(K_{T_1^{(1)}}/K) \times \cdots \times \Ga(K_{T_{m_1}^{(1)}}/K).
\end{equation}
Since $T_i^{(1)}$ is $K$-irreducible and nonsplit, $X(T_i^{(1)})$ does not
contain any nontrivial $\Ga(K_{T_i^{(1)}}/K)$-fixed elements. So, it follows from
(\ref{E:FF162}) that there exists $\sigma \in \Ga(\overline{K}/K)$
such that $\sigma \chi^{(1)}_i \neq \chi^{(1)}_i$ but $\sigma
\chi^{(1)}_j = \chi^{(1)}_j$ for $j \neq i.$ Applying $\sigma - 1$
to (\ref{E:FF160}), we obtain
\begin{equation}\label{E:FF163}
\mu_i^{(1)}(\gamma_i^{(1)}) = \mu_1^{(2)}(\gamma_1^{(2)}) \cdots
\mu_{m_2}^{(2)}(\gamma_{m_2}^{(2)}),
\end{equation}
where $\mu_j^{(s)} = \sigma \chi_j^{(s)} - \chi_j^{(s)},$ noting
that $\mu_i^{(1)} \neq 0.$ Again, since $T_i^{(1)}$ is
$K$-irreducible and nonsplit, the $\Ga(\overline{K}/K)$-submodule of
$X(T^{(1)}_i)$ generated by $\mu^{(1)}_i$ has finite index, hence it  
contains $d_iX(T_i^{(1)})$ for some integer $d_i > 0.$ Then any
$\delta_i^{(1)} \in d_i X(T_i^{(1)})$ can be written as
$$
\delta_i^{(1)} = \sum n_{\sigma} \sigma(\mu_i^{(1)})\  \ \text{for\
some} \ \ \sigma \in \Ga(\overline{K}/K) \ \ \text{and}\  \
n_{\sigma} \in \Z.
$$
So, using (\ref{E:FF163}) we obtain that
$$
\delta_i^{(1)}(\gamma_i^{(1)}) = \delta_1^{(2)}(\gamma_1^{(2)})
\cdots \delta_{m_2}^{(2)}(\gamma_{m_2}^{(2)})
$$
with $\delta_j^{(2)} = \sum n_{\sigma} \sigma(\mu_j^{(2)})$ for $j
\leqslant {m_2}.$ Finally, if $\gamma_i^{(1)}$ is of infinite
order then it generates a Zariski-dense subgroup of the
$K$-irreducible torus $T_i^{(1)},$ and therefore
$\delta_i^{(1)}(\gamma_i^{(1)}) \neq 1$ for any nonzero
$\delta_i^{(1)} \in X(T_i^{(1)}).$
\end{proof}

\vskip3mm

The following theorem is an adaptation of a part of the Isogeny
Theorem (Theorem 4.2) of \cite{PR6} suitable for our purposes.
\begin{thm}\label{T:10}
Let $T^{(1)}_1, \ldots , T^{(1)}_{m_1}$ and $T^{(2)}_1, \ldots ,
T^{(2)}_{m_2}$ be two finite families of algebraic $K$-tori, and
suppose we are given a relation of the form
\begin{equation}\label{E:FF380}
\chi^{(1)}_1(\gamma^{(1)}_1) \cdots
\chi^{(1)}_{m_1}(\gamma^{(1)}_{m_1}) = \chi^{(2)}_1(\gamma^{(2)}_1)
\cdots \chi^{(2)}_{m_2}(\gamma^{(2)}_{m_2}),
\end{equation}
where $\gamma^{(s)}_i \in T^{(s)}_i(K)$ and $\chi^{(s)}_i \in
X(T^{(s)}_i).$ Assume that the tori $T^{(1)}_1, \ldots ,
T^{(1)}_{m_1}$ are independent, irreducible and nonsplit over $K,$
and that the elements $\gamma^{(1)}_1, \ldots , \gamma^{(1)}_{m_1}$
all have infinite order. Then for each $i \leqslant {m_1}$ such that
the corresponding character $\chi_i^{(1)}$ in {\rm (\ref{E:FF380})}
is nontrivial, there exists a surjective $K$-homomorphism $T_j^{(2)}
\to T_i^{(1)}$ for some $j \leqslant m_2$, hence, in particular, 
$K_{T_i^{(1)}} \subset K_{T_j^{(2)}}.$ Moreover, if all the tori are
of the same dimension, the above homomorphism is an isogeny and
$K_{T_i^{(1)}} = K_{T_j^{(2)}}.$
\end{thm}
\begin{proof}
Fix $i \leqslant {m_1}$ such that $\chi^{(1)}_i \neq 0.$ Applying
Lemma \ref{L:unscramble}, we see that there is a relation of the
form
$$
\delta_i^{(1)}(\gamma_i^{(1)}) = \delta_1^{(2)}(\gamma_1^{(2)})
\cdots \delta_{m_2}^{(2)}(\gamma_{m_2}^{(2)})
$$
with $\delta_i^{(1)} \in X(T_i^{(1)}),$ $\delta_i^{(1)} \neq 0,$ and
$\delta_j^{(2)} \in X(T_j^{(2)})$ for $j \leqslant m_2$.  To
simplify our notation, we set $$T^{(1)} = T_i^{(1)}, \ \  \
\gamma^{(1)} = \gamma_i^{(1)}, \ \  \ \delta^{(1)} =
\delta_i^{(1)}$$ and
$$T^{(2)} = T_1^{(2)} \times \cdots \times T_{m_2}^{(2)}, \  \ \
\gamma^{(2)} = (\gamma_1^{(2)}, \ldots , \gamma_{m_2}^{(2)}), \ \ \
\delta^{(2)} = (\delta_1^{(2)}, \ldots , \delta_{m_2}^{(2)}).$$ Then
$$
\delta^{(1)}(\gamma^{(1)}) = \delta^{(2)}(\gamma^{(2)}) =: \lambda.
$$
First, we will show that the Galois conjugates $\sigma(\lambda)$ for
$\sigma \in \Ga(K_{T^{(1)}}/K)$ generate $K_{T^{(1)}}$ over $K.$
Indeed, suppose $\tau \in \Ga(K_{T^{(1)}}/K)$ fixes all the
$\sigma(\lambda)$'s. Then for any $\sigma \in \Ga(K_{T^{(1)}}/K)$ we
have
$$
(\tau\sigma(\delta^{(1)}))(\gamma^{(1)}) = \tau(\sigma(\lambda)) =
\sigma(\lambda) = (\sigma(\delta^{(1)}))(\gamma^{(1)}).
$$
Since $T^{(1)}$ is  $K$-irreducible, the element $\gamma^{(1)} \in
T^{(1)}(K),$ being of infinite order, generates a Zariski-dense
subgroup of $T^{(1)}.$ Hence, we conclude that
$\tau(\sigma(\delta^{(1)})) = \sigma(\delta^{(1)})$ for all $\sigma
\in \Ga(K_{T^{(1)}}/K).$ But the elements $\sigma(\delta^{(1)})$
span $X(T^{(1)}) \otimes_{\Z} \Q$ as $\Q$-vector space, so $\tau =
\mathrm{id},$ and our claim follows.

Now, since all the elements $\sigma(\lambda)$ for $\sigma \in
\Ga(K_{T^{(1)}}/K)$ belong to $K_{T^{(2)}},$ we obtain the inclusion
$K_{T^{(1)}} \subset K_{T^{(2)}}$. So the restriction map

$$
\cG := \Ga(K_{T^{(2)}}/K) \longrightarrow \Ga(K_{T^{(1)}}/K)
$$
is a surjective homomorphism. In the rest of the proof, we will view $X(T^{(1)})$ as a $\cG$-module
via this homomorphism. Define $\nu_i \colon \Q[\cG] \to X(T^{(i)})
\otimes_{\Z} \Q$ by
$$
\sum_{\sigma \in \cG} n_{\sigma} \sigma \mapsto \sum_{\sigma \in
\cG} n_{\sigma} \sigma(\delta^{(i)}).
$$
We observe that  $\delta^{(1)}(\gamma^{(1)}) =
\delta^{(2)}(\gamma^{(2)})$ implies that for any $a = \sum
n_{\sigma}\sigma \in \Z[\cG],$ we have
\begin{equation}\label{E:FF333*}
\nu_2(a)(\gamma^{(2)}) = \prod
\sigma(\delta^{(2)}(\gamma^{(2)}))^{n_{\sigma}} = \prod
\sigma(\delta^{(1)}(\gamma^{(1)}))^{n_{\sigma}} =
\nu_1(a)(\gamma^{(1)}).
\end{equation}
It is now easy to show that
\begin{equation}\label{E:FF334*}
\mathrm{Ker}\: \nu_2 \subset \mathrm{Ker}\: \nu_1.
\end{equation}
Indeed, let $a \in \Z[\cG]$ be such that $\nu_2(a) = 0.$ Then it
follows from (\ref{E:FF333*}) that
$$
\nu_2(a)(\gamma^{(2)}) = 1 = \nu_1(a)(\gamma^{(1)}).
$$
As $\gamma^{(1)}$ generates a Zariski-dense subgroup of $T^{(1)},$
we conclude that $\nu_1(a) = 0,$ and (\ref{E:FF334*}) follows.

Combining (\ref{E:FF334*}) with the fact that $\delta^{(1)}$
generates $X(T^{(1)}) \otimes_{\Z} \Q$ as a $\Q[\cG]$-module, we get a
surjective homomorphism
$$
\alpha \colon \mathrm{Im}\: \nu_2 \longrightarrow \mathrm{Im}\:
\nu_1 = X(T^{(1)}) \otimes_{\Z} \Q.
$$
of $\Q[\cG]$-modules. Because of semi-simplicity of $\Q[\cG],$ there
exists an injective $\Z[\cG]$-module homomorphism $X(T^{(1)}) \to
X(T^{(2)}),$ hence a surjective $K$-homomorphism $\theta \colon
T^{(2)} \to T^{(1)}.$ Pick $j \leqslant {m_2}$ so that the
restriction $\theta \vert_{T^{(2)}_j}$ is nontrivial. As $T^{(1)}$
is $K$-irreducible, we conclude that the resulting homomorphism
$T_j^{(2)} \to T^{(1)} = T_i^{(1)}$ is surjective, hence the
inclusion $K_{T_i^{(1)}} \subset K_{T_j^{(2)}}.$ If $\dim T_j^{(2)}
= \dim T_i^{(1)}$, then the above homomorphism is an isogeny
implying that in fact $K_{T_i^{(1)}} = K_{T_j^{(2)}}.$
\end{proof}

\section{Existence of independent irreducible tori}\label{S:Ex}

In order to apply Theorem \ref{T:10} in our analysis of the weak
containment relation, we need to provide an adequate supply of
regular semi-simple elements in a given finitely generated
Zariski-dense subgroup whose centralizers yield arbitrarily large
families of independent irreducible tori. Such elements are
constructed in this section using a suitable generalization, along
the lines indicated in \cite{PR5},  of the result established in
\cite{PR4} (see also \cite{PR6}, \S 3) guaranteeing the existence,
in any Zariski-dense subgroup, of elements whose centralizers are
irreducible tori.

Let $G$ be a connected semi-simple algebraic group defined over a
field $K,$ and let $T$ be a maximal torus of $G$ defined over a
field extension $L$ of $K$. We will systematically use the
notations introduced after Definition 4 in \S\ref{S:WC},
particularly the natural homomorphism $\theta_T \colon \Ga(L_T/L)
\to \mathrm{Aut}(\Phi(G , T)).$ For the convenience of reference, we
now quote Theorem 3.1 of \cite{PR6}.
\begin{thm}\label{T:Ex0}
Let $G$ be a connected absolutely almost simple algebraic group
defined over a finitely generated field $K$ of characteristic zero,
and $L$ be a finitely generated field containing $K.$ Let $r$ be the
number of nontrivial conjugacy classes in the (absolute) Weyl group of $G,$ and
suppose we are given $r$ inequivalent nontrivial discrete valuations
$v_1, \ldots , v_r$ of $K$ such that the completion $K_{v_i}$ is
locally compact and contains $L,$ and $G$ splits over $K_{v_i},$ for
each $i \leqslant r.$ Then there exist maximal $K_{v_i}$-tori
$T(v_i)$ of $G,$ one for each $i\leqslant r,$ with the property that
for any maximal $K$-torus $T$ of $G$ which is conjugate to $T(v_i)$
by an element of $G(K_{v_i})$ for all $i\leqslant r,$ we have
\begin{equation}\label{E:Ex25}
\theta_T(\Ga(L_T/L)) \supset W(G , T).
\end{equation}
\end{thm}

\vskip3mm

The following corollary (see Corollary 3.2 in \cite{PR6}) is derived
from Theorem~\ref{T:Ex0} using weak approximation property of the variety of
maximal tori of $G.$
\begin{cor}\label{C:Ex1}
Let $G,$ $K$ and $L$ be as in Theorem \ref{T:Ex0}, and let $V$ be a
finite set of inequivalent nontrivial  rank $1$ valuations of $K$.  
Suppose that for each $v \in V$ we are
given a maximal $K_v$-torus $T(v)$ of $G.$ Then there exists a
maximal $K$-torus $T$ of $G$ for which (\ref{E:Ex25}) holds and
which is conjugate to $T(v)$ by an element of $G(K_v),$ for all $v
\in V.$
\end{cor}
(In Corollary 3.2 of \cite{PR6} it was assumed that for each $v\in
V$, the completion $K_v$ is locally compact. But as the Implicit Function Theorem holds over $K_v$ 
for any rank 1 valuation $v$ of $K$, the proof of Corollary 3.2 in \cite{PR6} can be modified to prove the above more 
general result.)

\vskip3mm

%{\bf Gopal: I inserted ``discrete'' in the above corollary as we
%need to use the Implicit Function Theorem.}

\vskip3mm

We will now strengthen  the above corollary to obtain 
arbitrarily large families of irreducible independent tori.
\begin{thm}\label{T:Ex1}
Let $G$ be a connected absolutely almost simple algebraic group
defined over a finitely generated field $K$ of characteristic zero,
and $L$ be any finitely generated field extension of $K$ over which
$G$ is of inner type. Furthermore, let $V$ be a finite set of
inequivalent nontrivial rank $1$ valuations of $K$ such that any $v
\in V$ is either discrete or the corresponding completion $K_v$ is
locally compact. Fix $m \geqslant 1,$ and suppose that for each $v
\in V$ we are given $m$ maximal $K_{v}$-tori $T_1(v), \ldots ,
T_m(v)$ of $G.$ Then there exist maximal $K$-tori $T_1, \ldots ,
T_m$ of $G$ such that

\vskip2mm

\noindent \ (i) {for each $j\leqslant m,$ the torus $T_j$ satisfies
\vskip1mm
\begin{equation}
\theta_{T_j}(\Ga(L_{T_j}/L)) \supset W(G , T_j),
\end{equation}
\vskip1mm in particular, $T_j$ is $L$-irreducible; \vskip2mm
\noindent \ (ii) $T_j$ is conjugate to $T_j(v)$ by an element of
$G(K_{v})$ for all $v \in V;$}

\vskip2mm

\noindent (iii) the tori $T_1, \ldots , T_m$ are independent over
$L.$
\end{thm}
\begin{proof}
We will induct on $m.$ If $m = 1$, then the existence of a maximal
$K$-torus $T = T_1$ satisfying $(i)$ and $(ii)$ is established in
Corollary \ref{C:Ex1}, while condition $(iii)$ is vacuous in this
case. Now, let $m > 1$ and assume that the maximal tori $T_1, \ldots
, T_{m-1}$ satisfying conditions $(i),$ $(ii),$ and independent over
$L,$ have already been found. Let $L'$ denote the compositum
of the fields $L_{T_1}, \ldots , L_{T_{m-1}}.$ Applying Corollary
\ref{C:Ex1} with $L'$ in place of $L,$ we find a maximal $K$-torus
$T_m$ which is conjugate to $T_m(v)$ by an element of $G(K_{v})$ for
all $v \in V$ and satisfies
\begin{equation}\label{E:P130}
\theta_{T_m}(\Ga(L'_{T_m}/L')) \supset W(G , T_m).
\end{equation}
Then $T_m$ obviously satisfies conditions $(i)$ and $(ii).$ To see
that $T_1, \ldots , T_m$ satisfy condition $(iii),$ we observe that
as the group $G$ is of inner type  over $L$, according to
\cite{PR6}, Lemma 4.1, we have
$$
\theta_{T_j}(\Ga(L_{T_j}/L)) = W(G , T_j) \ \ \text{for all} \ \ j\leqslant m.
$$
Since $L' = L_{T_1} \cdots L_{T_{m-1}},$ it follows from
(\ref{E:P130}) that
$$
[L_{T_1} \cdots L_{T_m} : L_{T_1} \cdots L_{T_{m-1}}] = \vert W(G ,
T_m) \vert.
$$
By induction hypothesis, $T_1, \ldots , T_{m-1}$ are independent
over $L,$ hence
$$
[L_{T_1} \cdots L_{T_{m-1}} : L] = \prod_{j = 1}^{m-1} [L_{T_j} : L]
= \prod_{j = 1}^{m-1} \vert W(G , T_j) \vert.
$$
Thus,
$$
[L_{T_1} \cdots L_{T_m} : L] = \prod_{j = 1}^m \vert W(G , T_j)
\vert = \prod_{j = 1}^m [L_{T_j} : L],
$$
and therefore $T_1, \ldots , T_m$ are independent over $L.$
\end{proof}

\vskip2mm

Next, we will establish a variant of Theorem \ref{T:Ex1} which
asserts the existence of regular semi-simple elements in a given
Zariski-dense subgroup whose centralizers possess properties $(i),$
$(ii)$ and $(iii)$ of the preceding theorem.
\begin{thm}\label{T:Ex2}
Let $G,$ $K$ and $L$ be as in Theorem \ref{T:Ex1} and $V$ be a finite set of inequivalent 
nontrivial discrete valuations of $K$ such that for every
$v\in V$, the completion $K_v$ of $K$ is locally compact.
Again, fix $m \geqslant 1,$ and suppose that for each $v \in V$ we
are given $m$ maximal $K_{v}$-tori $T_1(v), \ldots , T_m(v)$ of $G.$
Let $\Gamma \subset G(K)$ be a finitely generated Zariski-dense
subgroup such that the closure of the image of the diagonal map
$$
\Gamma \hookrightarrow \prod_{v \in V} G(K_{v})
$$
is open. Then there exist regular semi-simple elements $\gamma_1,
\ldots , \gamma_m \in \Gamma$ of infinite order such that the
maximal $K$-tori $T_j = Z_G(\gamma_j)^{\circ}$ for $j\leqslant m,$ satisfy \vskip2mm

\noindent \ (i) {for each $j \leqslant m$ we have \vskip1mm
\begin{equation}
\theta_{T_j}(\Ga(L_{T_j}/L)) \supset W(G , T_j)
\end{equation}
\vskip1mm \noindent (in particular, $T_j$ is $L$-irreducible, hence
$\gamma_j$ generates a Zariski-dense subgroup of $T_j$); \vskip2mm

\noindent (ii) $T_j$ is conjugate to $T_j(v)$ by an element of
$G(K_{v})$ for all $v \in V;$}

\vskip2mm

\noindent (iii) the tori $T_1, \ldots , T_m$ are independent over
$L.$
\end{thm}
\begin{proof}
We begin with the following lemma.
\begin{lemma}\label{L:Ex1}
Let $\mcG$ be a connected absolutely almost simple algebraic group
over a field $\mcK$ of characteristic zero, $\Gamma$ be a Zariski-dense
subgroup of $\mcG(\mcK)$. Furthermore,
let $\mathscr{V}$ be a finite set of nontrivial discrete valuations
such that for each $v \in \mathscr{V},$ the completion
$\mcK_v$ is locally compact, hence a finite extension of
$\Q_{p_v}$ for some prime $p_v.$ Assume that the closure of the
image of the diagonal map
$$
\Gamma \longrightarrow \prod_{v \in \mathscr{V}}
\mcG(\mcK_v) =: \mcG_{\mathscr{V}}
$$
is open in $\mcG_{\mathscr{V}}.$ Let now $\mathscr{W}$ be another
finite set of nontrivial discrete valuations of $\mcK$ such that for
each $w \in \mathscr{W}$ we have $\mcK_w = \Q_{p_w}$ for the
corresponding prime $p_w$ and that $\Gamma$ is a nondiscrete
subgroup of $\mcG(\mcK_w)$ (which is automatically the case if
$\Gamma$ is relatively compact in $\mcG(\mcK_w)$). If the primes
$p_w$ for $w \in \mathscr{W}$ are pairwise distinct and none of them
is contained in $\Pi_{\mathscr{V}} = \{p_v \vert v \in \mathscr{V}
\}$, then the closure $\overline{\Gamma}^{(\mathscr{V} \cup
\mathscr{W})}$ of the image of the diagonal map
$$
\Gamma \longrightarrow \prod_{v \in \mathscr{V} \cup \mathscr{W}}
\mcG(\mcK_v) =: \mcG_{\mathscr{V} \cup
\mathscr{W}}
$$
is also open.
\end{lemma}
\begin{proof}
Replacing $\Gamma$ with $\Gamma \cap \Omega$ for a suitable open
subgroup $\Omega$ of $\mcG_{\mathscr{V}},$ we can assume that
the closure $\overline{\Gamma}^{(\mathscr{V})}$ of $\Gamma$ in
$\mcG_{\mathscr{V}}$ is of the form
$$
\overline{\Gamma}^{(\mathscr{V})} = \prod_{v \in \mathscr{V}}
\mathcal{U}_v
$$
where $\mathcal{U}_v$ is an open pro-$p_v$ subgroup of
$\mcG(\mcK_v).$ (We notice that for any open subgroup $\Omega
\subset \mcG_{\mathscr{V}},$ the intersection $\Gamma \cap \Omega$
is still Zariski-dense in $G$ as its closure in $\mcG(\mcK_v)$
contains an open subgroup, for every $v \in \mathscr{V}.$) A standard
argument (cf.\,\cite{PR4}, Lemma 2) shows that the closure
$\overline{\Gamma}^{(w)}$ of $\Gamma$ in $\mcG(\mcK_w)$ is open for
any $w \in \mathscr{W}.$ Moreover, as above, we can assume, after replacing $\Gamma$ with
a subgroup of finite index, that $\overline{\Gamma}^{(w)}$ is a pro-$p_w$
group. It is enough to prove that
\begin{equation}\label{E:Ex222}
\overline{\Gamma}^{(\mathscr{V} \cup \mathscr{W})} =
\overline{\Gamma}^{(\mathscr{V})} \times \prod_{w \in \mathscr{W}}
\overline{\Gamma}^{(w)} =: \Theta.
\end{equation}
Since the primes $p_w,$ $w \in \mathscr{W},$ are pairwise distinct
and none of them is contained in $\Pi_{\mathscr{V}},$ we conclude
that $\overline{\Gamma}^{(w)}$ is the unique Sylow $p_w$-subgroup of
$\Theta,$ for all $w \in \mathscr{W}.$ As the projection
$\overline{\Gamma}^{(\mathscr{V} \cup \mathscr{W})} \to
\overline{\Gamma}^{(w)}$ is a surjective homomorphism of profinite
groups, a Sylow pro-$p_w$ subgroup of
$\overline{\Gamma}^{(\mathscr{V} \cup \mathscr{W})}$ must map
onto $\overline{\Gamma}^{(w)}.$ This implies  that
$\overline{\Gamma}^{(w)} \subset \overline{\Gamma}^{(\mathscr{V}
\cup \mathscr{W})}$ for each $w \in \mathscr{W},$  and
(\ref{E:Ex222}) follows.
\end{proof}

\vskip2mm

\vskip2mm

Continuing the proof of Theorem \ref{T:Ex2}, we fix a matrix
realization of $G$ as a $K$-subgroup of $\mathrm{GL}_n,$ and pick a
finitely generated subring $R$ of $K$ such that $\Gamma \subset
\mathrm{GL}_n(R).$ We will now argue by induction on $m.$ Let $r$ be the number
of nontrivial conjugacy classes in the (absolute) Weyl group of $G$. For $m =
1$ the argument basically mimics the proof of Theorem 2 in
\cite{PR4}. More precisely, by Proposition 1 of \cite{PR4}, we can
choose $r$ distinct primes $p_1, \ldots , p_r \notin \Pi_V$ such
that for each $i \in \{1, \ldots , r\}$ there exists an embedding
$\iota_{p_i} \colon L \hookrightarrow \Q_{p_i}$ such that
$\iota_{p_i}(R) \subset \Z_{p_i}$ and $G$ splits over $\Q_{p_i}$.
For a nontrivial discrete valuation $v$ of $K$ and a given maximal
$K_v$-torus $T$ of $G,$ we let $\mathscr{U}(T , v)$ denote the set
of elements of the form $gtg^{-1}$, with $t \in T(K_v)$ regular and
$g \in G(K_v).$ It is known that $\mathscr{U}(T , v)$ is a
solid\footnote{We recall that a subset of a topological group was
called {\it solid} in \cite{PR6} if it meets every open subgroup of
that group.} open subset of $G(K_v)$ (cf.\,\cite{PR6}, Lemma 3.4).
Let $v_i$ be pullback to $L$ of the $p_i$-adic valuation on
$\Q_{p_i}$ under $\iota_{p_i}$ (so that $L_{v_i} = \Q_{p_i}$). Let
$T(v_1), \ldots , T(v_r)$ be the tori given by Theorem \ref{T:Ex0}.
By our construction, for each $i\leqslant r,$ the group $\Gamma$ is
contained in $G(\Z_{p_i}),$ hence is relatively compact. Thus
Lemma~\ref{L:Ex1} applies, and since for any $v \in V \cup
\{v_1,\ldots\, ,v_r\},$ the group $G(K_v)$ contains a torsion-free
open subgroup, it follows from Lemma \ref{L:Ex1} that there exists
an element of infinite order
$$
\gamma_1 \in \Gamma \bigcap \left(\prod_{v \in V} \mathscr{U}(T_1(v)
, v) \times \prod_{i\leqslant r} \mathscr{U}(T(v_i) , v_i) \right),
$$
and this element is as required. For $m > 1,$ we proceed as in the
proof of Theorem \ref{T:Ex1}. Suppose that the elements $\gamma_1,
\ldots , \gamma_{m-1}$ for which the corresponding $T_1, \ldots ,
T_{m-1}$ satisfy $(i)$ and $(ii),$ and are independent over $L,$
have already been found. Let $L'$ denote the compositum of the
fields $L_{T_1}, \ldots , L_{T_{m-1}}.$ We then again use
Proposition 1 of \cite{PR4} to find $r$ distinct primes $p'_1,
\ldots , p'_r \notin \Pi_V$ such that for each $i\leqslant r$, there
exists an embedding $\iota'_{p'_i} \colon L' \hookrightarrow
\Q_{p'_i}$ with the property $\iota'_{p'_i}(R) \subset \Z_{p'_i}$.
As $G$ splits over $L'$, it splits over $\Q_{p'_i}$. Let $v'_i$ be
the pullback of the $p'_i$-adic valuation on $\Q_{p'_i}$ under
$\iota'_{p'_i}$ (and then $L'_{v'_i} = \Q_{p'_i}$). We use Theorem
\ref{T:Ex0} to find, for each $i\leqslant r$, an $L'_{v'_i}$-torus
$T'(v'_i)$ of $G$ such that for any maximal $K$-torus $T'$ of $G$
which is conjugate to $T'(v'_i)$ by an element of $G(L'_{v'_i})$ for
all $i\leqslant r$, we have $$\theta_{T'}(\Ga (L'_{T'}/L') )\supset
W(G,T').$$ As above, there exists an element of infinite order
$$
\gamma_m \in \Gamma \bigcap \left(\prod_{v \in V} \mathscr{U}(T_m(v)
, v) \times \prod_{i\leqslant r} \mathscr{U}(T'(v'_i) , v'_i)
\right)
$$
Then $\gamma_m$ clearly satisfies $(i)$ and $(ii),$ and the fact
that $T_1, \ldots , T_m$ are independent over $L$ is established
just as in the proof of Theorem \ref{T:Ex1}.
\end{proof}

\vskip5mm
%
%\section{Property $(C_i)$}
%
%Let $G_1, \ldots , G_r$ be absolutely almost simple algebraic groups
%defined over a field $F$ of characteristic zero, and let $\Gamma_i$ be a Zariski-dense subgroup
%of $G_i(F)$ for $i \leqslant r.$
%
%\vskip2mm
%
%\noindent {\bf Definition.} Fix $i \in \{1, \ldots , r\}.$ We say
%that the collection of subgroups $\Gamma_1, \ldots , \Gamma_r$ has
%{\it property $(C_i)$} if for any $m \geqslant 1$ there exist
%regular semi-simple elements $\gamma_1, \ldots , \gamma_m \in
%\Gamma_i$ that are multiplicatively independent and are not weakly
%contained in $\Gamma_1\times\ldots \times \Gamma_{i-1}\times \Gamma_{i+1}\times \ldots
%\times \Gamma_r.$
%
%
%\vskip5mm
%
%\noindent Variant to consider:
%
%\vskip3mm
%
%\noindent such that
%
%\vskip2mm
%
%(1) $\gamma_j$ generates a Zariski-dense subgroup of $T_j :=
%Z_{G_i}(\gamma_j)^{\circ}$ for all $j\leqslant m;$
%
%\vskip1mm
%
%(2) the elements $\gamma_1, \ldots , \gamma_m \in \Gamma_i$ that are
%multiplicatively independent and are not weakly contained in $\Gamma_1\times\ldots \times \Gamma_{i-1}\times \Gamma_{i+1}\times \ldots
%\times \Gamma_r.$
%
%\vskip5mm

\section{Field of definition}\label{S:FDef}

Let $G_1$ and $G_2$ be connected absolutely simple algebraic groups of adjoint type 
defined over a field $F$ of characteristic zero. As before, we let $w_i$
denote the order of the (absolute) Weyl group of $G_i$ for $i = 1,
2.$ Suppose that for each $i \in \{1 , 2\}$ we are given a finitely
generated Zariski-dense subgroup $\Gamma_i$ of $G_i(F).$ Our goal in
\S\S \ref{S:FDef}-\ref{S:AG} is to develop a series of conditions
which must hold in order to prevent the subgroups $\Gamma_1$ and
$\Gamma_r$ from satisfying condition $(C_i)$ (see Definition 3 in
\S\ref{S:I}) for at least one $i \in \{1 , 2\}.$ Here is our
first, rather straightforward, result in this direction.

\vskip2mm

\begin{thm}\label{T:property-Ci}
$(i)$ If every regular semi-simple element $\gamma \in \Gamma_1$ of
infinite order is weakly contained in $\Gamma_2$ then $\mathrm{rk}\:
G_{1} \leqslant \mathrm{rk}\: G_2$ and $w_1$ divides $w_2$.

\vskip2mm

$(ii)$ If $w_1 > w_2$, then property $(C_1)$ holds.
\end{thm}
\begin{proof}
$(i)$ We fix a finitely generated subfield $K$ of $F$ such that for
$i = 1$ and $2$, the group $G_i$ is defined and of inner type over
$K$ and $\Gamma_i \subset G_i(K).$  By Theorem \ref{T:Ex2}, there
exists a regular semi-simple element $\gamma \in \Gamma_{1}$ of
infinite order such that for the corresponding torus $T =
Z_{G_1}(\gamma)^{\circ}$ we have
$$
\theta_T(\Ga(K_T/K)) \supset W(G_{1} , T);
$$
we notice that since $G_1$ is of inner type over $K,$ this inclusion
is actually an equality, cf.\:Lemma 4.1 of \cite{PR6}. The fact that
$\gamma$ is weakly contained in $\Gamma_2$ means that one can find
semi-simple elements $\gamma^{(2)}_1, \ldots , \gamma^{(2)}_{m_2}
\in \Gamma_2$ so that for some characters $\chi \in X(T)$ and
$\chi^{(2)}_j \in X(T^{(2)}_j),$ where $T^{(2)}_j$ is a maximal
$K$-torus of $G_2$ containing $\gamma^{(2)}_j,$ there is a relation
of the form
$$
\chi(\gamma) =  \chi_1^{(2)}(\gamma_1^{(2)}) \cdots
\chi_{m_2}^{(2)}(\gamma_{m_2}^{(2)}) \neq 1.
$$
Then it follows from Theorem \ref{T:10} that for some $j \leqslant
m_2,$ there exists a surjective $K$-homomorphism $T_j^{(2)} \to T.$
Then $\mathrm{rk}\: G_{1} \leqslant \mathrm{rk}\: G_2$ and there
exists a surjective homomorphism $\Ga(K_{T_j^{(2)}}/K) \to
\Ga(K_T/K).$ Since $$\theta_{T_j^{(2)}}(\Ga(K_{T_j^{(2)}}/K))
\subset W(G_2, T_j^{(2)})$$ (Lemma 4.1 of  \cite{PR6}), our
assertion follows.

\vskip2mm

$(ii)$ The argument here basically repeats the argument given above
with minor modifications. Let $K$ be chosen as in the proof of
$(i).$ To verify property $(C_{1}),$ we use Theorem \ref{T:Ex2} to
find, for any given $m \geqslant 1,$ regular semi-simple elements
$\gamma_1, \ldots , \gamma_m \in \Gamma_{1}$ of infinite order such
that for the corresponding maximal $K$-tori $T_i =
Z_{G_{1}}(\gamma_i)^{\circ}$ of $G_{1}$ we have
$$
\theta_{T_i}(\Ga(K_{T_i}/K)) \supset W(G_{1} , T_i) \ \ \ \text{for
all} \ \  i\leqslant m,
$$
and the tori $T_1, \ldots , T_m$ are independent over $K.$ Then the
elements $\gamma_1, \ldots , \gamma_m$ are multiplicatively
independent by Lemma \ref{L:P1}, and we only need to show that they
are not weakly contained in $\Gamma_2$ given that $w_1 > w_2$.
Otherwise, we would have a relation of the form
$$
\chi_1(\gamma_1) \cdots \chi_m(\gamma_m) =
\chi_1^{(2)}(\gamma_1^{(2)}) \cdots
\chi_{m_2}^{(2)}(\gamma_{m_2}^{(2)}) \neq 1
$$
with $\chi_j \in X(T_j)$ and the other objects as in the proof of
$(i).$ Invoking again Theorem \ref{T:10}, we see that for some $i
\leqslant m$ and $j \leqslant m_2$, there exists a surjective
$K$-homomorphism $T_j^{(2)} \to T_i.$ As above, this implies that
$w_1$ divides $w_2$, contradicting the fact that by our assumption
$w_1 > w_2$.
\end{proof}

\vskip2mm

Now, let $K_i = K_{\Gamma_i}$ denote the field of definition of
$\Gamma_i,$ i.e.\:the subfield of $F$ generated by the traces $\Tr\:
\mathrm{Ad}_{G_i}(\gamma)$ for all $\gamma \in \Gamma_i$ (cf.\,\cite{Vin}).
Since $\Gamma_i$ is finitely generated, $\mathrm{Ad}_{G_i}(\Gamma_i)$  is contained
in $\mathrm{GL}_{n_i}(F_i)$ for some finitely generated subfield
$F_i$ of $F.$ Then $K_i$ is a subfield of $F_i,$ hence it is  finitely
generated. Since $G_i$ is adjoint, according to the results of Vinberg \cite{Vin},  
it is defined over $K_i$ and $\Gamma_i \subset
G_i(K_i).$

\vskip1mm

The following theorem, announced in the introduction, is the main
result of this section.

\vskip2mm

\begin{thm}\label{T:F1} $(i)$ If $w_1
> w_2$ then condition $(C_1)$ holds;

\vskip2mm

\noindent (ii) If $w_1 = w_2$ but $K_1 \not\subset K_2$ then again
$(C_1)$ holds.

\vskip2mm

\noindent Thus, unless $w_1 = w_2$ and $K_1 = K_2$, condition
$(C_i)$ holds for at least one $i \in \{1 , 2\}$.
%
%
%If $I(i_0) = \{ i_0 \}$ then condition $(C_{i_0})$ holds. Thus
%property $(C_i)$ holds for some $i \in I$ unless there are two
%distinct $i_1 , i_2 \in I$ with $K_{i_1} = K_{i_2}.$
%In other words, if $(C_{i_0})$ does not hold, then there exists $i
%\in I$, $i\ne i_0$, with $K_{i} = K_{i_0}.$
\end{thm}
%\begin{proof}
\begin{proof}
Assertion $(i)$ has already been established in Theorem
\ref{T:property-Ci}. For $i=$ 1, 2, as the group $G_i$ has been assumed 
to be of adjoint type, it is
defined over $K_i$ and $\Gamma_i \subset G_i(K_i)$. 
Set $K = K_1K_2$, and pick a finite extension $L$ of $K$ so that
$G_i$ splits over $L$ for both $i \in \{1 , 2\}$; clearly, $L$ is
finitely generated. Fix a matrix realization of $G_1$ as a $K_1$-subgroup of
$\mathrm{GL}_n$, and pick a finitely generated subring $R$ of $K_1$
so that $\Gamma \subset G_1(R)$.
%Let $r$ be the number of the conjugacy classes of
%the Weyl group of $G_1$. Using Proposition 1 of \cite{PR4} we can
%find $r$ distinct primes $p_1, \ldots , p_r$ such that for each $i
%\leqslant r$ there exists an embedding $\varepsilon_i \colon L
%\hookrightarrow \Q_{p_i}$ with $\varepsilon_i(R) \subset \Z_{p_i}$.
%Let $v_i$ denote the valuation of $K_1$ obtained as the pullback of
%the $p_i$-adic valuation under $\varepsilon_i$. Set $V' = \{v_1,
%\ldots , v_r\}$; then $\Pi_{V'} = \{p_1, \ldots , p_r\}$. Now, using
%Theorem \ref{T:Ex0} we can find a maximal ${K_1}_v$-torus $T^{(v)}$
%of $G_1$ for each $v \in V'$ so that for any maximal $K$-torus $T$
%of $G_1$ that is conjugate to $T^{(v)}$ by an element of $G_1(K_v)$
%for all $v \in V'$ we have
%\begin{equation}\label{E:F8801}
%\theta_T(\Ga(L_T/L) = W(G_1 , T)
%\end{equation}
%(as we already noted earlier, the equality automatically follows
%from the containment of $W(G_1 , T)$ in the image of $\theta_T$ as
%$G_1$ is of inner type over $L$).
Since by our assumption $K_1 \not\subset K_2,$ we have $K_2
\subsetneqq K \subset L$. So, using Proposition 5.1 of \cite{PR6},
we can find a prime $q$ such that there exists a pair of embeddings
$$
\iota^{(1)} , \iota^{(2)} \colon L \hookrightarrow \Q_q
$$
which have the same restrictions to $K_2$ but different restrictions
to $K$, hence to $K_1$, and which satisfy the condition
$\iota^{(j)}(R) \subset \Z_q$ for $j = 1, 2.$ Let $v^{(j)}$ be the
pullback to $K_1$ of the $q$-adic valuation of $\Q_q$ under
$\iota^{(j)} \vert_{K_1}$. The group $G_1((K_1)_{v^{(j)}})$ can be
naturally identified with $G^{(j)}_1(\Q_q)$, where $G^{(j)}_1$
denotes the algebraic $\Q_q$-group obtained from the $K_1$-group $G_1$
by the extension of scalars $\iota^{(j)} \vert_{K_1} \colon K_1 \to
\Q_q$, for $j = 1, 2$. Since $\iota^{(1)}$ and $\iota^{(2)}$ have
different restrictions to $K_1$, it follows from Proposition 5.2 of
\cite{PR6} that the closure of the image of $\Gamma_1$ under the diagonal embedding
%$$
%\Gamma_1 \longrightarrow G^{(1)}_1(\Q_q) \times G^{(2)}_1(\Q_q)
%$$
%is open. On the other hand, $G^{(j)}_1(\Q_q)$ can be naturally
%identified with $G_1(K_{1 v^{(j)}})$, so we conclude that the
%closure of the image of the diagonal map
\begin{equation}\label{E:F8802}
\Gamma_1 \longrightarrow G_1((K_1)_{v^{(1)}}) \times G_1((K_1)_
{v^{(2)}})
\end{equation}
is open. By our construction, $G_1$ splits over $(K_1)_{v^{(1)}} =
\Q_q$ (recall that $\iota^{(1)}(L) \subset \Q_q$ and $G_1$ splits
over $L$), so we can pick a $(K_1)_{v^{(1)}}$-split torus
$T^{(v^{(1)})}$ of $G_1$. Furthermore, by Theorem 6.21 of \cite{PlR}
there exists a maximal $(K_1)_{v^{(2)}}$-torus $T^{(v^{(2)})}$ of
$G_1$ which is anisotropic over $(K_1)_ {v^{(2)}}$.
%Set
%$$
%V'' = \{v^{(1)} , v^{(2)}\} \ \ \ \text{and} \ \ \ V = V' \bigcup
%V''.
%$$
%It follows from the openness of the closure of the image of
%(\ref{E:F8802}) and Lemma \ref{L:Ex1} that the closure of the image
%of the diagonal embedding $\Gamma_1 \hookrightarrow (G_1)_V$ is open
%as well.

%\vskip1cm

Set $V = \{ v^{(1)} , v^{(2)} \}$. It follows from Theorem
\ref{T:Ex2} that for any $m \geqslant 1$ there exist regular
semi-simple elements $\gamma_1, \ldots , \gamma_m \in \Gamma_1$ of
infinite order such that the maximal tori $T_i =
Z_{G_1}(\gamma_i)^{\circ}$ for $i \leqslant m$ are independent over
$L$ and satisfy the following conditions for all $i \leqslant m$:

\vskip2mm

\noindent $\bullet$ $\theta_{T_i}(\Ga(L_{T_i}/L)) \supset W(G_1 ,
T_i)$;

\vskip1mm

\noindent $\bullet$ $T_i$ is conjugate to $T^{(v)}$ for $v \in V$.

\vskip2mm

\noindent We claim that these elements are allow us to check the property 
$(C_1)$. Indeed, it follows from Lemma \ref{L:P1} that these
elements are multiplicatively independent, and we only need to show
that they are not weekly contained in $\Gamma_2$. Assume the
contrary. As $w_1 =w_2$, we conclude that $\mathrm{rk}\: G_1 =\mathrm{rk}\:G_2$, and  there exists a
maximal $K_2$-torus $T'$ of $G_2$ that admits an $L$-isogeny
$\kappa \colon T' \to T$ onto $T = T_i$ for some $i \leqslant m$ (see the proof of Theorem
\ref{T:property-Ci}(ii)),
and then
$$
L_{T} = L_{T'} =: \cF.
$$
Observe that
\begin{equation}\label{E:F8803}
\cF = L \cdot {K_1} _{T} = L \cdot {K_2}_{T'}.
\end{equation}
Fix some extensions
$$
\tilde{\iota}^{(1)} , \tilde{\iota}^{(2)} \colon \cF \to
\overline{\Q}_q \ \ \ \ \ (\overline{\Q}_q \ \text{is the algebraic
closure of} \ \Q_q)
$$
of $\iota^{(1)}$ and $\iota^{(2)}$ respectively. Let $u$ be the
pullback to $K_2$ of the $q$-adic valuation of $\Q_q$ under
$\iota^{(1)} \vert_{K_2} = \iota^{(2)} \vert_{K_2}$. Furthermore,
let $\tilde{v}^{(1)} , \tilde{v}^{(2)}$ (resp., $\tilde{u}^{(1)} ,
\tilde{u}^{(2)}$) be the valuations of ${K_1}_{T}$ (resp., of ${K_2}_{T'}$) 
obtained as pullbacks of the valuation of $\overline{\Q}_q$
under appropriate restrictions of $\tilde{\iota}^{(1)}$ and
$\tilde{\iota}^{(2)}$. Then $\tilde{u}^{(1)}$ and $\tilde{u}^{(2)}$
are two extensions of $u$ to the Galois extension ${K_2}_{T'}/K_2$,
and therefore
\begin{equation}\label{E:F8804}
\left[\left( {K_2}_{T'} \right)_{\tilde{u}^{(1)}} : \left( K_2
\right)_u \right] = \left[\left( {K_2}_{T'} \right)_{\tilde{u}^{(2)}}
: \left( K_2 \right)_u \right].
\end{equation}
On the other hand, since $\iota^{(j)}(L) \subset \Q_q$ for $j = 1,
2$, we have
$$
(K_2)_u = \Q_q \ \ \ \text{and} \ \ \ (K_1)_{v^{(1)}} = \Q_q =
(K_1)_{v^{(2)}}.
$$
Moreover, it follows from (\ref{E:F8803}) that
\begin{equation}\label{E:F8805}
({K_2}_{T'})_{\tilde{u}^{(j)}} = ({K_1}_{T})_{\tilde{v}^{(j)}} \ \ \
\text{for} \ \ \ j = 1, 2.
\end{equation}
But, by our construction, $T$ is $(K_1)_{v^{(1)}}$-split and $(K_1)_
{v^{(2)}}$-anisotropic. So,
$$
\left[ \left({K_1}_{T} \right)_{\tilde{v}^{(1)}} : (K_1)_{v^{(1)}}
\right] = 1 \ \ \ \text{and} \ \ \ \left[ \left({K_1}_{T}
\right)_{\tilde{v}^{(2)}} : (K_1)_{v^{(2)}} \right] \ne 1
$$
This, in view of (\ref{E:F8805}), contradicts (\ref{E:F8804}). So,
the elements $\gamma_1, \ldots , \gamma_m$ are not weakly contained
in $\Gamma_2$, verifying condition $(C_1)$.
\end{proof}

\vskip5mm

\section{Arithmetic groups}\label{S:AG}

In this section, we will treat the case where the Zariski-dense
subgroups $\Gamma_i \subset G_i(F)$ are $S$-arithmetic. For our
purposes, it is convenient to use the description of these subgroups
introduced in \cite{PR6}, \S 1, and for the reader's convenience we
briefly recall here the relevant definitions and results. So, let
$G$ be a connected absolutely almost simple algebraic group
defined over a field $F$ of characteristic zero, let $\overline{G}$
be the corresponding adjoint group, and let $\pi \colon G \to
\overline{G}$ be the natural isogeny. Suppose we are given:

\vskip2mm

$\bullet$ a number field $K$ together with a {\it fixed} embedding
$K \hookrightarrow F;$

\vskip2mm

$\bullet$ \parbox[t]{12.5cm}{an $F/K$-form $\cG$ of $\overline{G}$
(which means that the group $_{F}\!\cG$ obtained by the base change
$K \hookrightarrow F$ is $F$-isomorphic to $\overline{G}$);}

\vskip2mm

$\bullet$ \parbox[t]{12.5cm}{a finite set $S$ of places of $K$ that
contains $V^{\infty}_K$ but does not contain any nonarchimedean
places where $\cG$ is anisotropic.}

\vskip2mm

\noindent We then have an embedding $\iota \colon \cG(K)
\hookrightarrow \overline{G}(F),$ which is well-defined up to an
$F$-automorphism of $\overline{G}.$ Now, let $\cO_K(S)$ be the ring
of $S$-integers in $K$ (with $\cO_K = \cO_K(V^{\infty}_K)$ denoting
the ring of algebraic integers in $K$). Fix a $K$-embedding $\cG
\hookrightarrow \mathrm{GL}_n,$ and set $\cG(\cO_K(S)) = \cG(K) \cap
\mathrm{GL}_n(\cO_K(S)).$ A subgroup $\Gamma \subset G(F)$ is called $(\cG,
K, S)$-{\it arithmetic} if $\pi(\Gamma)$ is commensurable with
$\sigma(\iota(\cG(\cO_K(S))))$ for some $F$-automorphism $\sigma$ of
$\overline{G}.$ As usual, $(\cG, K, V^{\infty}_K)$-arithmetic
subgroups will simply be called $(\cG, K)$-arithmetic. We recall
(Lemma 2.6 of \cite{PR6}) that if $\Gamma \subset G(F)$ is a
Zariski-dense $(\cG, K, S)$-arithmetic subgroup then the trace field
$K_{\Gamma}$ coincides with $K.$

\vskip1mm

Now, for $i = 1, 2,$ let $G_i$ be a connected absolutely 
simple $F$-group of adjoint type. We will say that
the subgroups $\Gamma_i \subset G_i(F)$ are {\it commensurable up to
an $F$-isomorphism} between ${G}_1$ and ${G}_2$ if
there exists an $F$-isomorphism $\sigma \colon {G}_1 \to
{G}_2$ such that $\sigma(\Gamma_1)$ is commensurable
with $\Gamma_2$ in the usual sense, i.e. their intersection
is of finite index in both of them. According to Proposition 2.5 of
\cite{PR6}, if $\Gamma_i$ is a Zariski-dense $(\cG_i, K_i,
S_i)$-arithmetic subgroup of $G_i(F)$ for $i = 1, 2,$ then
$\Gamma_1$ and $\Gamma_2$ are commensurable up to an $F$-isomorphism
between ${G}_1$ and ${G}_2$ if and only if $K_1 =
K_2 =:K,$ $S_1 = S_2$ and $\cG_1$ and $\cG_2$ are $K$-isomorphic.

In this section, unless stated otherwise, we will assume that the absolute Weyl groups of $G_1$
and $G_2$ are of equal order.

\vskip1mm

\begin{thm}\label{T:ArG1}
Let $G_1$ and $G_2$ be connected absolutely simple algebraic groups of 
adjoint type defined 
over a field $F$ of characteristic zero such that $w_1 = w_2$, and
let $\Gamma_i \subset G_i(F)$ be a Zariski-dense $(\cG_i, K_i,
S_i)$-arithmetic subgroup for $i = 1, 2$. Furthermore, let $L_i$ be
the minimal Galois extension of $K_i$ over which $\cG_i$ becomes an
inner form. Then, unless {\bf all} of the following conditions are
satisfied:

\vskip2mm

\noindent {\rm (a)} $K_1 = K_2 =: K$,

\vskip1mm

\noindent {\rm (b)} $\mathrm{rk}_{K_v}\,\cG_1 = \mathrm{rk}_{K_v}\,
\cG_2$ for all $v \in V^K$,

\vskip1mm

\noindent {\rm (c)} $L_1 = L_2$,

\vskip1mm

\noindent {\rm (d)} $S_1 = S_2$,

\vskip2mm

\noindent condition $(C_i)$ holds for at least one $i \in \{1
, 2\}$.
\end{thm}
\begin{proof}
(a): Since the trace field $K_{\Gamma_i}$ coincides with $K_i$, our
assertion in case (a) fails to hold follows from Theorem \ref{T:F1}.
So, in the rest of the proof we may (and we will) assume that $K_1 =
K_2 =: K$. Then $\Gamma_i \subset \cG_i(K)$ for $i = 1, 2$.

\vskip2mm

(b): Suppose that for some $v_0 \in V^K$ we have
\begin{equation}\label{E:ArG0}
\mathrm{rk}_{K_{v_0}}\,\cG_1 > \mathrm{rk}_{K_{v_0}}\,\cG_2.
\end{equation}
We will now show that condition $(C_1)$ holds. Set $V = S_1 \cup \{
v_0 \},$ and for each $v \in V$ pick a maximal $K_v$-torus $T^{(v)}$
of $\cG_1$ satisfying $\mathrm{rk}_{K_v}\, T^{(v)} =
\mathrm{rk}_{K_v}\, \cG_1$. Given $m \geqslant 1$, we can use
Theorem \ref{T:Ex1} to find maximal $K$-tori $T_1, \ldots , T_m$ of
$\cG_1$ that are independent over $L_1$ and satisfy the following
properties for each $i\leqslant m$:

\vskip2mm

\noindent $\bullet$ $\theta_{T_i}(\Ga({L_1}_{T_i}/L_1)) = W(\cG_1 ,
T_i)$;

\vskip1mm

\noindent $\bullet$ $T_i$ is conjugate to $T^{(v)}$ by an element of
$\cG_1(K_v)$ for all $v \in V$.

\vskip2mm

\noindent We recall that by Dirichlet's Theorem (cf.\,\cite{PlR},
Theorem 5.12), for a $K$-torus $T$ and a finite subset $S$ of  $V^K$
containing $V_{\infty}^K$ we have
$$
T(\cO_K(S)) \simeq H \times \Z^{d_T(S) - \mathrm{rk}_K\: T},
$$
where $H$ is a finite group and $d_T(S) = \sum_{v \in S}
\mathrm{rk}_{K_v}\: T$. Since $\Gamma_1$ has been assumed to be Zariski-dense in $\cG_1$, it is infinite, and hence, $ \sum_{v
\in S_1} \mathrm{rk}_{K_v}\: \cG_1 > 0$. Now  we have
$$
d_{T_i}(S_1) := \sum_{v \in S_1} \mathrm{rk}_{K_v}\:T_i = \sum_{v
\in S_1} \mathrm{rk}_{K_v}\,\cG_1 > 0.
$$ As $T_i$ is clearly $K$-anisotropic, we conclude from the above that 
the group $T_i(\cO_K(S_1))$ contains a
subgroup isomorphic to $\Z^{d_{T_i}(S_1)}$, and so, in particular,  one
can find an element $\gamma_i \in \Gamma_1 \cap T_i(K)$ of infinite
order. We will use the elements $\gamma_1, \ldots , \gamma_m$ to verify property $(C_1)$. 
Indeed, these elements are
multiplicatively independent by Lemma \ref{L:P1}, and it remains to
show that they are not weakly contained in $\Gamma_2$. Otherwise,
there would exist a relation of the form
\begin{equation}\label{E:ArG3}
\chi_1(\gamma_1) \cdots \chi_m(\gamma_m) =
\chi^{(2)}_1(\gamma^{(2)}_1) \cdots
\chi^{(2)}_{m_2}(\gamma^{(2)}_{m_2}) \neq 1
\end{equation}
for some semi-simple elements $\gamma^{(2)}_1, \ldots ,
\gamma^{(2)}_{m_2} \in \Gamma_2 \subset \cG_2(K)$, some
characters $\chi_i \in X(T_i)$, some tori $T_j^{(2)}\subset \cG_2$ such that $\gamma_j^{(2)}\in T_j^{(2)}(K)$ and some characters $\chi^{(2)}_j \in X(T^{(2)}_j)$.
Since $w_1 = w_2$ and therefore $G_1$ and $G_2$ have the same
absolute rank, it would follow from Theorem \ref{T:10} that for some $i
\leqslant m$ and $j \leqslant m_2$ there is a $K$-isogeny $T^{(2)}_j \to T_i$, and therefore
$$
\mathrm{rk}_{K_{v_0}}\,T_i = \mathrm{rk}_{K_{v_0}}\,T^{(2)}_j.
$$
Since by our choice 
$$
\mathrm{rk}_{K_{v_0}}\,T_i = \mathrm{rk}_{K_{v_0}}\,\cG_1 \ \ \
\text{and} \ \ \ \mathrm{rk}_{K_{v_0}}\,T^{(2)}_j \leqslant
\mathrm{rk}_{K_{v_0}}\,\cG_2,
$$
this would contradict (\ref{E:ArG0}).

\vskip2mm

(c): Let us show that $L_1 = L_2$ automatically follows from the
fact that
\begin{equation}\label{E:ArG1}
\mathrm{rk}_{K_v}\,\cG_1 = \mathrm{rk}_{K_v}\,\cG_2 \ \ \text{for
all} \ \ v \in V^K
\end{equation}
(which we may assume in view of (b)). By symmetry, it is enough to
establish the inclusion $L_1 \subset L_2$. Assume the contrary. Then
for the finite Galois extension $L := L_1L_2$ of $K$ we can find a
nontrivial element $\sigma \in \Ga(L/L_2) \subset \Ga(L/K)$.
According to Theorem 6.7 of \cite{PlR}, there exists a finite subset
$S$ of $V^K$ such that for any $v \in V^K \setminus S$, the group
$\cG_2$ is quasi-split over $K_v$. Furthermore, by Chebotarev's
Density Theorem, there exists a nonarchimedean place $v \in V^K
\setminus S$ with the property that for its extension $\bar{v}$ to
$L$, the field extension $L_{\bar{v}}/K_v$ is unramified and its
Frobenius automorphism $\mathrm{Fr}(L_{\bar{v}} \vert K_v)$ is
$\sigma$. Then $L_2 \subset K_v$, and therefore $\cG_2$ is
$K_v$-split. On the other hand, $L_1 \not\subset K_v$, implying that
$\cG_1$ is not $K_v$-split. Since $G_1$ and $G_2$ have the same
absolute rank (as $w_1 = w_2$), this contradicts (\ref{E:ArG1}).

\vskip2mm

(d): If $S_1 \neq S_2$ then, by symmetry, we can assume that there
exists $v_0 \in S_1 \setminus S_2$ (any such $v_0$ is automatically
nonarchimedean). We will show that then condition $(C_1)$ holds. As
in part (b), for a given $m \geqslant 1$, we can pick maximal
$K$-tori $T_1, \ldots , T_m$ of $\cG_1$ so that they are independent
over $L_1$ and satisfy the following conditions for each $i
\leqslant m$:

\vskip2mm

\noindent $\bullet$ $\theta_{T_i}(\Ga({L_1}_{T_i}/L_1)) = W(\cG_1 ,
T_i)$;

\vskip1mm

\noindent $\bullet$ $\mathrm{rk}_{K_{v_0}}\,T_i =
\mathrm{rk}_{K_{v_0}}\,\cG_1$.

\vskip2mm

\noindent Due to our convention that $S_1$ does not contain any
nonarchimedean anisotropic places for $\cG_1,$ we have
$\mathrm{rk}_{K_{v_0}}\,T_i= \mathrm{rk}_{K_{v_0}}\, \cG_1 > 0$,
hence
$$
d_{T_i}(S_1 \setminus \{ v_0 \}) < d_{T_i}(S_1).
$$
Consequently, it follows from Dirichlet's Theorem (cf.\,(b)) that one
can pick $\gamma_i \in \Gamma_1 \cap T_i(\cO_K(S_1))$ so that its
image in $T_i(\cO_K(S_1))/T_i(\cO_K(S_1 \setminus \{ v_0 \}))$ has
infinite order for $i= 1, \ldots , m$. We claim that the elements
$\gamma_1, \ldots , \gamma_m$ verify property $(C_1)$.

As in (b), these elements are multiplicatively independent by Lemma
\ref{L:P1}, and we only need to show that they are not weakly
contained in $\Gamma_2$. Assume the contrary. Then there exists a
relation of the form (\ref{E:ArG3}) as in (b). Invoking Lemma
\ref{L:unscramble}, we see that there exist $i \leqslant m$ and $d_i
> 0$ such that for any $\lambda_i \in d_i X(T_i)$ there is a
relation of the form
\begin{equation}\label{E:ArG4}
\lambda_i(\gamma_i) = \prod_{j = 1}^{m_2}
\lambda^{(2)}_j(\gamma^{(2)}_j)
\end{equation}
with $\lambda^{(2)}_j \in X(T^{(2)}_j)$. On the other hand, by our
construction the image of $\gamma_i$ in
$T_i(\cO_K(S_1))/T_i(\cO_K(S_1 \setminus \{ v_0 \}))$ has infinite
order, and therefore the subgroup $\langle \gamma_i \rangle$ is
unbounded in $T_i(K_{v_0})$. It follows that there exists $\lambda_i
\in d_iX(T_i)$ for which $\lambda_i(\gamma_i) \in
\overline{K_{v_0}}$ is not a unit (with respect to the extension of
$v_0$). Pick for this $\lambda_i$ the corresponding expression
(\ref{E:ArG4}). Since $v_0 \notin S_2$, for each $j \leqslant m_2$,
the subgroup $\langle \gamma^{(2)}_j \rangle$ is bounded in
$T^{(2)}_j(K_{v_0})$. Hence, the value
$\lambda^{(2)}_j(\gamma^{(2)}_j) \in \overline{K_{v_0}}$ is a unit.
Then (\ref{E:ArG4}) leads to a contradiction.
\end{proof}

\vskip2mm

\noindent {\bf Remark 5.2.} The argument used in parts (b) and (d)
actually proves the following: Let $G_1$ and $G_2$ be absolutely
simple algebraic groups defined over a field $F$ of
characteristic zero such that $w_1 = w_2$, and let $\Gamma_i \subset
G_i(F)$ be a Zariski-dense $(\cG_i, K , S_i)$-arithmetic subgroup
for $i = 1, 2$. Furthermore, let $V$  be a finite subset of $V^K$ 
containing $S_1$ and let $L$ be a finite extension of $K$. If
condition $(C_1)$ does not hold then there exists a maximal
$K$-torus $T_1$ of $\cG_1$ satisfying $\theta_{T_1}(\Ga(L_{T_1}/L))
\supset W(\cG_1 , T_1)$ and $\mathrm{rk}_{K_v}\,T_1 =
\mathrm{rk}_{K_v}\, \cG_1$ for all $v \in V$  such that for some
maximal $K$-torus $T_2$ of $\cG_2$ there is a $K$-isogeny
$T_2 \to T_1$. We will use this statement below.

\addtocounter{thm}{1}

\vskip2mm

Here is an algebraic counterpart of Theorem 2 of the introduction.

\begin{thm}\label{T:ArG2}
Let $G_1$ and $G_2$ be two connected absolutely simple algebraic groups
of the same Killing-Cartan type different from $A_n$, $D_{2n+1}$ $(n
> 1)$ and $E_6$, defined over a field $F$ of characteristic zero, and
let $\Gamma_i \subset G_i(F)$ be a Zariski-dense $(\cG_i, K_i,
S_i)$-arithmetic subgroup for $i = 1, 2$. If $\Gamma_1$ and
$\Gamma_2$ are not commensurable (up to an $F$-isomorphism between
${G}_1$ and ${G}_2$) then condition $(C_i)$ holds
for at least one  $i \in \{1 , 2\}$.
\end{thm}
\begin{proof}
If either $K_1 \neq K_2$ or $S_1 \neq S_2$, condition $(C_i)$ for
some $i \in \{1 , 2\}$ holds by Theorem \ref{T:ArG1}. So, we may
assume that
\begin{equation}\label{E:ArG5}
K_1 = K_2 =: K \ \ \ \text{and} \ \ \ S_1 = S_2 = S.
\end{equation}
We first treat the case where the common type of $G_1$ and $G_2$ is
not $D_{2n}$ $(n \geqslant 2)$, i.e. it is one of the following:
$A_1$, $B_n$, $C_n$ $(n \geqslant 2)$, $E_7$, $E_8$, $F_4$, $G_2$.
According to Theorem \ref{T:ArG1}(b), if $\mathrm{rk}_{K_v}\,\cG_1
\neq \mathrm{rk}_{K_v}\, \cG_2$ for at least one $v \in V^K$, then
condition $(C_i)$ again holds for at least one $i \in \{1 , 2\}$.
Thus, we may assume that
\begin{equation}\label{E:ArG6}
\mathrm{rk}_{K_v}\, \cG_1 = \mathrm{rk}_{K_v}\, \cG_2 \ \ \
\text{for all} \ \ \ v \in V^K.
\end{equation}
As we discussed in (\cite{PR6}, \S 6, proof of Theorem 4), for the
types under consideration (\ref{E:ArG6}) implies that $\cG_1 \simeq
\cG_2$ over $K$, combining which with (\ref{E:ArG5}), we obtain that
$\Gamma_1$ and $\Gamma_2$ are commensurable (cf.\,\cite{PR6},
Proposition 2.5).

\vskip2mm

Consideration of groups of type $D_{2n}$ relies on some additional
results. In an earlier version of this paper, these were derived
from \cite{PR7} for $n > 2$ (and then Theorem \ref{T:ArG2} was also
formulated for type $D_{2n}$ with $n > 2$). Recently, Skip Garibaldi
\cite{Gar} gave an alternate proof of the required fact which
works for all $n \geqslant 2$ (including triality forms of type
$D_4$). This led to the current (complete) form of Theorem
\ref{T:ArG2} and also showed that groups of type $D_4$ do not need
to be excluded in Theorem 4 of \cite{PR6} and its (geometric)
consequences (such as Theorem 8.16 of \cite{PR6}). Here is the precise formulation
of Garibaldi's result.
\begin{thm}\label{T:D2n-G}
{\rm (\cite{Gar}, Theorem 14)} Let $G_1$ and $G_2$ be connected absolutely simple adjoint groups 
of type $D_{2n}$ for some $n \geqslant 2$ over a global field $K$
such that $G_1$ and $G_2$ have the same quasi-split inner form --
i.e., the smallest Galois extension of $K$ over which $G_1$ is of
inner type is the same as for $G_2.$ If there exists a maximal torus
$T_i$ in $G_i$ for $i = 1$ and $2$ such that

\vskip2mm

\noindent {\rm (1)} \parbox[t]{12.5cm}{there exists a $K_{\rm
sep}$-isomorphism $\phi \colon G_1 \to G_2$ whose restriction to
$T_1$ is a $K$-isomorphism $T_1 \to T_2;$ and}

\vskip1mm

\noindent {\rm (2)} \parbox[t]{12.5cm}{there is a finite set
$\mathcal{V}$ of places of $K$ such that: \newline {\rm (a)} For all
$v \notin \mathcal{V},$ $G_1$ and $G_2$ are quasi-split over $K_v,$
\newline {\rm (b)} \parbox[t]{12cm}{For all $v \in \mathcal{V},$ $(T_i)_{K_v}$ contains a
maximal $K_v$-split subtorus in $(G_i)_{K_v};$}}

\vskip2mm

\noindent then $G_1$ and $G_2$ are isomorphic over $K.$
\end{thm}

\vskip2mm

We will actually use the following consequence of the preceding theorem. 

\vskip2mm

\begin{thm}\label{T:D2n-G'}  
Let $G_1$ and $G_2$
be connected absolutely simple algebraic groups of type $D_{2n}$ over a number
field $K$ such that

\vskip2mm

\noindent {\rm (a)} $\mathrm{rk}_{K_v}\, G_1 = \mathrm{rk}_{K_v}\,
G_2$ for all $v \in V^K;$

\vskip1mm

\noindent {\rm (b)} \parbox[t]{12.5cm}{$L_1 = L_2$ where $L_i$ is the
minimal Galois extension of $K$ over which $G_i$ becomes an inner
form.}

\vskip2mm

\noindent Let $\mathcal{V} \subset V^K$ be a finite set of places
such that $G_1$ is quasi-split over $K_v$ for $v \in V^K \setminus
\mathcal{V}.$ Let $T_1$ be a maximal $K$-torus of $G_1$ satisfying

\vskip2mm

\noindent  $(\alpha)$\  $\theta_{T_1}(\Ga(K_{T_1}/K)) \supset W(G_1
, T_1),$

\vskip1mm

\noindent $(\beta)$\  $\mathrm{rk}_{K_v}\, T_1 = \mathrm{rk}_{K_v}\,
G_1$ for all $v \in \mathcal{V}.$

\vskip2mm

\noindent If there exists a $K$-isogeny $\varphi \colon T_2 \to T_1$
from a maximal $K$-torus $T_2$ of $G_2$, then $G_1$ and $G_2$ are
isogenous over $K.$
\end{thm}

\vskip3mm

To derive Theorem \ref{T:D2n-G'} from Theorem \ref{T:D2n-G}, we
can assume that both $G_1$ and
$G_2$ are adjoint. Now note that it follows from Lemma 4.3 in \cite{PR6}
that, due to condition $(\alpha),$ one can assume without any loss
of generality that the comorphism $\varphi^* \colon X(T_1) \to
X(T_2)$ satisfies $\varphi^*(\Phi(G_1 , T_1)) =  \Phi(G_2 , T_2).$
Then $\varphi$ is actually a $K$-isomorphism of tori that
extends to a $\overline{K}$-isomorphism $\phi \colon G_2 \to
G_1.$ So, we can use Theorem \ref{T:D2n-G} to obtain Theorem \ref{T:D2n-G'}. 

\vskip3mm

To complete the proof of Theorem \ref{T:ArG2}, we observe that if
neither $(C_1)$ nor $(C_2)$ holds,  then according to
Theorem \ref{T:ArG1}, conditions (a) and (b) of Theorem
\ref{T:D2n-G'} are satisfied for $\cG_1$ and $\cG_2$. Fix a finite
set of places $V \subset V^K$ that contains $S_1$ and is big enough
so that $\cG_1$ and $\cG_2$ are quasi-split over $K_v$ for all $v
\in V^K \setminus V$. Using Remark 5.2, we can find a maximal
$K$-torus $T_1$ of $\cG_1$ that satisfies conditions $(\alpha)$ and
$(\beta)$ of Theorem \ref{T:D2n-G'} and a maximal $K$-torus $T_2$ of $G_2$ 
which is isogeneous to $T_1$ over $K$. 
Then $\cG_1 \simeq \cG_2$ over $K$ by Theorem \ref{T:D2n-G'},
making $\Gamma_1$ and $\Gamma_2$ commensurable as above.
\end{proof}

\vskip3mm

Our next result contains more restrictions on the arithmetic groups
$\Gamma_1$ and $\Gamma_2$ given the fact that both the conditions $(C_1)$ and $(C_2)$ 
fail to hold.
\begin{thm}\label{T:ArG3}
Let $G_1$ and $G_2$ be two connected absolutely simple algebraic groups
over a field $F$ of characteristic zero such that $w_1 = w_2$, and
let $\Gamma_i \subset G_i(F)$ be a Zariski-dense $(\cG_i, K,
S)$-arithmetic subgroup for $i = 1, 2$. If both $(C_1)$ and $(C_2)$ fail 
to hold, then $\mathrm{rk}_K\,
\cG_1 = \mathrm{rk}_K\, \cG_2$. Moreover, if $G_1$ and $G_2$ are of
the same Killing-Cartan type, then the Tits indices $\cG_1/K_v$ and
$\cG_2/K_v$ are isomorphic for all $v \in V^K$, and the Tits indices
$\cG_1/K$ and $\cG_2/K$ are isomorphic.
\end{thm}
\begin{proof}
The proof relies on the following statement which was actually
established in \cite{PR6}, \S 7 (although it was not stated there
explicitly).
\begin{thm}\label{T:equal-rank}
Let $G_1$ and $G_2$ be two connected absolutely  simple algebraic
$K$-groups, let $L_i$ be the minimal Galois extension of $K$ over
which $G_i$ are of inner type, and let   $\mathcal{V}$ 
be a finite subset of $V^K$ such that  both $G_1$ and $G_2$ are 
$K_v$-quasi-split for all $v \notin \mathcal{V}.$ Furthermore, let
$T_i$ be a maximal $K$-torus of $G_i,$ where $i = 1, 2,$ such that

\vskip2mm

{\rm (1)} $\theta_{T_i}(\Ga(K_{T_i}/K)) \supset W(G_i , T_i);$

\vskip1mm

{\rm (2)} $\mathrm{rk}_{K_v}\,T_i = \mathrm{rk}_{K_v}\, G_i$ for
all $v \in \mathcal{V}.$

\vskip2mm

\noindent If $L_1 = L_2$ and there exists a $K$-isogeny $T_1
\to T_2$, then $\mathrm{rk}_K\: G_1 = \mathrm{rk}_K\: G_2.$ Moreover,
if $G_1$ and $G_2$ are of the same Killing-Cartan type then the Tits
indices $G_1/K_v$ and $G_2/K_v$ are isomorphic for all $v \in V^K,$
and the Tits indices of $G_1/K$ and $G_2/K$ are isomorphic.
\end{thm}
For the reader's convenience, we will give a proof of this theorem in the Appendix.

\vskip3mm

To derive Theorem \ref{T:ArG3} from Theorem \ref{T:equal-rank}, we
basically mimic the argument used to consider type $D_{2n}$ in
Theorem \ref{T:ArG2}. More precisely, we pick a finite set $V$ of places
of $K$ containing $S_1$ so that the groups $\cG_1$ and
$\cG_2$ are quasi-split over $K_v$ for all $v \in V^K \setminus V$.
Since by our assumption both $(C_1)$ and $(C_2)$ fail to hold, 
we can use Remark 5.2 to find of maximal $K$-torus $T_1$
of $\cG_1$ that satisfies conditions (1) and (2) of Theorem
\ref{T:equal-rank} for $i = 1$, and a maximal $K$-torus $T_2$ of $\cG_2$ 
which is isogeneous to $T_1$ over $K$.  
Since $\mathrm{rk}_{K_v}\,\cG_1 = \mathrm{rk}_{K_v}\,\cG_2$, we
obtain that condition (2) holds also for $i = 2$. Furthermore,
condition (1) for $i = 1$ combined with the fact that $L_1 = L_2$,
by order consideration, yields that the inclusion
$\theta_{T_2}(\Ga({L_2}_{T_2}/L_2)) \subset W(\cG_2 , T_2)$ is in fact
an equality, so (2) holds for $i = 2$ as well. Now, applying Theorem
\ref{T:equal-rank} we obtain Theorem \ref{T:ArG3}.
\end{proof}

\vskip2mm

We conclude this section with a variant of Theorem \ref{T:ArG3}
which has an interesting geometric application (see Theorem 5 in the Introduction; 
this theorem will be proved in  \S\ref{S:FGeod}). Let $\Gamma_i$ is a Zariski-dense $(\cG_i, K_i,
S_i)$-arithmetic subgroup of $G_i,$ and assume that $\cG_1$ is
$K_1$-isotropic and $\cG_2$ is $K_2$-anisotropic. It follows from
Theorem \ref{T:ArG1} (for $K_1 \neq K_2$) and Theorem \ref{T:ArG3}
(for $K_1 = K_2$) that then condition $(C_i)$ holds for at least one
$i \in \{1 , 2\}.$ In fact, assuming that $w_1 = w_2$, one can
always guarantee that condition $(C_1)$ holds:
\begin{thm}\label{T:isotr}
Let $G_1$ and $G_2$ be two connected absolutely simple algebraic groups
with  $w_1 = w_2$. Let $\Gamma_i$ be a Zariski-dense $(\cG_i,
K_i, S_i)$-arithmetic subgroup of $G_i$ for $i = 1, 2,$ and assume
that $\cG_1$ is $K_1$-isotropic and $\cG_2$ is $K_2$-anisotropic.
Then property $(C_1)$ holds.
\end{thm}

The proof relies on the following version of Theorem
\ref{T:equal-rank} which treats the case where the fields of
definitions of $\Gamma_1$ and $\Gamma_2$ are not necessarily the
same.

\vskip2mm

\noindent {\bf Theorem \ref{T:equal-rank}$'$.} {\it For $i = 1, 2,$
let $G_i$ be a connected absolutely simple algebraic group over a
number field $K_i,$ and let $L_i$ be the minimal Galois extension of
$K_i$ over which $G_i$ is of inner type. Assume that $K_1
\subset K_2,$ $L_2 \subset K_2L_1,$ $w_1 = w_2$ and
$\mathrm{rk}_{K_1}\,G_1
> 0.$ Furthermore, let $\mathcal{V}_1 \subset V^{K_1}$ be a finite
subset such that $G_2$ is quasi-split over ${K_2}_v$ for all $v \notin
\mathcal{V}_2,$ where $\mathcal{V}_2$ consists of all extensions of
places contained in $\mathcal{V}_1$ to $K_2,$ and let $T_1$ be a maximal
$K_1$-torus of $G_1$ such that

\vskip2mm

{\rm (1)} $\theta_{T_1}(\Ga({K_1}_{T_1}/K_1)) \supset W(G_1 , T_1);$

\vskip1mm

{\rm (2)} $\mathrm{rk}_{{K_1}_v}\,T_1 = \mathrm{rk}_{{K_1}_v}\,
G_1$ for all $v \in \mathcal{V}_1.$

\vskip2mm

\noindent If there exists a maximal torus $T_2$ of $G_2$ and a $K_2$-isogeny $T_2 \to T_1$,  then
$\mathrm{rk}_{K_2}\,G_2
> 0.$}

\vskip2mm

This result is also proved in the Appendix along with Theorem
\ref{T:equal-rank}.

\vskip3mm

\noindent {\it Proof of Theorem \ref{T:isotr}.} If $K_1 \not\subset
K_2$ then the fact that $(C_1)$ holds follows from Theorem
\ref{T:F1} (cf.\,the proof of Theorem \ref{T:ArG1}(a)). So, in the
rest of the argument we may assume that $K_1 \subset K_2.$

Next, suppose that $L_2 \not\subset K_2L_1.$ In this case, the
argument imitates the proof of Theorem \ref{T:ArG1}(c). More
precisely, we have $K_2L_1 \subsetneqq L_1L_2.$ So, if $\mathfrak{L}$ is the
normal closure of $L_1L_2$ over $K_1,$ then there exists $\sigma \in
\Ga({\mathfrak L}/K_1)$ that restricts trivially to $K_2L_1$ and nontrivially to
$L_1L_2.$ By Chebotarev's Density Theorem, we can find $v_0 \in
V^{K_1} \setminus S_1$ which is unramified in ${\mathfrak L}/K_1$ and for which
the Frobenius automorphism $\mathrm{Fr}(\tilde{v}_0 \vert v_0)$
equals $\sigma$ for an appropriate extension $\tilde{v}_0 \vert v_0$, 
and in addition the group $\cG_1$ is quasi-split over ${K_1}_{v_0}$.
Let $u_0$ be the restriction of $\tilde{v}_0$ to $K_2.$ By
construction, we have $L_1 \subset {K_1}_{v_0},$ which means that
$\cG_1$ is actually split over ${K_1}_{v_0};$ at the same time, $L_2
\not\subset {K_2}_{u_0},$ and therefore $\cG_2$ is not split over
${K_2}_{u_0}.$ Set $L = L_1L_2$ and $\mathcal{V}_1 = S_1 \cup
\{v_0\}.$ Fix $m \geqslant 1,$ and using Theorem \ref{T:Ex1} pick
maximal $K_1$-tori $T_1, \ldots , T_m$ of $\cG_1$ that are
independent over $L$ and satisfy the following two conditions for
each $j \leqslant m$:

\vskip2mm

\noindent $\bullet$ $\theta_{T_j}(\Ga(L_{T_j}/L)) = W(G_1 , T_j)$;

\vskip1mm

\noindent $\bullet$ $\mathrm{rk}_{{K_1}_v}\,T_j =
\mathrm{rk}_{{K_1}_v}\, \cG_1$ for all $v \in \mathcal{V}_1$.

\vskip2mm

\noindent As in the proof of Theorem \ref{T:ArG1}(b), it follows
from Dirichlet's Theorem that one can pick elements $\gamma_j \in
\Gamma_1 \cap T_j(K_1)$ for $j \leqslant m$ of infinite order. By
Lemma~\ref{L:P1}, the elements $\gamma_1, \ldots , \gamma_m$ are
multiplicatively independent, so to establish property $(C_1)$ in the case at
hand, it remains to show that these elements are not weakly
contained in $\Gamma_2.$ Assume the contrary. Then according to
Theorem \ref{T:10} (with $K = K_2$), there exists a maximal $K_2$-torus $T^{(2)}$ 
of $\cG_2$ and a $K_2$-isogeny $T^{(2)}\to T_j $ for some $j\leqslant m$. Clearly, $T_j$ is split over
${K_1}_{v_0},$ hence also over ${K_2}_{u_0}.$ We conclude that
$T^{(2)}$  is also split over ${K_2}_{u_0},$ which is impossible as
$\cG_2$ is not ${K_2}_{u_0}$-split. This verifies property $(C_1)$
in this case. (We note that so far we have not used the assumption that $\cG_1$ is $K_1$-isotropic and $\cG_2$ is $K_2$-anisotropic.)

\vskip2mm

It remains to consider the case where $K_1 \subset K_2$ and $L_2
\subset K_2L_1.$ Here the argument is very similar to the one given
above but uses a different choice of the set $\mathcal{V}_1$ and
relies on Theorem \ref{T:equal-rank}$'$. More precisely, pick a
finite subset $\mathcal{V}_1 \subset V^{K_1}$ containing $S_1$ so
that $\cG_2$ is quasi-split over ${K_2}_v$ for all $v \in
\mathcal{V}_2,$ where $\mathcal{V}_2$ consists of all extensions of
places in $\mathcal{V}_1$ to $K_2.$ Assume that $(C_1)$ does not
hold, i.e., there exists $m \geqslant 1$ such that any $m$
multiplicatively independent semi-simple elements of $\Gamma_1$ of infinite order are
necessarily weakly contained in $\Gamma_2.$ Fix such an $m,$ and
using the same $L$ as above, pick maximal $K_1$-tori $T_1, \ldots ,
T_m$ of $\cG_1$ that are independent over $L$ and satisfy the above
bulleted conditions for this new choice of $\mathcal{V}_1.$ Arguing
as in the previous paragraph, we see that again, there exists a maximal $K_2$-torus $T^{(2)}$ of $\cG_2$ 
and a $K_2$-isogeny $T^{(2)}\to T_j$ for some $j
\leqslant m.$ Then it follows from
Theorem \ref{T:equal-rank}$'$ that $\cG_2$ is $K_2$-isotropic, a
contradiction. \hfill $\Box$

\vskip3mm

It would be interesting to determine if the assumption that $w_1 =
w_2$ in Theorem \ref{T:isotr} can be omitted.

\vskip2mm

\noindent {\bf Question.} {\it Is it possible to construct
$K_1$-isotropic $\cG_1$ and $K_2$-anisotropic $\cG_2$, with $K_1 \subset
K_2$ so that every $K_1$-anisotropic torus of $\cG_1$ is
$K_2$-isomorphic to a $K_2$-torus of $\cG_2$?}

\vskip2mm

(Obviously, the affirmative answer to this question with $K_1 = \Q$
would lead to an example where every semi-simple element of infinite
order in $\Gamma_1$ would be weakly contained in $\Gamma_2$ and
therefore $(C_1)$ would not hold.)

\vskip3mm

\vskip5mm

\section{Groups of types $A_n,$ $D_n$ and $E_6$}\label{S:ADE}

It is known that the assertion of Theorem \ref{T:ArG2} may fail if
the common Killing-Cartan type of the groups $G_1$ and $G_2$ is one
of the following: $A_n,$ $D_{2n+1}$ $(n > 1)$ or $E_6$
(cf.\:Examples 6.5, 6.6, 6.7 and \S 9 in \cite{PR6}). Nevertheless, a
suitable analog of Theorem \ref{T:ArG2} with interesting geometric
consequences can still be given (cf.\:Theorem \ref{T:ADE6} below).
It is based on the following notion.

\vskip2mm

\noindent {\bf Definition.} Let $G_1$ and $G_2$ be connected absolutely almost
simple algebraic groups defined over a field $K.$ We say that $G_1$
and $G_2$ have {\it equivalent systems of maximal $K$-tori} if for
every maximal $K$-torus $T_1$ of $G_1$ there exists a
$\overline{K}$-isomorphism $\varphi \colon G_1 \to G_2$ such that
the restriction $\varphi \vert_{T_1}$ is defined over $K,$ and
conversely, for every maximal $K$-torus $T_2$ of $G_2$ there exists
a $\overline{K}$-isomorphism $\psi \colon G_2 \to G_1$ such that  the
restriction $\psi \vert_{T_2}$ is defined over $K$.

\vskip2mm

We note that given a $\overline{K}$-isomorphism $\varphi \colon G_1
\to G_2$ as in the definition, the torus $T_2 = \varphi(T_1)$ is
defined over $K$ and the corresponding map $X(T_2) \to X(T_1)$
induces a bijection $\Phi(G_2 , T_2) \to \Phi(G_1 , T_1).$ This
observation implies that if $G_i$ is a connected  absolutely almost simple real
algebraic group, $\Gamma_i \subset G_i(\R)$ is a torsion-free
$(\cG_i , K)$-arithmetic subgroup and $\fX_{\Gamma_i}$ is the
associated locally symmetric space, where $i = 1, 2,$ then the fact
that $\cG_1$ and $\cG_2$ have equivalent systems of maximal $K$-tori
entails that $\fX_{\Gamma_1}$ and $\fX_{\Gamma_2}$ are
length-commensurable (see Proposition 9.14 of \cite{PR6}). For
technical reasons, in this section it is more convenient for us to
deal with simply connected groups rather than with adjoint ones
which are more natural from the geometric standpoint. So, we observe
in this regard that if simply connected $K$-groups $G_1$ and $G_2$
have equivalent systems of maximal $K$-tori then so do the
corresponding adjoint groups $\overline{G}_1$ and $\overline{G}_2$
(and vice versa).

\vskip2mm

We will now describe fairly  general conditions  guaranteeing that
two forms over a number field $K,$ of an absolutely almost simple
simply connected group of one of  types $A_n$, $D_{2n+1}$ $(n> 1)$,
or $E_6$, have equivalent systems of maximal $K$-tori.

\begin{thm}\label{T:equiv-tori}
Let $G_1$ and $G_2$ be two connected absolutely almost simple simply connected
algebraic groups of one of the following types: $A_n,$  $D_{2n+1}$
$(n > 1)$ or $E_6,$ defined over a number field $K,$ and let $L_i$
be the minimal Galois extension of $K$ over which $G_i$ is of
inner type. Assume that
\begin{equation}\label{E:rk}
\mathrm{rk}_{K_v}\,G_1 = \mathrm{rk}_{K_v}\, G_2 \ \ \text{for all}
\ \ v \in V^K,
\end{equation}
hence\footnotemark \:$L_1 = L_2 =: L$. Moreover, if $G_1$ and $G_2$
are of type $D_{2n+1}$ we assume that  for each real place $v$ of
$K$, we can find maximal $K_v$-tori $T_i^v$ of $G_i$ for $i = 1, 2,$
such that $\mathrm{rk}_{K_v}\,T_i^v = \mathrm{rk}_{K_v}\, G_i$ and
there exists a $K_v$-isomorphism $T_1^v \to T_2^v$ that extends to a
$\overline{K_v}$-isomorphism $G_1 \to G_2.$ If

\vskip2mm

\noindent {\rm (1)} \parbox[t]{12.5cm}{one can pick  maximal $K$-tori
$T_i^0$ of $G_i$ for $i = 1, 2$ with a $K$-isomorphism $T_1^0 \to
T_2^0$ that extends to a $\overline{K}$-isomorphism $G_1 \to G_2,$
and}

\vskip1mm

\noindent {\rm (2)} \parbox[t]{12.5cm}{there exists a place $v_0$ of $K$ such
that one of the groups $G_i$ is $K_{v_0}$-anisotropic (and then both
are such due to (\ref{E:rk})),}

\vskip2mm

\noindent then $G_1$ and $G_2$ have equivalent systems of maximal
$K$-tori. \footnotetext{As we have seen in the proof of Theorem
\ref{T:ArG1}, the former condition automatically implies the
latter.}
\end{thm}
\begin{proof}
We begin by establishing first the corresponding local assertion.
\begin{lemma}\label{L:local1}
Let $\cG_1$ and $\cG_2$ be two connected absolutely almost simple simply
connected algebraic groups of one of the following types: $A_{\ell}$
$(\ell \geqslant 1),$ $D_{\ell}$ $(\ell \geqslant 5)$ or $E_6,$ over
a nondiscrete locally compact field $\cK$ of characteristic zero,
and let $\cL_i$ be the minimal Galois extension of $\cK$ over which
$\cG_i$ is of inner type. Assume that
$$
\cL_1 = \cL_2 =: \cL \ \ \text{and} \ \ \mathrm{rk}_{\cK}\:\cG_1 =
\mathrm{rk}_{\cK}\: \cG_2,
$$
and moreover, in case $\cG_1$ and $\cG_2$ are of type $D_{\ell}$ and
$\cK = \R,$ there exist maximal $\cK$-tori $\cT_i$ of $\cG_i$ such
that $\mathrm{rk}_{\cK}\: \cT_i = \mathrm{rk}_{\cK}\: \cG_i$ for $i
= 1, 2,$ with a $\cK$-isomorphism $\cT_1 \to \cT_2$ that extends to
a $\overline{\cK}$-isomorphism $\cG_1 \to \cG_2.$ Then

\vskip2mm

\noindent \ {\rm (i)} \parbox[t]{12.5cm}{except in the case where
$\cG_1$ and $\cG_2$ are inner $K$-forms of a split group of type $A_{\ell}$ with $\ell > 1,$
we have $\cG_1 \simeq \cG_2$ over $\cK;$}

\vskip2mm

\noindent {\rm (ii)} \parbox[t]{12.5cm}{in all cases, $\cG_1$ and $\cG_2$
have equivalent systems of maximal $\cK$-tori.}
\end{lemma}
\begin{proof}
(i): First, let $\cG_1$ and $\cG_2$ be {\it outer} $\cK$-forms of a split group of type
$A_{\ell}$ associated with a quadratic extension $\cL$ of  $\cK.$ Then
$\cG_i = \mathrm{SU}(\cL , h_i)$ where $h_i$ is a nondegenerate
Hermitian form on $\cL^n,$ $n = \ell + 1,$ with respect to the
nontrivial automorphism of $\cL/\cK.$ Since $\mathrm{rk}_{\cK}\:
\cG_i$ coincides with the Witt index of the Hermitian form $h_i$,
the forms $h_1$ and $h_2$ have equal Witt index. On the other hand,
it is well-known, and easy to see, that the similarity class of an
{\it anisotropic} Hermitian form over $\cL$ is determined by its
dimension (which for nonarchimedean $v$ is necessarily $\leqslant
2$). So, the fact that $h_1$ and $h_2$ have equal Witt index implies
that $h_1$ and $h_2$ are similar, hence $\cG_1 \simeq \cG_2,$ as
required.

\vskip1mm

Now, suppose $\cG_1$ and $\cG_2$ are of type $D_{\ell}$ with $\ell
\geqslant 5.$ If $\cK = \C$ then there is nothing to prove;
otherwise there is a unique quaternion central division algebra
$\cD$ over $\cK.$ For each $i \in \{1 , 2\}$, we have two
possibilities: either $\cG_i = \mathrm{Spin}_n(q_i)$ where $q_i$ is
a nondegenerate quadratic form over $\cK$ of dimension $n = 2\ell$
(orthogonal type), or $\cG_i$ is the universal cover of
$\mathrm{SU}(\cD , h_i)$ where $h_i$ is a nondegenerate
skew-Hermitian form on $\cD^{\ell}$ with respect to the canonical
involution on $\cD$ (quaternionic type). We will now show that in
our situation, $\cG_1$ and $\cG_2$ are both of the same, orthogonal
or quaternionic, type. First, we treat the case where $\cK$ is 
nonarchimedean. Assume that $\cG_1$ is of orthogonal, and $\cG_2$ is
of quaternionic, type. Then $\mathrm{rk}_{\cK}\: \cG_1 \geqslant
(2\ell - 4)/2 = \ell - 2,$ while $\mathrm{rk}_{\cK}\: \cG_2
\leqslant \ell/2.$ So, $\mathrm{rk}_{\cK}\: \cG_1 =
\mathrm{rk}_{\cK}\: \cG_2$ is impossible as $\ell \geqslant 5,$ a
contradiction. Over $\cK = \R,$ however, one can have $\cG_1$ of
orthogonal type and $\cG_2$ of quaternionic type with the same
$\cK$-rank, so to prove our assertion in this case we need to use
the hypothesis  that there exist maximal $\cK$-tori $\cT_i$ of
$\cG_i$ such that $\mathrm{rk}_{\cK}\, \cT_i = \mathrm{rk}_{\cK}\,
\cG_i$,  with a $\cK$-isomorphism $\cT_1 \to \cT_2$ that extends to
a $\overline{\cK}$-isomorphism $\cG_1 \to \cG_2.$ Such an
isomorphism induces an isomorphism between the Tits indices of
$\cG_1/\cK$ and $\cG_2/\cK$ (cf.\:the discussion in \S 7.1 of
\cite{PR6}). However, if $\cG_1$ is of orthogonal type, and $\cG_2$
of quaternionic, the corresponding Tits indices are not isomorphic,
and our assertion follows.

Now, let $\cG_1$ and $\cG_2$ be of quaternionic type. It is known
that two nondegenerate skew-Hermitian forms over $\cD$ are
equivalent if they have the same dimension and in addition the same
discriminant in the nonarchimedean case (cf.\:\cite{Sch}, Ch.\,10,
Theorem 3.6 in the nonarchimedean case, and Theorem 3.7 in the
archimedean case). If $h_1$ and $h_2$ are the skew-Hermitian forms
defining $\cG_1$ and $\cG_2$ respectively, then the condition that $h_1$ and $h_2$
have the same discriminant is equivalent to the fact that $\cL_1 =
\cL_2,$ and therefore holds in our situation. Thus, $h_1$ and $h_2$
are equivalent, hence $\cG_1$ and $\cG_2$ are $\cK$-isomorphic.

Next, let $\cG_1$ and $\cG_2$ be of orthogonal type, $\cG_i =
\mathrm{Spin}(q_i).$ To show that $\cG_1 \simeq \cG_2,$ it is enough
to show that $q_1$ and $q_2$ are similar. The condition
$\mathrm{rk}_{\cK}\: \cG_1 = \mathrm{rk}_{\cK}\: \cG_2$ yields that
$q_1$ and $q_2$ have the same Witt index, so we just need to show
that the maximal anisotropic subforms $q_1^a$ and $q_2^a$ are
similar. If $\cK = \R$, then any two anisotropic forms of the same
dimension are similar, and there is nothing to prove. Now, let $\cK$
be nonarchimedean. Our claim is obvious if $q_1^a = q_2^a = 0;$ in
the two remaining cases the common dimension of $q_1^a$ and $q_2^a$
can only be $2$ or $4.$ To treat binary forms, we observe that $q_1$
and $q_2,$ hence also $q_1^a$ and $q_2^a,$ have the same
discriminant, and two binary forms of the same
discriminant are similar. The claim for quaternary forms follows
from the fact that there exists a single equivalence class of such
anisotropic forms (this equivalence class is represented by the
reduced-norm form of $\cD$).

\vskip1mm

Finally, we consider groups of type $E_6.$ If $\cK = \R$ then by
inspecting the tables in \cite{Ti} we find that there are two
possible indices for the inner forms with the corresponding groups
having $\R$-ranks $2$ and $6,$ and there are three possible indices
for outer forms for which the $\R$-ranks are $0,$ $2$ and $4.$ Thus,
since $G_1$ and $G_2$ are simultaneously either inner or outer forms
and have the same $\R$-rank,  they are $\R$-isomorphic. To establish
the same conclusion in the nonarchimedean case, we recall that then
an outer form of type $E_6$ is always quasi-split (cf.\:\cite{PlR},
Proposition 6.15), so for outer forms our assumption that $\cL_1 =
\cL_2$ implies that $G_1 \simeq G_2.$ Since there exists only one
nonsplit inner form of type $E_6$ (this follows, for example, from
the proof of Lemma 9.9(ii) in \cite{PR6}), our assertion holds in
this case as well.

\vskip2mm

(ii): It remains to be shown that if $\cG_1$ and $\cG_2$ are inner forms
of type $A_{\ell}$ over $\cK$ such that $\mathrm{rk}_{\cK}\: \cG_1 =
\mathrm{rk}_{\cK}\: \cG_2$, then $\cG_1$ and $\cG_2$ have equivalent
systems of maximal $\cK$-tori. We have $\cG_i = \mathrm{SL}_{d_i ,
\cD_i}$ where $\cD_i$ is a central division algebra over $\cK$ of
degree $n_i$ and
$$
\mathrm{rk}_{\cK}\, \cG_i = d_i - 1 \ \ \ \text{and} \ \ \ d_im_i =
\ell + 1 =: n.
$$
Thus, in our situation $d_1=d_2$ and $n_1 = n_2.$ Furthermore, it is
well-known (cf.\,\cite{PR7}, Proposition 2.6) that a commutative \'etale
$n$-dimensional $\cK$-algebra $\cE = \prod_{j = 1}^s \cE^{(j)},$
where $\cE^{(j)}/\cK$ is a finite (separable) field extension,
embeds in $\cA_i := M_{d_i}(\cD_i)$ if and only if each degree
$[\cE^{(j)} : \cK]$ is divisible by $n_i.$ So, we conclude that
$\cE$ embeds in $\cA_1$ if and only if it embeds in $\cA_2.$ On the
other hand, any maximal $\cK$-torus $\cT_1$ of $\cG_1$ coincides
with the torus $\mathrm{R}_{\cE_1/\cK}^{(1)}(\mathrm{GL}_1)$
associated with the group of norm one elements in some
$n$-dimensional commutative \'etale subalgebra $\cE_1$ of $\cA_1.$ As we noted
above, $\cE_1$ embeds in $\cA_2,$ and then using the Skolem-Noether
Theorem (see Footnote 1 on p.\,592  in \cite{PR7}) one can construct an isomorphism
$\cA_1 \otimes_{\cK} \overline{\cK} \simeq \cA_2 \otimes_{\cK}
\overline{\cK}$ that maps $\cE_1$ to a subalgebra $\cE_2 \subset
\cA_2.$ This isomorphism gives rise to a $\overline{K}$-isomorphism
$\cG_1 \simeq \cG_2$ that induces a $\cK$-isomorphism between
$\cT_1$ and $\cT_2 := \mathrm{R}_{\cE_2/\cK}(\mathrm{GL}_1).$ By
symmetry, $\cG_1$ and $\cG_2$ have equivalent systems of maximal
$\cK$-tori.
\end{proof}

\vskip2mm

To complete the proof of Theorem \ref{T:equiv-tori}, we fix a
$\overline{K}$-isomorphism $\varphi_0 \colon G_1 \to G_2$ such that
the restriction $\varphi_0 \vert_{T_1^0}$ is a $K$-isomorphism
between $T_1^0$ and $T_2^0.$ Let $T_1$ be an arbitrary maximal
$K$-torus of $G_1.$ Then by Lemma \ref{L:local1}, for any $v \in
V^K,$ there exists a $\overline{K_v}$-isomorphism $\varphi_v: G_1\to
G_2$ whose  restriction to $T_1$ is defined over $K_v$.  Then
$\varphi_v = \alpha \cdot \varphi_0$ for some $\alpha \in
\mathrm{Aut}\: G_2.$ There exists an automorphism of $G_2$ that acts
as $t \mapsto t^{-1}$ on $T_2 := \varphi_v(T_1)$. Moreover, for groups
of the types listed in the theorem, this automorphism represents the
only nontrivial element of $\mathrm{Aut}\: G_2 / \mathrm{Inn}\:
G_2$. So, if necessary, we can replace $\varphi_v$ by the composite 
of $\varphi_v$ with this automorphism to ensure that $\alpha$ is inner
(and the restriction of $\varphi_v$ to $T_1$ is still defined over
$K$, cf.\:the proof of Lemma 9.7 in \cite{PR6}). This shows that
$T_1$ admits a {\it coherent} (relative to $\varphi_0$)
$K_v$-embedding in $G_2$ (in the terminology introduced in
\cite{PR6}, \S 9), for every $v \in V^K.$ Since $T_1$ is
$K_{v_0}$-anisotropic, {\brus SH}$^2(T_1)$ is trivial
(cf.\:\cite{PlR}, Proposition 6.12). So, by Theorem 9.6 of
\cite{PR6}, $T_1$ admits a coherent $K$-defined embedding in $G_2$
which in particular is a $K$-embedding $T_1 \to G_2$ which extends
to a $\overline{K}$-isomorphism $G_1 \to G_2.$ By symmetry, $G_1$
and $G_2$ have equivalent systems of maximal $K$-tori.
\end{proof}

\vskip2mm

The following proposition complements Theorem \ref{T:equiv-tori} for
groups of type $A_n$ in that it does not assume the existence of a
place $v_0 \in V^K$ where the groups are anisotropic.
\begin{prop}\label{P:An}
Let $G_1$ and $G_2$ be two connected absolutely almost 
simple simply connected
algebraic groups of type $A_n$ over a number field $K,$ and let
$L_i$ be the minimal Galois extension of $K$ over which $G_i$
is of inner type. Assume that
\begin{equation}\label{E:rank}
\mathrm{rk}_{K_v}\,G_1 = \mathrm{rk}_{K_v}\,G_2 \ \ \text{for all}
\ \ v \in V^K,
\end{equation}
hence $L_1 = L_2 =: L.$ In each of the following situations:

\vskip2mm

\noindent {\rm (1)} \parbox[t]{13.3cm}{$G_1$ and $G_2$ are inner
forms,}

\vskip2mm

\noindent {\rm (2)} \parbox[t]{13.3cm}{$G_1$ and $G_2$ are outer
forms, and one of them is represented by $\mathrm{SU}(D , \tau)$, 
where $D$ is a central \emph{division} algebra over $L$ with an
involution $\tau$ of the second kind that restricts to the
nontrivial automorphism $\sigma$ of $L/K$ (then both groups are of
this form),}

\vskip2mm

\noindent the groups $G_1$ and $G_2$ have equivalent systems of
maximal $K$-tori.
\end{prop}
\begin{proof}
(1): We have $G_i = \mathrm{SL}_{1 , A_i}$ where $A_i$ is a central
simple algebra over $K$ of dimension $(n+1)^2,$ and as in the proof
of Lemma \ref{L:local1}, it is enough to show that a commutative  \'etale
$(n+1)$-dimensional $K$-algebra $E$ embeds in $A_1$ if and only if
it embeds in $A_2.$ For $v \in V^K,$ we can write
$$
A_i \otimes_K K_v = M_{d_i^{(v)}}(\Delta_i^{(v)})
$$
where $\Delta_i^{(v)}$ is a central division algebra over $K_v,$ of
degree $m_i^{(v)}.$ As in the proof of Lemma \ref{L:local1}, we
conclude that (\ref{E:rank}) implies $m_1^{(v)} = m_2^{(v)}.$ On the
other hand, it is well-known (cf.\,\cite{PR7}, Propositions 2.6 and
2.7) that an $(n+1)$-dimensional commutative  \'etale $K$-algebra $E = \prod_{j =
1}^s E^{(j)},$ where $E^{(j)}/K$ is a finite (separable) field
extension, embeds in $A_i$ if and only if for each $j \leqslant s$
and all $v \in V^K,$ the local degree $[E^{(j)}_w : K_v]$ is
divisible by $m_i^{(v)}$ for all extensions $w \vert v,$ and the
required fact follows.

\vskip2mm

(2): We have $G_i = \mathrm{SU}(D_i , \tau_i)$, where $D_i$ is a
central simple algebra of degree $m = n + 1$ over $L$ with an
involution $\tau_i$ such that $\tau_i \vert L = \sigma.$ Assume that
$D_1$ is a division algebra. Then it follows from the
Albert-Hasse-Brauer-Noether Theorem that $m = \mathrm{lcm}_{w \in
V^L}(m_1^{(w)})$, where for $w \in V^L$, $D_i \otimes_L L_w =
M_{d_i^{(w)}}(\Delta_i^{(w)})$ with $\Delta_i^{(w)}$ a central
division algebra over $L_w$ of degree $m_i^{(w)}.$ For $j = 1, 2,$
set
$$
V_j^L = \{ w \in V^L \: \vert \: [L_w : K_v] = j \ \ \text{where} \
\ w \vert v \}.
$$
It is well-known that $m_i^{(w)} = 1$ for $w \in V_2^L,$ so
$$
m = \mathrm{lcm}_{w \in V_1^L}(m_1^{(w)}).
$$
On the other hand, for $w \in V_1^L$ we have $G_i \simeq
\mathrm{SL}_{d_i^{(w)} , \Delta_i^{(w)}}$ over $K_v = L_w,$ hence
$\mathrm{rk}_{K_v}\, G_i = d_i^{(w)} - 1.$ Thus, (\ref{E:rank})
implies that $m_1^{(w)} = m_2^{(w)}$ for all $w \in V_1^L,$ and
therefore
$$
m = \mathrm{lcm}_{w \in V_1^L}(m_2^{(w)}).
$$
It follows that $D_2$ is a division algebra, as required.

Next, since any maximal $K$-torus of $G_i$ is of the form
$\mathrm{R}_{E/K}(\mathrm{GL}_1) \cap G_i$ for some $m$-dimensional commutative 
\'etale $L$-algebra invariant under $\tau_i$ (cf.\,\cite{PR7},
Proposition 2.3), it is enough to show that  for an $m$-dimensional
commutative \'etale $L$-algebra $E$ with an involutive automorphism $\tau$ such
that $\tau \vert L = \sigma,$ the existence of an embedding $\iota_1
\colon (E , \tau) \hookrightarrow (D_1 , \tau_1)$ as $L$-algebras
with involutions is equivalent to the existence of an embedding
$\iota_2 \colon (E , \tau) \hookrightarrow (D_2 , \tau_2).$ Since $D_1$ is a division algebra, the existence of
$\iota_1$ implies that $E/L$ is a field extension, and then by
Theorem 4.1 of \cite{PR7}, the existence of $\iota_2$ is equivalent
to the existence of an $(L \otimes_K K_v)$-embedding
$$
\iota_2^{(v)} \colon (E \otimes_K K_v , \tau \otimes
\mathrm{id}_{K_v}) \hookrightarrow (D_2 \otimes_K K_v , \tau_2
\otimes \mathrm{id}_{K_v})
$$
for all $v \in V^K.$ If $v \in V^K$ has two extensions $w' , w'' \in
V_1^L$, then $m_i^{(w')} = m_i^{(w'')} =: m_i^{(v)}$ and the
necessary and sufficient condition for the existence of
$\iota_i^{(v)}$ is that for any extension $u$ of $v$ to $E$, 
the local degree $[E_u : K_v]$ is divisible by $m_i^{(v)}$
(cf. Proposition A.3 in \cite{PR1}). Therefore, since $m_1^{(v)} = m_2^{(v)},$
the existence of $\iota_1^{(v)}$ implies that of $\iota_2^{(v)}.$ If
$v$ has only one extension $w$ to $L$, then $w \in V_2^L$ and $$(D_i \otimes_K K_v ,
\tau_i  \otimes \mathrm{id}_{K_v}) \simeq (M_m(L_w) , \theta_i)$$
with $\theta_i$  given by $\theta((x_{st})) =
a_i^{-1}(\overline{x}_{ts})a_i$ where $x \mapsto \overline{x}$
denotes the nontrivial automorphism of $L_w/K_v$ and $a_i$ is a
Hermitian matrix. Furthermore, $\mathrm{rk}_{K_v}\, G_i$ equals the
Witt index $i(h_i)$ of the Hermitian form $h_i$ with matrix $a_i.$
Then (\ref{E:rank}) yields that $i(h_1) = i(h_2)$ which as we have
seen in the proof of Lemma \ref{L:local1}(i) implies that $h_1$ and
$h_2$ are similar.  Hence, 
$$
(D_1 \otimes_K K_v , \tau_1  \otimes \mathrm{id}_{K_v}) \simeq (D_2
\otimes_K K_v , \tau_2  \otimes \mathrm{id}_{K_v}),
$$
and therefore again the existence of $\iota_1^{(v)}$ implies the
existence of $\iota_2^{(v)}.$

Finally,  since $D_2$ is also a division algebra, we can use the
above argument to conclude that $(D_1 , \tau_1)$ and $(D_2 ,
\tau_2)$ in fact have the same $m$-dimensional commutative  \'etale
$L$-subalgebras invariant under the involutions as claimed.
\end{proof}

\vskip2mm

\noindent {\bf Remark 6.4.} (1) We have already noted prior to
Proposition \ref{P:An} that the assumption (2) of Theorem
\ref{T:equiv-tori} is not needed in its statement. So, it is worth
mentioning that assumption (1) in this situation is in fact 
satisfied automatically: for groups of outer type $A_n$ this follows
from Corollary 4.5 in \cite{PR7}, while for groups of inner type
$A_n$ it is much simpler, viz. in the notation used in the proof of
Proposition \ref{P:An}(1), one shows that the algebras $A_1$ and
$A_2$ contain a common field extension of $K$ of degree $(n+1)$.
This can also be established for groups of type $D_n$ with $n$ odd
using Proposition A of \cite{PR7}.

(2) We would like to clarify that assumption (2) of Theorem
\ref{T:equiv-tori} is only needed to conclude that {\brus
SH}$^2(T_1)$ is trivial for any maximal $K$-torus $T_1$ of $G_1.$
However, this fact holds for any maximal $K$-torus in a connected absolutely 
almost simple simply connected algebraic $K$-group of inner type $A_n$ {\it
unconditionally,} cf.\:Remark 9.13 in \cite{PR6}. So, the proof of
Theorem \ref{T:equiv-tori} actually yields part (1) of Proposition
\ref{P:An}.

\vskip2mm

\addtocounter{thm}{1}

\begin{cor}\label{C:Ap}
Let $G_1$ and $G_2$ be two connected absolutely almost simple simply connected
algebraic groups of type $A_{p-1},$ where $p$ is a prime, over a
number field $K.$ Assume that (\ref{E:rank}) holds and that $L_1 =
L_2 =: L.$ Then $G_1$ and $G_2$ have equivalent systems of maximal
$K$-tori.
\end{cor}

Indeed, if $G_1$ and $G_2$ are inner forms (in particular, if $p =
2$) then our assertion immediately follows from Proposition
\ref{P:An}(1). Furthermore, if one of the groups is of the form
$\mathrm{SU}(D , \tau)$ where $D$ is a central \emph{division}
algebra over $L$ of degree $p$ then we can use Proposition
\ref{P:An}(2). It remains to consider the case where $G_i =
\mathrm{SU}(L , h_i)$ with $h_i$ a nondegenerate hermitian form on
$L^p$ for $i = 1, 2.$ Then the proof of Lemma \ref{L:local1}(i)
shows that $h_1$ and $h_2$ are similar over $L_w$ for all $w \in
V_2^L.$ But then $h_1$ and $h_2$ are similar, i.e.,  $G_1 \simeq G_2$
over $K$ and there is nothing to prove.

\vskip2mm

Here is a companion to Theorem \ref{T:ArG2} for groups of types $A,$
$D$ and $E_6.$
\begin{thm}\label{T:ADE6}
Let $G_1$ and $G_2$ be two connected absolutely almost simple algebraic groups
of the same Killing-Cartan type which is one of the following:
$A_n$, $D_{2n+1}$ $(n > 1)$ or $E_6$ defined over a field $F$ of
characteristic zero, and let $\Gamma_i \subset G_i(F)$ be a
Zariski-dense $(\cG_i, K_i, S_i)$-arithmetic subgroup. Assume that
for at least one $i \in \{1 , 2\}$ there exists $v^{(i)}_0 \in
V^{K_i}$ such that $\cG_i$ is anisotropic over
${(K_i)}_{v^{(i)}_0}$. Then either condition $(C_i)$ holds for some
$i \in \{1 , 2\}$, or $K_1 = K_2 =: K$ and the groups $\cG_1$ and
$\cG_2$ have equivalent systems of maximal $K$-tori.
\end{thm}

(We note that if $\cG_1$ and $\cG_2$ have equivalent systems of
maximal $K$-tori then $(C_i)$ can hold only if $S_1 \neq S_2$.)

\begin{proof}
We can obviously assume that for $i = 1, 2,$ the group $G_i$ is
adjoint and $\Gamma_i \subset \cG_i(K_i)$. According to Theorem
\ref{T:ArG1}, if neither $(C_1)$ nor $(C_2)$ hold, then we
have
$$
K_1 = K_2 =: K, \ \ \ L_1 = L_2 =:L, \ \ \ S_1 = S_2 =: S
$$
and
$$
 \mathrm{rk}_{K_v}\,\cG_1 = \mathrm{rk}_{K_v}\,\cG_2 \ \
\text{for all} \ \ v \in V^K.
$$
Furthermore, there exists $m \geqslant 1$ such that any $m$
multiplicatively independent semi-simple elements $\gamma_1, \ldots
, \gamma_m \in \Gamma_1$ are necessarily weakly contained in
$\Gamma_2$. Arguing as in the proof of Theorem \ref{T:ArG1}, we can
find $m$ multiplicatively independent elements $\gamma_1, \ldots ,
\gamma_m \in \Gamma_1$ so that the corresponding tori $T_i =
Z_{\cG_1}(\gamma_i)^{\circ}$ satisfy the following:

\vskip2mm

\noindent $\bullet$ $\theta_{T_i}(\Ga(L_{T_i}/L)) = W(\cG_1 , T_i)$;

\vskip1mm

\noindent $\bullet$ $\mathrm{rk}_{K_v}\, T_i = \mathrm{rk}_{K_v}\,
\cG_1$ for all $v \in S$.

\vskip2mm

\noindent Then the fact that $\gamma_1, \ldots , \gamma_m$ are
weakly contained in $\Gamma_2$ would imply that there exists a maximal $K$-torus $T_2^0$ of  $\cG_2$ and an $i
\leqslant m$ such that there is a $K$-isogeny $T^0_2
\to T^0_1 :=T_i$. Since the common type
of $\cG_1$ and $\cG_2$ is different from $B_2 = C_2$, $F_4$ and
$G_2$, it follows from Lemma 4.3 and Remark 4.4 in \cite{PR6} that
one can scale the isogeny so that it induces an isomorphism between
the root systems $\Phi(\cG_1 , T^0_1)$ and $\Phi(\cG_2 , T^0_2)$,
and therefore extends to a $\overline{K}$-isomorphism $\cG_1 \to
\cG_2$ as these groups are adjoint. Passing to the simply connected
groups $\widetilde{\cG}_1$ and $\widetilde{\cG}_2$ and the corresponding
tori $\widetilde{T}^0_1$ and $\widetilde{T}^0_2$, we see that there exists a
$K$-isomorphism $\widetilde{T}^0_1 \to \widetilde{T}^0_2$ that extends to a
$\overline{K}$-isomorphism $\widetilde{\cG}_1 \to \widetilde{\cG}_2$. Note
that by our construction we have $\mathrm{rk}_{K_v}\,T^0_i =
\mathrm{rk}_{K_v}\,\cG_i$ for $i = 1, 2$ and all real places $v$ of
$K$. In view of our assumptions, we can invoke Theorem
\ref{T:equiv-tori} to conclude that $\widetilde{\cG}_1$ and
$\widetilde{\cG}_2$ have equivalent systems of maximal $K$-tori, and
then the same remains true for $\cG_1$ and $\cG_2$.
\end{proof}

\vskip3mm

It follows from Proposition \ref{P:An} and Corollary \ref{C:Ap} that
the assertion of Theorem \ref{T:ADE6} remains valid without the
assumption that there be $v^{(i)}_0 \in V^{K_i}$ such that $\cG_i$
is ${(K_i)}_{v^{(i)}_0}$-anisotropic for groups of type $A_n$ in the
following three situations: (1) one of the $\cG_i$'s is an inner
form; (2) the simply connected cover of one of the $\cG_i$'s is isomorphic to  $\mathrm{SU}(D ,
\tau)$ where $D$ is a central {\it division} algebra over $L$ with
an involution $\tau$ of the second kind that restricts to the
nontrivial automorphism of $L/K$; (3) $n = p-1$ where $p$ is a
prime.

\section{Fields generated by the lengths of closed
geodesics}\label{S:FGeod}

Let $G$ be an absolutely simple adjoint 
algebraic $\R$-group such that $\mathcal{G} := G(\R)$ is noncompact. Pick
a maximal compact subgroup $\mathcal{K}$ of $\mathcal{G},$ and let
$\fX = \mathcal{K} \backslash \mathcal{G}$ denote the corresponding
symmetric space considered as a Riemannian manifold with the metric
induced by the Killing form. Given a discrete torsion-free subgroup
$\Gamma \subset \mathcal{G},$ we consider the associated locally
symmetric space $\fX_{\Gamma} := \fX/\Gamma.$ It was shown in
\cite{PR6}, 8.4, that every (nontrivial) semisimple element $\gamma
\in \Gamma$ gives rise to a closed geodesic $c_{\gamma}$ in
$\fX_{\Gamma},$ and conversely, every closed geodesic can be
obtained that way. Moreover, the length $\ell(c_{\gamma})$ can be
written in the form $(1/n_{\gamma}) \cdot \lambda_{\Gamma}(\gamma)$
where $n_{\gamma} \geqslant 1$ is an integer and
\begin{equation}\label{E:LG-1}
\lambda_{\Gamma}(\gamma) = \left(\sum_{\alpha} (\log \vert
\alpha(\gamma) \vert)^2 \right)^{1/2}
\end{equation}
where the summation is over all roots $\alpha$ of $G$ with respect
to an arbitrary maximal $\R$-torus $T$ containing $\gamma$ (Proposition 8.5 of \cite{PR6}). In
particular, for the set $L(\fX_{\Gamma})$  of lengths of all closed
geodesics in $\fX_{\Gamma}$ we have
$$
\Q \cdot L(\fX_{\Gamma}) = \Q \cdot \{\lambda_{\Gamma}(\gamma) \:
\vert \: \gamma \in \Gamma \ \text{nontrivial semisimple} \},
$$
and the subfield of $\R$ generated by $L(\fX_{\Gamma})$ coincides
with the subfield generated by the values $\lambda_{\Gamma}(\gamma)$
for all semisimple $\gamma \in \Gamma.$

\vskip2mm

Now, let $G_1$ and $G_2$ be two absolutely simple adjoint algebraic
$\R$-groups such that  the group  $\mathcal{G}_i :=G_i(\R)$ is
noncompact for both $i = 1, 2.$ For each $i \in \{1 , 2\}$, we pick
a maximal compact subgroup $\mathcal{K}_i$ of $\mathcal{G}_i :=
G_i(\R)$ and consider the symmetric space $\fX_i = \mathcal{K}_i
\backslash \mathcal{G}_i.$ Furthermore, given a discrete
torsion-free Zariski-dense subgroup $\Gamma_i$ of $\mathcal{G}_i,$
we let $\fX_{\Gamma_i} := \fX_i / \Gamma_i$ denote the associated
locally symmetric space.  As above, for $i = 1, 2$, we let $w_i$
denote the order of the Weyl group of $G_i$ with respect to a
maximal torus, and let $K_{\Gamma_i}$ be the field of definition of
$\Gamma_i,$ i.e. the subfield of $\R$ generated by the traces
$\mathrm{Tr}\: \mathrm{Ad}\: \gamma$ for $\gamma \in \Gamma_i$. In
this section, we will  focus our attention on the fields
$\cF_i$ generated by the set $L(\fX_{\Gamma_i}),$ for
$i = 1, 2.$

\vskip1mm

The results of this section depend on the truth of Schanuel's
conjecture from transcendental number theory (hence they are {\it
conditional}). For the reader's convenience we recall its
statement (cf.\:\cite{A}, \:\cite{Ba}, p.\,120).

\vskip2mm

\noindent {\bf Schanuel's conjecture.} {\it If $z_1, \ldots , z_n
\in \C$ are linearly independent over $\Q,$ then the transcendence
degree (over $\Q$) of the field generated by
$$
z_1, \ldots , z_n; \  e^{z_1}, \ldots , e^{z_n}
$$
is $\geqslant n.$}

\vskip3mm

Assuming Schanuel's conjecture and developing the techniques of
\cite{PR5}, we prove the following proposition which enables us to
connect the results of the previous sections to some geometric
problems involving the sets $L(\fX_{\Gamma_i})$ and the fields
$\cF_i.$

\vskip1mm

\begin{prop}\label{P:LG1}
Let $\mathscr{K} \subset \R$ be a subfield of finite
transcendence degree $d$ over $\Q,$  let $G_1$ and $G_2$ be
semisimple $\mathscr{K}$-groups, and for $i \in \{1 , 2\}$, let
$\Gamma_i \subset G_i(\mathscr{K}) \subset G_i(\R)$ be a discrete
Zariski-dense torsion-free subgroup. As above, for $i =1, 2$, let $\mathscr{F}_i$ be the
subfield of $\R$ generated by the 
$\lambda_{\Gamma_i}(\gamma)$ for all nontrivial semi-simple $\gamma
\in \Gamma_i$, where $\lambda_{\Gamma_i}(\gamma)$ is given by
equation (\ref{E:LG-1}) for $G=G_i.$ If
nontrivial semisimple elements $\gamma_1, \ldots , \gamma_m \in
\Gamma_1$ are multiplicatively independent and are not weakly
contained in $\Gamma_2$, then the transcendence degree of
$\mathscr{F}_2(\lambda_{\Gamma_i}(\gamma_1), \ldots ,
\lambda_{\Gamma_i}(\gamma_m))$ over $\mathscr{F}_2$ is $\geqslant m
- d.$
\end{prop}
\begin{proof}
We can assume that $m > d$ as otherwise there is nothing to prove.
It was shown in \cite{PR6} (see the remark after Proposition 8.5)
that for $i = 1, 2$ and any nontrivial semisimple element $\gamma
\in \Gamma_i$, the value $\lambda_{\Gamma_i}(\gamma)^2$, where
$\lambda_{\Gamma_i}(\gamma)$ is provided by (\ref{E:LG-1}), can be
written in the form
\begin{equation}\label{E:LG-2}
\lambda_{\Gamma_i}(\gamma)^2 = \sum_{k = 1}^p s_k \left(\log
\chi_k(\gamma) \right)^2,
\end{equation}
where $\chi_1, \ldots , \chi_p$ are some {\it positive} characters
of a maximal $\R$-torus $T$ of $G_i$ containing $\gamma,$ and $s_1,
\ldots , s_p$ are some positive rational numbers. Furthermore, we
note that if $\gamma \in \Gamma_i$ is a semisimple element $\neq 1$
and $T$ is a maximal $\R$-torus of $G_i$ containing $\gamma$ then
the condition $\vert \alpha(\gamma) \vert = 1$ for all roots
$\alpha$ of $G_i$ with respect to $T$ would imply that the
nontrivial subgroup $\langle \gamma \rangle$ is discrete and
relatively compact, hence finite. This is impossible as $\Gamma_i$
is torsion-free, so we conclude from (\ref{E:LG-1}) that
$\lambda_{\Gamma_i}(\gamma) > 0$ for any nontrivial $\gamma \in
\Gamma_i.$ Thus, assuming that $\gamma \in \Gamma_i$ is nontrivial
and renumbering the characters in (\ref{E:LG-2}), we can arrange so
that
$$
a_{\gamma,1} = \log \chi_1(\gamma), \ldots , a_{\gamma,d_\gamma} = \log \chi_{d_\gamma}(\gamma) \ \
\text{with} \ d_{\gamma} \geqslant 1,
$$
form a basis of the $\Q$-vector subspace of $\R$ spanned by $\log
\chi_1(\gamma), \ldots , \log \chi_p(\gamma).$ Then we can write
$\lambda_{\Gamma_i}(\gamma)^2 = q_{\gamma}(a_{\gamma,1}, \ldots ,
a_{\gamma,d_\gamma})$ where $q_{\gamma}(t_1, \ldots , t_{d_\gamma})$ is a
nontrivial rational quadratic form. Thus, for any nontrivial
semisimple $\gamma \in \Gamma_i$ there exists a finite set
$A_{\gamma} = \{a_{\gamma,1}, \ldots , a_{\gamma,d_{\gamma}} \},$ with $d_{\gamma}
\geqslant 1,$ of real numbers linearly independent over $\Q,$ each
of which is the logarithm of the value of a positive character on
$\gamma,$ such that
$$\lambda_{\Gamma_i}(\gamma)^2 = q_{\gamma}(a_{\gamma,1}, \ldots ,
a_{\gamma, d_{\gamma}}),$$ where $q_{\gamma}(t_1, \ldots , t_{d_{\gamma}})$
is a nonzero rational quadratic form. We fix such $A_{\gamma}$ and
$q_{\gamma}$ for each nontrivial semi-simple $\gamma \in \Gamma_i$,
where  $i = 1, 2$, for the remainder of the argument. Let
$\mathscr{M}_i$ be the subfield of $\R$ generated by the values
$\lambda_{\Gamma_i}(\gamma)^2 = q_{\gamma}(a_{\gamma,1}, \ldots ,
a_{\gamma,d_{\gamma}})$ for all nontrivial semisimple $\gamma \in
\Gamma_i$.

\vskip2mm

Now, suppose $\gamma_1, \ldots , \gamma_m \in \Gamma_1$ are as in
the statement of the proposition. It is enough to show that for any {\it
finitely generated} subfield $\mathscr{M}'_2 \subset \mathscr{M}_2$, 
we have
$$\mathrm{tr.\: deg}_{\mathscr{M}'_2}\:
\mathscr{M}'_2(\lambda_{\Gamma_i}(\gamma_1)^2, \ldots ,
\lambda_{\Gamma_i}(\gamma_m)^2) \geqslant m - d.$$ Indeed, this
would imply that $ \mathrm{tr.\: deg}_{\mathscr{M}_2}\:
\mathscr{M}_2(\lambda_{\Gamma_i}(\gamma_1)^2, \ldots ,
\lambda_{\Gamma_i}(\gamma_m)^2)$, and hence (as
$\mathscr{F}_2/\mathscr{M}_2$ is algebraic) $ \mathrm{tr.\:
deg}_{\mathscr{F}_2}\: \mathscr{F}_2(\lambda_{\Gamma_i}(\gamma_1)^2,
\ldots , \lambda_{\Gamma_i}(\gamma_m)^2)$ is $\geqslant m - d$,
yielding the proposition. We now note that any finitely generated
subfield $\mathscr{M}'_2 \subset \mathscr{M}_2$ is contained in a
subfield of the form $\mathscr{P}_{\Theta_2}$ for some finite set
$\Theta_2 = \{\gamma_1^{(2)}, \ldots , \gamma_{m_2}^{(2)}\}$ of
nontrivial semisimple elements of $\Gamma_2,$ which by definition is
generated by $\bigcup_{k = 1}^{m_2} A_{\gamma_k^{(2)}}.$ So, it is
enough to prove that if $\gamma_1, \ldots , \gamma_m \in \Gamma_1$
are as in the statement of the proposition then for any finite set
$\Theta_2$ of nontrivial semi-simple elements of $\Gamma_2$ we
have
\begin{equation}\label{E:TD-1}
\mathrm{tr.\: deg}_{\mathscr{P}_{\Theta_2}}\,
\mathscr{P}_{\Theta_2}(\lambda_{\Gamma_i}(\gamma_1), \ldots ,
\lambda_{\Gamma_i}(\gamma_m)) \geqslant m - d.
\end{equation}

Since the elements $\gamma_1, \ldots , \gamma_m$ are
multiplicatively independent, the elements of
$$
A = \bigcup_{j = 1}^m A_{\gamma_j}
$$
are linearly independent (over $\Q$). Let $B$ be a maximal linearly
independent (over $\Q$) subset of $\bigcup_{k = 1}^{m_2}
A_{\gamma_{k}^{(2)}}$. Since $\gamma_1, \ldots , \gamma_m$ are not
weakly contained in $\Gamma_2$, the elements of $A \cup B$ are
linearly independent over $\Q.$ Let $\alpha = \vert A \vert$ and
$\beta = \vert B \vert.$ Then by Schanuel's conjecture, the
transcendence degree over $\Q$ of the field generated by
$$
A \cup B \cup \widetilde{A} \cup \widetilde{B}, \ \text{where} \ \
\widetilde{A} = \{ e^s \:\vert\: s \in A\} \ \ \text{and} \ \ \widetilde{B}
= \{ e^s \:\vert\: s \in B \},
$$
is $\geqslant \alpha + \beta.$ But the set $\widetilde{A} \cup
\widetilde{B}$ consists of the values of certain characters on certain
semi-simple elements lying in  $\Gamma_i \subset G_i(\mathscr{K})$, 
and therefore is contained in
$\overline{\mathscr{K}}.$ It follows that the transcendence degree
over $\Q$ of the field generated by $\widetilde{A} \cup \widetilde{B}$ is
$\leqslant d,$ and therefore the transcendence degree of the field
generated by $A \cup B$ is $\geqslant \alpha + \beta - d.$ So,
$$
\mathrm{tr.\: deg}_{\Q(B)}\, \Q(A \cup B) = \mathrm{tr.\:
deg}_{\Q}\,\Q(A \cup B) - \mathrm{tr.\: deg}_{\Q}\,\Q(B) \geqslant
$$
$$
\geqslant (\alpha + \beta - d) - \beta = \alpha - d.
$$
Thus, there exists a subset $C \subset A$ of cardinality $\leqslant
d$ such that the elements of $A \setminus C$ are algebraically
independent over $\Q(B)$. Since $C$ intersects at most $d$ of the
sets $A_{\gamma_j}$, $j\leqslant m$,  we see that after renumbering,
we can assume that the elements of
$$
D = \bigcup_{j = 1}^{m-d} A_{\gamma_j}
$$
are algebraically independent over $\Q(B)$.  Since $\Q(B)$ coincides with $\mathscr{P}_{\Theta_2}$, (\ref{E:TD-1}) follows from the following simple lemma.\end{proof}

\begin{lemma}\label{L:LG1}
Let $F$ be a field, and let $E = F(t_1, \ldots , t_n)$, where $t_1,
\ldots , t_n$ are algebraically independent over $F.$ Let
$$
\{1, 2, \ldots , n\} = I_1 \cup \cdots \cup I_s
$$
be an arbitrary partition, and let $E_j$ be the field generated over
$F$ by the $t_i$ for $i \in I_j.$ For each $j \in \{1, \ldots ,
s\},$ pick $f_j \in E_j \setminus F.$ Then
$$
\mathrm{tr.\: deg}_F\,F(f_1, \ldots , f_s) = s.
$$
\end{lemma}

\vskip2mm

Now if property $(C_i)$ holds for $i=1$ or $2$, then  Proposition
\ref{P:LG1} implies the following at once.

\begin{cor}\label{C:LG1}
Notations and assumptions are as in Proposition \ref{P:LG1}, assume that
condition $(C_i)$ holds for either $i=1$ or \:$2$. Then the
transcendence degree of $\mathscr{F}_1  \mathscr{F}_2$ over
$\mathscr{F}_{3-i}$  is infinite, i.e. condition $(T_i)$ (of the
introduction) holds.
\end{cor}

%Indeed, condition $(C_i),$ which was described at the beginning  of \S \ref{S:FDef},
%states that for any $m \geqslant 1$ one can find semisimple elements
%$\gamma_1, \ldots , \gamma_m \in \Gamma_i$ that are multiplicatively
%independent and are not weakly contained in $\Gamma_1 \times \cdots
%\times \Gamma_{i-1} \times \Gamma_{i+1} \times \cdots \times
%\Gamma_r.$ So, it follows from the proposition that
%$$
%\mathrm{tr.\: deg.}_{\mathscr{L}} \mathscr{L}_1 \cdots \mathscr{L}_r
%\geqslant \mathrm{tr.\: deg.}_{\mathscr{L}}
%\mathscr{L}(\lambda_{\Gamma_i}(\gamma_1), \ldots ,
%\lambda_{\Gamma_i}(\gamma_m)) \geqslant m - d
%$$
%for any $m,$ and our assertion follows.
%
%
%
%\vskip3mm
%
%The assertion of the corollary motivates the following property:
%
%\vskip2mm
%
%\noindent $(T_i)$ \: \parbox[t]{10.5cm}{$\mathscr{L}_1 \cdots
%\mathscr{L}_r$ has infinite transcendence degree over $\mathscr{L}_1
%\cdots \mathscr{L}_{i-1}\mathscr{L}_{i+1} \cdots \mathscr{L}_r.$}
%
%
%\vskip3mm
%
%\noindent Thus, Corollary \ref{C:LG1} asserts that under the
%assumptions made in the statement of Proposition \ref{P:LG1}, $(C_i)
%\ \Rightarrow \ (T_i).$
%
%\vskip2mm

%Before formulating our next result, we recall the definition of the
%set $\bJ(i_0)$ from \S \ref{S:FDef}: we pick $i_0 \in \bJ$ such that
%the field $K_{\Gamma_{i_0}}$ is maximal among the fields
%$K_{\Gamma_i}$ for $i \in \bJ,$ and then set $\bJ(i_0) = \{ i \in
%\bJ \: \vert \: K_{\Gamma_i} = K_{\Gamma_{i_0}}\}.$

Combining the corollary with Theorem \ref{T:F1}, we obtain Theorem 1
(of the introduction). This theorem has the following important
consequence. In \cite{PR6}, \S8, we had to single out the following
exceptional case

\vskip2mm

\noindent $(\mathcal{E})$ \parbox[t]{13cm}{One of the locally
symmetric spaces, say, $\fX_{\Gamma_1}$, is 2-dimensional and the
corresponding discrete subgroup $\Gamma_1$ cannot be conjugated into
$\mathrm{PGL}_2(K)$, for any number field $K \subset \R$, and the
other space, $\fX_{\Gamma_2}$, has dimension $> 2$,}

\vskip2mm

\noindent which was then excluded in some of our results. Theorem
1(1) shows that the locally symmetric spaces as in $(\mathcal{E})$
can {\it never} be length-commensurable (assuming Schanuel's
conjecture), and therefore all our results are in fact valid without
the exclusion of case $(\mathcal{E})$.

As we mentioned in the introduction, much more precise results are
available when the groups $\Gamma_1$ and $\Gamma_2$ are arithmetic.
In this section we will prove Theorems 2 and 3 that treat the case
where $G_1$ and $G_2$ are of the same Cartan-Killing type, and
postpone the proof of Theorem 4, where one of the group is of type
$B_n$ and the other is of type $C_n$ for some $n \geqslant 3$, until
the next section. In fact, Theorem 2 follows immediately from
Corollary \ref{C:LG1} and Theorem \ref{T:ArG2}. It should be noted
that while Theorem 2 asserts that conditions $(T_i)$ and $(N_i)$
hold for {\it at least one}  $i \in \{1 , 2\}$, these may not
hold for { both} $i$.

\vskip2mm

\addtocounter{thm}{1}

\noindent {\bf Example 7.4.} Let $D_1$ and $D_2$ be the quaternion
algebras over $\Q$ with the sets of ramified places $\{ 2 , 3 \}$
and $\{2, 3, 5, 7\}$, respectively. Set $G_i = \mathrm{PSL}_{1 ,
D_i}$, and let $\Gamma_i$ be a torsion-free subgroup of ${G_i}(\Q)$, for $i = 1, 2$. 
Over $\R$, both $G_1$ and
$G_2$ are isomorphic to $G = \mathrm{PSL}_2$, so $\Gamma_1$ and
$\Gamma_2$ can be viewed as arithmetic subgroups of $\mathcal{G} =
G(\R)$. The symmetric space $\fX$ associated with $\mathcal{G}$ is
the hyperbolic plane $\mathbb{H}^2$, so the corresponding locally
symmetric spaces $\fX_{\Gamma_1}$ and $\fX_{\Gamma_2}$ are
arithmetically defined hyperbolic 2-manifolds that are not
commensurable as the groups $G_1$ and $G_2$ are not
$\Q$-isomorphic. At the same time, our choice of $D_1$ and $D_2$
implies that every maximal subfield of $D_2$ is isomorphic to a
maximal subfield of $D_1$ which entails that $\Q \cdot
L(\fX_{\Gamma_2}) \subset \Q \cdot L(\fX_{\Gamma_1})$, hence $\cF_2
\subset \cF_1$. Thus, $\cF_1\cF_2 = \cF_1$, so $(T_1)$ does not hold 
(although $(T_2)$ does hold).
\vskip2mm

Next, we will derive Theorem 3 from Theorem \ref{T:ADE6}. Let
$\Gamma_i$ be $(\cG_i, K_i)$-arithmetic. Assume that $(T_i)$, hence
$(C_i)$, does not hold for either $i = 1$ or \,$2$. Then by
Theorem \ref{T:ADE6} we necessarily have $K_1 = K_2 =: K$, and the groups $\cG_1$, $\cG_2$ have equivalent systems of maximal $K$-tori. By
the assumption made in Theorem 3, $K \neq \Q$. The field $K$ has the
real place associated with the identity embedding $K \hookrightarrow
\R$ but since $K \neq \Q$, it necessarily has another archimedean
place $v_0$, and the discreteness of $\Gamma_i$ implies that $\cG_i$
is $K_{v_0}$-anisotropic. Thus, Theorem \ref{T:ADE6} applies to the
effect that the groups $\cG_1$ and $\cG_2$ have equivalent systems
of maximal $K$-tori. Then the fact that $\Q \cdot L(\fX_{\Gamma_1})
= \Q \cdot L(\fX_{\Gamma_2})$ follows from the following.
\begin{prop}\label{P:length-commen}
{\rm (cf.\,\cite{PR6}, Proposition 9.14)} Let $G_1$ and $G_2$ be
connected absolutely simple algebraic groups such that
$\mathcal{G}_i = G_i(\R)$ is noncompact for both $i = 1, 2$, and let
$\fX_i$ be the symmetric space associated with $\mathcal{G}_i$.
Furthermore, let $\Gamma_i \subset \mathcal{G}_i$ be a discrete
torsion-free $(\cG_i , K)$-arithmetic subgroup (where $K \subset \R$
is a number field), and $\fX_{\Gamma_i} = \fX/\Gamma_i$ be the
corresponding locally symmetric space for $i = 1, 2$. If $\cG_1$ and
$\cG_2$ have equivalent systems of maximal $K$-tori, then
$\fX_{\Gamma_1}$ and $\fX_{\Gamma_2}$ are length-commensurable.
\end{prop}

This is essentially Proposition 9.14 of \cite{PR6} except that here
we require that the groups $\cG_1$ and $\cG_2$ have equivalent
systems of maximal $K$-tori instead of the more technical
requirement of having {\it coherently} equivalent systems of maximal
$K$-tori used in \cite{PR6}; this change however does not affect the
proof.

\vskip2mm

The analysis of our argument in conjunction with Proposition
\ref{P:An} and Corollary \ref{C:Ap} shows that the assertion of
Theorem 3 remains valid without the assumption that $K_{\Gamma_i}
\neq \Q$ at least in the following situations where $\cG_1$ and
$\cG_2$ are of type $A_n$: (1) one of the $\cG_i$'s is an inner
form; (2) one of the $\cG_i$'s is represented by $\mathrm{SU}(D ,
\tau)$ where $D$ is a central {\it division} algebra over $L$ with
an involution $\tau$ of the second kind that restricts to the
nontrivial automorphism of $L/K$; (3) $n = p-1$, where $p$ is a
prime.

\vskip2mm

To illustrate our general results in a concrete geometric situation,
we will now prove Corollary 1 of the introduction. The hyperbolic
$d$-space $\mathbb{H}^d$ is the symmetric space of the group $G(d) =
{\mathrm{PSO}}(d , 1)$. For $d \geqslant 2$, set $\displaystyle \ell =
\left[\frac{d + 1}{2} \right]$. Then for $d \neq 3$, $G(d)$ is an
absolutely simple group of type $B_{\ell}$ if $d$ is even, and of
type $D_{\ell}$ if $d$ is odd. Furthermore, the order $w(d)$ of the
Weyl group of $G(d)$ is given by:
$$
w(d) = \left\{ \begin{array}{ccl} 2^{\ell} \cdot\ell! & , & d \ \
\ \text{is even}, \\ 2^{\ell-1} \cdot\ell! & , & d \ \ \ \text{is
odd.}
\end{array} \right.
$$
One easily checks that $w(d) < w(d + 1)$ for any $d \geqslant 2$,
implying that $w(d_1) > w(d_2)$ whenever $d_1 > d_2$. With these
remarks, assertions $(i)$ and $(ii)$ follow from Theorem 1.
Furthermore, using the above description of the Killing-Cartan type
of $G(d)$ one easily derives assertions $(iii)$ and $(iv)$ from Theorems
2 and 3, respectively.

\vskip2mm

It follows from (\cite{PlR}, Theorem 5.7) that given a discrete
torsion-free $(\cG_i , K_i)$-arithmetic subgroup of $\mathcal{G}_i$,
the compactness of the locally symmetric space $\fX_{\Gamma_i}$ is
equivalent to the fact that $\cG_i$ is $K_i$-anisotropic. Combining
this with Theorem \ref{T:isotr}, we obtain Theorem 5.

%We would like to conclude by deriving two consequences for pairs of
%(arithmetically defined) locally symmetric spaces. The first one
%states that the length spectra of compact and noncompact spaces are
%very much different.
%\begin{thm}\label{T:compact}
%Let $\fX_{\Gamma_1}$ and $\fX_{\Gamma_2}$ be arithmetically defined
%locally symmetric spaces such that one of them is compact and the
%other is not. Then $\mathscr{L}_1\mathscr{L}_2$ has infinite
%transcendence degree over either $\mathscr{L}_1$ or $\mathscr{L}_2,$
%and consequently, for at least one index $i \in \{1 , 2\}$ we have
%that
%$$
%L(\fX_{\Gamma_i}) \not\subset \Q \cdot A \cdot L(\fX_{\Gamma_{3-i}})
%$$
%for any finite set $A$ of real numbers.
%\end{thm}
%\begin{proof}
%Let $\Gamma_i$ be $(\cG_i , K_i)$-arithmetic. It is well-known that
%the compactness of $\fX_{\Gamma_i}$ is equivalent to
%$\mathrm{rk}_{K_i}\: \cG_i = 0.$ So, in our situation,
%$$
%\mathrm{rk}_{K_1}\,\cG_1 \neq \mathrm{rk}_{K_2}\,\cG_2,
%$$
%so it follows from Corollary \ref{C:Ji0infty} that condition $(C_i)$
%holds for at least one $i \in \{1 , 2\}.$ Then by Corollary
%\ref{C:LG1} and the subsequent remarks, the respective conditions
%$(T_i)$ and $(N_i)$ hold, yielding our claim.
%\end{proof}

\vskip3mm

Generalizing the notion of length-commensurability, one can define
two Riemannian manifolds $M_1$ and $M_2$ to be ``length-similar" if
there exists a real number $\lambda > 0$ such that $$\Q \cdot L(M_2)
= \lambda \cdot \Q \cdot L(M_1).$$ One can show, however, that for
arithmetically defined locally symmetric space, in most cases, this
notion is redundant, viz. it coincides with the notion of length
commensurability.
\begin{cor}\label{C:simil}
Let $\Gamma_i \subset G_i(\R)$ be a finitely generated Zariski-dense
torsion-free subgroup.
%Let $\fX_{\Gamma_1}$ and $\fX_{\Gamma_2}$ be  locally symmetric
%spaces of $G_1$ and $G_2$ of finite volume {\bf [Do we need to
%assume that we are not in the exceptional case $\mathcal{E}$?]}.
Assume that there exists $\lambda \in \R_{> 0}$ such that
\begin{equation}\label{E:length-sim}
\Q \cdot L(\fX_{\Gamma_1}) = \lambda \cdot \Q \cdot
L(\fX_{\Gamma_2}).
\end{equation}
Then

\vskip2mm

\noindent ${\mathrm{(i)}}$ \parbox[t]{13cm}{$w_1 = w_2$ (hence either $G_1$ and
$G_2$ are of the same type, or one of them is of type $B_n$ and the
other of type $C_n$ for some $n \geqslant 3$) and $K_{\Gamma_1} =
K_{\Gamma_2} =: K$.}

\vskip3mm

\noindent Assume now that ${\Gamma_1}$ and ${\Gamma_2}$ are
arithmetic. Then

\vskip2mm

\noindent  ${\mathrm{(ii)}}$ \parbox[t]{13cm}{$\mathrm{rk}_{\R}\: G_1 =
\mathrm{rk}_{\R}\: G_2,$ and either $G_1 \simeq G_2$ over
$\R,$ or one of the groups is of type $B_n$ and the other is of type
$C_n$;}

\vskip2mm

\noindent ${\mathrm{(iii)}}$ \parbox[t]{13cm}{if $\Gamma_i$ is $(\mathcal{G}_i ,
K)$-arithmetic then $\mathrm{rk}_{K}\: \mathcal{G}_1 =
\mathrm{rk}_K\: \mathcal{G}_2$, and consequently, if one of the
spaces is compact, the other must also be compact;}

\vskip2mm

\noindent ${\mathrm{(iv)}}$ \parbox[t]{13cm}{if $G_1$ and $G_2$ are of the same
type which is different from $A_n,$ $D_{2n+1}$ $(n > 1)$ or $E_6$
then $\fX_{\Gamma_1}$ and $\fX_{\Gamma_2}$ are commensurable, hence
length-commensurable;}

\vskip2mm

\noindent  ${\mathrm{(v)}}$ \parbox[t]{13cm}{if $G_1$ and $G_2$ are of the same
type which is one of the following: $A_n,$ $D_{2n+1}$ $(n > 1)$ or
$E_6,$ then provided that $K_{\Gamma_i} \neq \Q$ for at least one $i
\in \{1 , 2\},$ the spaces $\fX_{\Gamma_1}$ and $\fX_{\Gamma_2}$ are
length-commensurable (although not necessarily commensurable).}
\end{cor}
\begin{proof}
If (\ref{E:length-sim}) holds then obviously $(N_i)$ cannot possibly
hold for either $i = 1$ or $2.$ So, assertion (i) immediately follows from
Theorem 1. Now, if $\Gamma_i$ is $(\cG_i , K)$-arithmetic, then neither 
$(N_1)$ nor $(N_2)$ holds, so neither $(C_1)$ nor $(C_2)$ can hold 
(cf. Corollary \ref{C:LG1}). So by Theorem \ref{T:ArG1} we
have $\mathrm{rk}_{K_v}\: \cG_1 = \mathrm{rk}_{K_v}\: \cG_2$ for all
$v \in V^K$; in particular, $\mathrm{rk}_{\R}\: G_1 =
\mathrm{rk}_{\R}\: G_2$. Moreover, if $G_1$ and $G_2$ are of the same
type then by Theorem \ref{T:ArG3}, the Tits indices over $\R$ of
$G_1$ and $G_2$ are isomorphic, and therefore $G_1 \simeq G_2$, so
assertion (ii) follows. Regarding (iii), the fact that $\mathrm{rk}_K\:
\cG_1 = \mathrm{rk}_K\: \cG_2$ is again a consequence of Theorem
\ref{T:ArG3} in conjunction with Corollary \ref{C:LG1}; to relate
this to the compactness of the corresponding locally symmetric
spaces one argues as in the proof of Theorem 5 above. Finally, assertions
(iv) and (v) follow from Theorems 2 and 3 respectively.
\end{proof}

We note that assertions (iv) and (v) of the above  corollary assert that if
two arithmetically defined locally symmetric spaces of the same
group are not length-commensurable then, under certain assumption,
one cannot make them length-commensurable by scaling the metric on
one of them (cf., however, Theorem 4).

%{\bf [We can say somewhere that assuming Schanuel's conjecture
%enables us to avoid the exceptional case $(\mathcal{E})$ in
%\cite{PR6}. Should we mention some consequences for isospectral
%locally symmetric spaces - these consequences are fairly weak.]}

%\vskip1cm
%
%----------------------------
%
%\vskip5mm
%
%(We note that it can happen that $\mathscr{L}_2 \subset
%\mathscr{L}_1$ even if $\fX_{\Gamma_1}$ and $\fX_{\Gamma_2}$ are
%noncommensurable. For example, let $D_1$ and $D_2$ be the quaternion
%algebras over $\Q$ ramified precisely at $\{2, 3\}$ and $\{2, 3, 5,
%7\}$ respectively. Set $G_i = \mathrm{PSL}_{1 , D_i}$ (the adjoint
%group of $\mathrm{SL}_{1, D_i},$ the group associated with the norm
%one group in $D_i^{\times}$), and let $\Gamma_i$ be a torsion-free
%arithmetic subgroup of $G_i.$ Over $\R,$ both $G_1$ and $G_2$ are
%isomorphic to $G = \mathrm{SL}_2.$ Then the symmetric space $\fX$
%associated with $G$ is the hyperbolic plane $\mathbb{H}^2,$ and the
%corresponding locally symmetric spaces $\fX_{\Gamma_i}$ are
%noncommensurable arithmetically defined hyperbolic 2-manifolds. On
%the other hand, our choice of $D_1$ and $D_2$ implies that every
%maximal subfield of $D_2$ is isomorphic to a maximal subfield of
%$D_1,$ hence $\Q \cdot L(\fX_{\Gamma_2}) \subset \Q \cdot
%L(\fX_{\Gamma_1}),$ and therefore $\mathscr{L}_2 \subset
%\mathscr{L}_1.$ Thus, $\mathscr{L} = \mathscr{L}_1\mathscr{L}_2$
%coincides with $\mathscr{L}_1,$ which according to the theorem has
%infinite transcendence degree over $\mathscr{L}_2.$)

\vskip5mm

\section{Groups of types $B_n$ and $C_n$}

The goal of this section is to prove Theorem 4. Our argument will
heavily rely on the results of \cite{GR}. Here is one of the main
results.
\begin{thm}\label{T:BC1}
{\rm (\cite{GR}, Theorem 1.1)} Let $G_1$ and $G_2$ be connected absolutely
simple adjoint groups of types $B_n$ and $C_n$ $(n \geqslant 3)$
respectively over a field $F$ of characteristic zero, and let
$\Gamma_i$ be a Zariski-dense $(\cG_i, K, S)$-arithmetic subgroup.
Then $\Gamma_1$ and $\Gamma_2$ are weakly commensurable if and only
if  the following conditions hold:

\vskip2mm

\noindent {\rm (1)} \parbox[t]{13cm}{$\mathrm{rk}_{K_v}\,\cG_1 =
\mathrm{rk}_{K_v}\,\cG_2 = n$ (in other words, $\cG_1$ and $\cG_2$
are split over $K_v$) for all nonarchimedean $v \in V^K$, and}

\vskip1mm

\noindent {\rm (2)} \parbox[t]{13cm}{$\mathrm{rk}_{K_v}\,\cG_1 =
\mathrm{rk}_{K_v}\,\cG_2 = 0$ or $n$ (i.e., both $\cG_1$ and
$\cG_2$ are either anisotropic or split) for every archimedean $v \in
V^K$.}
\end{thm}

Furthermore, it has been shown in \cite{GR} that the same two conditions
precisely characterize the situations where $\cG_1$ and $\cG_2$ have
the same isogeny classes of maximal $K$-tori, or, equivalently, 
$\cG_1$ and $\widetilde{\cG}_2$ (the universal cover of $\cG_2$) have the same
isomorphism classes of maximal $K$-tori. We need the following
proposition which has actually been  established in the course of the proof
of Theorem \ref{T:BC1} in \cite{GR}.

\begin{prop}\label{P:BC2}
Notations and conventions be as in Theorem \ref{T:BC1}. Assume that $v_0 \in
V^K$ is such that the corresponding condition (1) or (2) fails. Then 
for at least one $i \in \{1 , 2\}$ there exists a
$K_{v_0}$-isotropic maximal torus $T_i(v_0)$ of $\cG_i$ such that no 
maximal $K$-torus $T_i$ of $\cG_i$ satisfying

\vskip2mm

\noindent \ $(i)$ $\theta_{T_i}(\Ga(K_{T_i}/K)) = W(\cG_i , T_i)$,

\vskip1mm

\noindent $(ii)$ $T_i$ is conjugate to $T_i(v_0)$ by an element of
$\cG_i(K_{v_0})$

\vskip2mm

\noindent is $K$-isogeneous to a maximal
$K$-torus of $\cG_{3-i}$.
\end{prop}

We will now use this proposition to prove the following.

\begin{prop}\label{P:BC3}
Notations and conventions be as in Theorem \ref{T:BC1}. Assume that there
exists $v_0 \in V^K$ such that the corresponding condition (1) or (2)
fails. Then condition $(C_i)$ holds for at least one $i \in
\{1 , 2\}$.
\end{prop}
\begin{proof}
As $G_1$ and $G_2$ are adjoint, $\Gamma_i
\subset \cG_i(K)$ for $i = 1, 2$. Pick $i \in \{1 , 2\}$ and a
maximal $K_{v_0}$-torus $T_i(v_0)$ of $\cG_i$ as in Proposition
\ref{P:BC2}; we will show that property  $(C_i)$ holds for this $i$.
Fix $m \geqslant 1$. Using Theorem \ref{T:Ex1}, we can find maximal
$K$-tori $T_1, \ldots , T_m$ of $\cG_i$ that are independent over $K$
and satisfy the following conditions for each $j \leqslant m$:

\vskip2mm

\noindent $\bullet$ $\theta_{T_j}(\Ga(K_{T_j}/K)) = W(\cG_i , T_j)$,

\vskip1mm

\noindent $\bullet$ \parbox[t]{13cm}{$T_j$ is conjugate to
$T_i(v_0)$ by an element of $\cG_i(K_{v_0})$, and \newline
$\mathrm{rk}_{K_v}\, T_j = \mathrm{rk}_{K_v}\, \cG_i$ for all $v \in
S \setminus \{ v_0 \}$.}

\vskip2mm

\noindent Since $T_i(v_0)$ is $K_v$-isotropic, we have  
$d_{T_j}(S):= \sum_{v\in S}\mathrm{rk}_{K_v} T_j > 0$
no matter whether  or not $v_0$ belongs to $S$. Besides, $T_j$ is
automatically $K$-anisotropic, so it follows from Dirichlet's
Theorem that one can pick an element of infinite order $\gamma_j \in
\Gamma_i \cap T_j(K)$ for each $j \leqslant m$. These elements are
multiplicatively independent by Lemma \ref{L:P1}, so we only need to
show that they are not weakly contained in $\Gamma_{3-i}$. However,
by Theorem \ref{T:10}, a relation of weak containment would imply
that $T_j$ for some $j \leqslant m$ would admit a $K$-isogeny 
onto a maximal $K$-torus $T^{(2)}$ of $\cG_2$. However, this
is impossible(cf.\,Proposition \ref{P:BC2}).
\end{proof}

Let $G_1$ and $G_2$ be connected absolutely simple adjoint algebraic $\R$-groups of
type $B_n$ and $C_n$ $(n \geqslant 3)$ respectively, and let
$\Gamma_i$ be a discrete torsion-free $(\cG_i, K_i)$-arithmetic
subgroup of $\mathcal{G}_i=G_i (\R)$, for $i = 1, 2$. If $K_1 \neq K_2$, then
either condition $(T_1)$ or $(T_2)$ holds for the locally symmetric spaces
$\fX_{\Gamma_1}$ and $\fX_{\Gamma_2}$ by Theorem 1. So, let us assume that $K_1 = K_2 =: K$. If
there exists $v_0 \in V^K$ such that the corresponding condition (1)
or (2) of Theorem \ref{T:BC1} fails, then by Proposition \ref{P:BC3}
the groups $\Gamma_1$ and $\Gamma_2$ satisfy $(C_i)$ for at least
one  $i \in \{1 , 2\}$, and then $(T_i)$ holds for the same $i$
(cf.\,Corollary \ref{C:LG1}). So, to complete the proof of Theorem 4,
it remains to be shown that if conditions (1) and (2) of Theorem
\ref{T:BC1} hold for all $v \in V^K$, then
\begin{equation}\label{E:BC1}
\Q \cdot L(\fX_{\Gamma_2}) = \lambda \cdot \Q \cdot
L(\fX_{\Gamma_1}) \ \ \ \text{where} \ \ \ \lambda =
\sqrt{\frac{2n+2}{2n-1}}.
\end{equation}
We will show that 
provided the conditions (1) and (2) of Theorem \ref{T:BC1} hold, 
given a maximal $K$-torus $T_1$ of $\cG_1$, 
there exists a maximal $K$-torus $T_2$ of ${\widetilde\cG}_2$ and a $K$-isomorphism $T_1 \to T_2$ such that for
any $\gamma_1 \in T_1(K)$, and the corresponding $\gamma_2 \in T_2(K)$, one
can {\it relate} the  following sets $$\{\alpha(\gamma_1)\ |\ \alpha \in \Phi(\cG_1,T_1)\} \ \ \ {\rm {and}}\ \ \ \{\alpha(\gamma_2)\ |\ \alpha \in \Phi({\widetilde\cG}_2,T_2)\},$$  and derive
information about the ratio of the lengths of the closed 
geodesics associated to $\gamma_1$ and $\gamma_2$. The easiest way to do this is to use the description of
maximal $K$-tori of $\cG_1$ and ${\widetilde{\cG}}_2$ in terms of commutative  \'etale
algebras.

The group $\cG_1$ can be realize as the special unitary group
$\mathrm{SU}(A_1 , \tau_1)$ where $A_1 = M_{2n+1}(K)$ and $\tau_1$
is an involution of $A_1$ of orthogonal type (which means that
$\dim_K A^{\tau_1}_1 = (2n + 1)(n + 1)$). Any maximal $K$-torus
$T_1$ of $\cG_1$ corresponds to a maximal commutative \'etale $\tau_1$-invariant
subalgebra $E_1$ of $A_1$ such that $\dim_K E^{\tau_1}_1 = n + 1$;
more precisely, $T = \left( \mathrm{R}_{E_1/K}(\mathrm{GL}_1) \cap
\cG_1 \right)^{\circ}$. It is more convenient for our purposes to
think that $T_1$ corresponds to an embedding $\iota_1 \colon (E_1 ,
\sigma_1) \hookrightarrow (A_1 , \tau_1)$ of algebras with
involution, where $E_1$ is a commutative  \'etale $K$-algebra of dimension $(2n
+ 1)$ equipped with an involution $\sigma_1$ such that $\dim_K
E^{\sigma_1}_1 = n + 1$.

Similarly, the group ${\widetilde{\cG}}_2$ can be realized as the special
unitary group $\mathrm{SU}(A_2 , \tau_2)$, where $A_2$ is a central
simple algebra over $K$ of dimension $4n^2$, and $\tau_2$ is an
involution of $A_2$ of symplectic type (i.e., $\dim_K A^{\tau_2}_2 =
(2n +-1)n$). Furthermore, any maximal $K$-torus $T_2$ corresponds to
an embedding $\iota_2 \colon (E_2 , \sigma_2) \hookrightarrow (A_2 ,
\tau_2)$ of algebras with involution where $E_2$ is a commutative \'etale
$K$-algebra of dimension $2n$ equipped with an involution $\sigma_2$
such that $\dim_K E^{\sigma_2}_2 = n$.

Now, any involutory commutative \'etale algebra $(E_1 , \sigma_1)$ as above
admits a decomposition
$$
(E_1 , \sigma_1) = (\widetilde{E}_1 , \tilde{\sigma}_1) \oplus (K ,
\mathrm{id}_K)
$$
where $\widetilde{E}_1 \subset E_1$ is a $(2n)$-dimensional
$\sigma_1$-invariant subalgebra and $\tilde{\sigma}_1 = \sigma_1
\vert \widetilde{E}_1$; note that $\dim_K \widetilde{E}^{\tilde{\sigma}_1}_1
= n$. It was shown in \cite{GR} using Theorem 7.3 of \cite{PR7} that
if conditions (1) and (2) of Theorem \ref{T:BC1} hold then $(E_1 ,
\sigma_1)$ as above admits an embedding $\iota_1 \colon (E_1 ,
\sigma_1) \hookrightarrow (A_1 , \tau_1)$ if and only if $(E_2 ,
\sigma_2) := (\widetilde{E}_1 , \tilde{\sigma}_1)$ admits an embedding
$\iota_2 \colon (E_2 , \sigma_2) \hookrightarrow (A_2 , \tau_2)$.
This implies that for any maximal $K$-torus $T_1$ of $\cG_1$ there
exists a $K$-isomorphism $\varphi \colon T_1 \to T_2$ onto a maximal
$K$-torus $T_2$ of $\widetilde{\cG}_2$ that is induced by the above correspondence between
the associated algebras $(E_1 , \sigma_1)$ and $(E_2 , \sigma_2)$,
and vice versa. Fix the tori $T_1 , T_2$, the $K$-isomorphism $\varphi$,
the algebras $(E_1 , \sigma_1) , (E_2 , \sigma_2)$ and the
embeddings $\iota_1 , \iota_2$ for the remainder of this section. We
also assume henceforth that the discrete torsion-free
subgroups $\Gamma_i \subset \mathcal{G}_i$ are $(\cG_i ,
K)$-arithmetic. Given $\gamma_1 \in T_1(K) \cap \Gamma_1$, set
$\gamma_2 = \varphi(\gamma_1) \in T_2(K)$. Then there exists $n_2
\geqslant 1$ such that $\gamma^{n_2}_2 \in \Gamma_2$. It follows
from the discussion at the beginning of \S \ref{S:FGeod} that the
ratio
$\ell_{\Gamma_2}(c_{\gamma^{n_2}_2})/\ell_{\Gamma_1}(c_{\gamma_1})$
of the lengths of the corresponding geodesics is a rational multiple
of the ratio
$\lambda_{\Gamma_2}(\gamma_2)/\lambda_{\Gamma_1}(\gamma_1)$. Let us
show that in fact
\begin{equation}\label{E:BC7}
\lambda_{\Gamma_2}(\gamma_2) / \lambda_{\Gamma_1}(\gamma_1) =
\sqrt{\frac{2n+2}{2n - 1}}.
\end{equation}
Indeed, let $x \in E_1$ such that $\iota_1(x) = \gamma_1$. The roots
of the characteristic polynomial of $x$ are of the form $$\lambda_1,
\ldots , \lambda_n,\: \lambda^{-1}_1, \ldots , \lambda^{-1}_n,\: 1$$
for some complex numbers $\lambda_1, \ldots , \lambda_n$. Then
\begin{equation}\label{E:BC8}
\{ \alpha(\gamma_1) \: \vert \: \alpha \in \Phi(\cG_1 , T_1) \} = \{
\lambda^{\pm 1}_i \} \cup \{ \lambda^{\pm 1}_i \cdot \lambda^{\pm
1}_j \: \vert \: i < j \}.
\end{equation}
For the corresponding ``truncated'' element $\tilde{x} \in
\widetilde{E}_1= E_2$, the roots of the characteristic polynomial are
$$
\lambda_1, \ldots , \lambda_n,\: \lambda^{-1}_1, \ldots ,
\lambda^{-1}_n,
$$
and
\begin{equation}\label{E:BC9}
\{ \alpha(\gamma_2) \: \vert \: \alpha \in \Phi(\cG_2 , T_2) \} = \{
\lambda^{\pm 2}_i \} \cup \{ \lambda^{\pm 1}_i \cdot \lambda^{\pm
1}_j \: \vert \: i < j \}.
\end{equation}
Set $\mu_i = \log \vert \lambda_i \vert$. Then it follows from
(\ref{E:BC8}) that
$$
\lambda_{\Gamma_1}(\gamma_1)^2 = \sum_{i = 1}^n (\pm \mu_i)^2 +
\sum_{1 \leqslant i < j \leqslant n} (\pm \mu_i \pm \mu_j)^2 = (4n -
2) \cdot \sum_{i = 1}^n \mu^2_i.
$$
Similarly, we derive from (\ref{E:BC9}) that
$$
\lambda_{\Gamma_2}(\gamma_2)^2 = \sum_{i = 1}^n (\pm 2 \mu_i)^2 +
\sum_{1 \leqslant i < j \leqslant n} (\pm \mu_i \pm \mu_j)^2 = 4(n +
1) \cdot \sum_{i = 1}^n \mu^2_i.
$$
Comparing these equations, we obtain (\ref{E:BC7}). Then the
inclusion $\supset$ in (\ref{E:BC1}) follows immediately, and the
opposite inclusion is established by a symmetric argument,
completing the proof of Theorem 4.

\vskip3mm

Since the symmetric space of the real rank-one form of
type $B_n$ is the (real) hyperbolic space $\mathbb{H}^{2n}$, and the
symmetric space of the real rank-one form of type $C_n$ is the
quaternionic hyperbolic space $\mathbb{H}^{n}_{\mathbf{H}}$, we
obtain the following.

\begin{cor}\label{C:BC1}
Let $M_1$ be an arithmetic quotient of $\mathbb{H}^{2n}$, and $M_2$
be an arithmetic quotient of $\mathbb{H}_{\mathbf{H}}^n$ where $n
\geqslant 3$. Then $M_1$ and $M_2$ satisfy $(T_i)$ and $(N_i)$ for
at least one $i \in \{1 , 2\}$; in particular, $M_1$ and $M_2$ are
not length-commensurable.
\end{cor}

(We see from Theorem 1 that  the same conclusion holds when $M_1$ is as in the above corollary  
but $M_2$ is an
arithmetic quotient of $\mathbb{H}_{\mathbf{H}}^m$ with $m \neq n$.)

\vskip2mm

On the other hand, using Theorem 4, one can construct compact
locally symmetric spaces with isometry groups of types $B_n$ and
$C_n$ $(n \geqslant 3)$, respectively, that are length-similar - so,
these spaces can be made length-commensurable by scaling the metric
on one of them. According to the results of Sai-Kee Yeung
\cite{SKY}, however, scaling will never make these spaces (or their
finite-sheeted covers) isospectral.

\vskip9mm

\centerline{\sc Appendix. Proofs of Theorems \ref{T:equal-rank} and
\ref{T:equal-rank}$'$}

\vskip3mm

First, we need to review some notions pertaining to the Tits index
and recall some of the results established in \cite{PR6}. Let $G$ be
a semi-simple algebraic $K$-group. Pick a maximal $K$-torus $T_0$ of
$G$ that contains a maximal $K$-split torus $S_0$ and choose
coherent orderings on $X(T_0) \otimes_{\Z} \R$ and $X(S_0)
\otimes_{\Z} \R$ (which means that the linear map between these
vector spaces induced by the restriction $X(T_0) \to X(S_0)$ takes
nonnegative elements to nonnegative elements). Let $\Delta_0 \subset \Phi(G ,
T_0)$ denote the system of simple roots corresponding to the chosen
ordering on $X(T_0) \otimes_{\Z} \R.$ Then a root $\alpha \in
\Delta_0$ (or the corresponding vertex in the Dynkin diagram) is
{\it distinguished} in the Tits index of $G/K$ if its restriction to
$S_0$ is nontrivial. Let $\Delta_0^{(d)}$ be the set of distinguished roots in $\Delta_0$ and $P$ be 
the minimal parabolic $K$-subgroup containing $S_0$ determined by the above ordering on 
$\Phi(G,T_0)\,(\subset X(T_0))$. Then $Z_G(S_0)$ is the unique Levi subgroup of $P$ containing 
$T_0$, and $\Delta_0\setminus\Delta_0^{(d)}$ is a basis of its root system with respect to $T_0$. Moreover, 
the set $\Phi(P,T_0)$ of roots of $P$ with respect to $T_0$ is the union of positive roots in $\Phi(G,T_0)$ (positive with respect to the ordering fixed above) and the roots $\Phi(Z_G(S_0), T_0)$ of the subgroup $Z_G(S_0)$; hence, $\Delta_0\setminus \Delta_0^{(d)} = \Delta_0\cap -\Phi(P,T_0)$.  The set of  roots of the unipotent radical of $P$ is the set of all positive roots except the roots which are nonnegative integral linear combination of the roots in $\Delta_0\setminus \Delta_0^{(d)}$. 

The notion of a distinguished vertex is
invariant in the following sense: choose another compatible orderings on
$X(T_0) \otimes_{\Z} \R$ and $X(S_0)
\otimes_{\Z} \R$.  Let ${\Delta'_0} \subset \Phi(G , T_0)$ be the
system of simple roots corresponding to this new ordering and ${\Delta'_0}^{(d)}$ 
be the set of distinguished simple roots. Then 
there
exists a unique element $w$ in the Weyl group $W(G , T_0)$ such that
${\Delta'_0} = w(\Delta_0)$ and we call the identification of
$\Delta_0$ with ${\Delta'_0}$ using $w$ the {\it canonical  identification}. We assert that the
canonical identification identifies distinguished roots  with 
distinguished roots. To see this, note that  if $P'$ is the minimal parabolic $k$-subgroup containing $S_0$ determined 
by the new ordering, then there exists  $n\in N_G(S_0)(K)$ such that $P'=nPn^{-1}$. As $nT_0n^{-1}\subset Z_G(S_0)$, we can find  $z\in Z_G(S_0)(K_{\rm{sep}})$ such that $znT_0(zn)^{-1} =T_0$, i.e., $zn$ normalizes $T_0$, and  $zn(\Delta_0\setminus \Delta_0^{(d}= {\Delta'_0}\setminus {\Delta'_)}^{(d)}$.   It is obvious that $znP(zn)^{-1} = P'$ and that $zn$ carries the set of roots  which are positive with respect to the first ordering into the set of roots which are positive with respect to the second ordering. Therefore, $nz$ carries $\Delta_0$ into  ${\Delta'_0}$, and hence $w$ is its image in the Weyl group. From this we conclude that $w(\Delta_0\setminus \Delta_0^{(d)})= {\Delta'_0}\setminus {\Delta'_0}^{(d)}$, which implies that $w(\Delta_0^{(d)} = {\Delta'_0}^{(d)}$. This  proves our assertion.

We recall that $G$ is
$K$-isotropic if and only if the Tits index of $G/K$ has a
distinguished vertex, and, more generally, $\mathrm{rk}_K\,G$
equals the number of distinguished orbits in $\Delta_0$ under the $*$-action (for the definition and properties of  the $*$-action see \cite{PR6}, \S4).

Let now $T$ be an arbitrary maximal $K$-torus of $G.$ Fix a system
of simple roots $\Delta \subset \Phi(G , T).$ Let $\mathscr{K}$ be a field extension of $K$ such that both $T$ and $T_0$ split over it. Then there exists $g \in G(\mathscr{K})$
such that the inner automorphism $i_g \colon x \mapsto gxg^{-1}$ carries $T_0$ onto $T$ and $i^*_g(\Delta) = \Delta_0.$
Moreover, such a $g$ is unique up to right multiplication by an
element of $T_0({\mathscr K}),$ implying that the identification of $\Delta$ with
$\Delta^0$ provided by $i^*_g$ does not depend on the choice of $g$, and we call it the {\it
canonical identification}. A vertex $\alpha \in \Delta$ is said to {\it correspond
to a distinguished vertex} in the Tits index of $G/K$ if  the 
vertex $\alpha_0 \in \Delta_0$ corresponding to $\alpha$ in the
canonical identification is distinguished; the set of all such vertices in $\Delta$
will be denoted by $\Delta^{(d)}(K).$ Clearly, the group $G$ is
quasi-split over $K$ if and only if $\Delta^{(d)}(K) = \Delta.$  The
notion of canonical identification can be extended in the obvious
way to the situation where we are given two maximal $K$-tori $T_1,
T_2$ of $G$ and the systems of simple roots $\Delta_i \in \Phi(G ,
T_i)$ for $i = 1, 2;$ under the canonical identification
$\Delta^{(d)}_1(K)$ is mapped onto $\Delta^{(d)}_2(K).$ The $*$-action of the absolute Galois group $\Ga(K_{sep}/K)$ on $\Delta_1$ and $\Delta_2$ commutes with the canonical identification of $\Delta_1$ with $\Delta_2$, see Lemma 4.1(a) of \cite{PR6}.  The 
set $\Delta^{(d)}(K)$ is invariant under the $*$-action, so it makes
sense to talk about distinguished orbits. 

Let now $K$ be a number
field. We say that an orbit of the $*$-action in $\Delta$ is
distinguished everywhere if it is contained in $\Delta^{(d)}(K_v)$
for all $v \in V^K.$ The following was established in  \cite{PR6}, Proposition 7.2:

\vskip2mm

\noindent $\bullet$ \parbox[t]{13cm}{An orbit of the $*$-action in
$\Delta$ is distinguished (i.e., is contained in $\Delta^{(d)}(K)$)
if and only if it is distinguished everywhere.}

\vskip2mm

\noindent This implies  the following (Corollary 7.4 in \cite{PR6}):

\vskip2mm

\noindent $\bullet$ \parbox[t]{13.5cm}{Let $G$ be an absolutely almost
simple group of one of the following types: $B_n$ $(n \geqslant 2),$
$C_n$ $(n \geqslant 2),$ $E_7,$ $E_8,$ $F_4$ or $G_2.$ If $G$ is
isotropic over $K_v$ for all real $v \in V^K_{\infty}$, then $G$ is
isotropic over $K.$ Additionally, if $G$ is as above, but not of
type $E_7,$ then $\displaystyle \mathrm{rk}_K\, G = \min_{v \in V^K}
\mathrm{rk}_{K_v}\, G.$}

Before proceeding to the proof of Theorem \ref{T:equal-rank}, we
observe that since by assumption $L_1 = L_2 =: L,$ it follows from
condition (1) in the statement of that theorem that
$$
\theta_{T_i}(\Ga(L_{T_i}/L)) = W(G_i , T_i) \ \ \text{for} \ \ i =
1, 2.
$$
So, the fact that there is a  isogeny $T_1 \to T_2$ defined over $L$ 
implies that $w_1 = w_2.$ Thus, this condition holds in both the 
Theorems \ref{T:equal-rank} and \ref{T:equal-rank}$'$. As we already
mentioned, this implies that either the groups $G_1$ and $G_2$ are
of the same Killing-Cartan type, or one of them is of type $B_n$ and
the other is of type $C_n$ for some $n \geqslant 3;$ in particular,
the groups have the same absolute rank.

\vskip3mm

\noindent {\it Proof of Theorem \ref{T:equal-rank} for types $B_n,$
$C_n,$ $E_8,$ $F_4$ and $G_2$.} As we mentioned above, for these
types we have
$$
\mathrm{rk}_K\,G_i = \min_{v \in V^K} \mathrm{rk}_{K_v}\,G_i \ \
\text{for} \ \ i = 1, 2.
$$
Condition (2) of the theorem implies that $\mathrm{rk}_{K_v}\,G_1 =
\mathrm{rk}_{K_v}\, G_2$ for all $v \in \mathcal{V}.$ On the other
hand, for $v \notin \mathcal{V},$ both $G_1$ and $G_2$ are split
over $K_v,$ which automatically makes the local ranks equal. It
follows that $\mathrm{rk}_K\,G_1 = \mathrm{rk}_K\,G_2.$
Furthermore, inspecting the tables in \cite{Ti}, one observes that
the Tits index of an absolutely almost simple group $G$ of one of
the above types over a local or global field $K$ is completely
determined by its $K$-rank, and our assertion about the local and
global Tits indices of $G_1$ and $G_2$ being isomorphic follows  (in case
$G_1$ and $G_2$ are of the same type). \hfill $\Box$

\vskip3mm

\noindent {\it Proof of Theorem} \ref{T:equal-rank}$'$ {\it for types
$B_n,$ $C_n,$ $E_7,$ $E_8,$ $F_4$ and $G_2$.} It is enough to show
that $G_2$ is ${K_2}_v$-isotropic for all $v \in V^{K_2},$ and in
fact, since $G_2$ is assumed to be quasi-split over ${K_2}_v$ for
all $v \notin \mathcal{V}_2,$ it is enough to check this only for $v
\in \mathcal{V}_2.$ However by our construction, each $v \in
\mathcal{V}_2$ is an extension of some $v_0 \in \mathcal{V}_1.$
Since $G_1$ is $K_1$-isotropic, we have
$$
\mathrm{rk}_{{K_2}_v}\,T_1 \geqslant \mathrm{rk}_{{K_1}_{v_0}}\,
T_1 = \mathrm{rk}_{{K_1}_{v_0}}\,G_1 > 0,
$$
so the existence of a $K_2$-isogeny $T_1 \to T_2$ implies that
$\mathrm{rk}_{{K_2}_v}\: T_2 > 0,$ hence $G_2$ is
${K_2}_v$-isotropic as required. \hfill $\Box$

\vskip3mm

Thus, it remains to prove Theorems \ref{T:equal-rank} and
\ref{T:equal-rank}$'$ assuming that $G_1$ and $G_2$ are of the same
type which is one of the following: $A_n,$ $D_n,$ $E_6$ and $E_7$
(recall that  Theorem \ref{T:equal-rank}$'$ has already been proven for groups of type $E_7$). Then
replacing the isogeny $\pi \colon T_1 \to T_2,$ which is defined
over $K$ in Theorem \ref{T:equal-rank} and over $K_2$ in Theorem
\ref{T:equal-rank}$'$, with a suitable multiple, we may (and we
will) assume that $\pi^*(\Phi(G_2 , T_2)) = \Phi(G_1 , T_1).$
Besides, we may assume through the rest of the appendix that $G_1$
and $G_2$ are adjoint, and then $\pi$ extends to an isomorphism
$\bar{\pi} \colon G_1 \to G_2$ over a separable closure of the field
of definition (cf.\:Lemma 4.3(2) and Remark 4.4 in \cite{PR6}). This has
two consequences that we will need. First, the assumption that $L_1
= L_2$ in Theorem \ref{T:equal-rank} implies that the orbits of the
$*$-action on a system of simple roots $\Delta_1 \subset \Phi(G_1 ,
T_1)$ correspond under $\pi^*$ to the orbits of the $*$-action on
the system of simple roots $\Delta_2 \subset \Phi(G_2 , T_2)$ such
that $\pi^*(\Delta_2) = \Delta_1,$ and this remains true over any
completion $K_v$. Thus, it is enough to prove for each $v \in V^K$
that  $\alpha_1 \in \Delta_1$ corresponds to a distinguished vertex
in the Tits index of $G_1/K_v$ if and only if $\alpha_2 :={\pi^*}^{-1}(\alpha_1)\in
\Delta_2$ corresponds to a
distinguished vertex in the Tits index of $G_2/K_v.$ Similarly, the
assumption that $L_2 \subset K_2L_1$ in Theorem
\ref{T:equal-rank}$'$ implies (in the above notations) that if $O_1
\subset \Delta_1$ is an orbit of the $*$-action, then
$(\pi^*)^{-1}(O_1)$ is a union of orbits of the $*$-action.
Consequently, it is enough to prove that if $\alpha_1 \in \Delta$
corresponds to a distinguished vertex in the Tits index of $G_1/K_1$,  
then $\alpha_2 :={\pi^*}^{-1}(\alpha_1)\in
\Delta_2$ 
corresponds to a distinguished vertex in the Tits index of
$G_2/{K_2}_v$ for all $v \in V^{K_2}.$

Second, given two systems of simple roots $\Delta'_1 , \Delta''_1
\subset \Phi(G_1 , T_1)$ and the corresponding systems of simple
roots $\Delta'_2 , \Delta''_2 \subset \Phi(G_2 , T_2),$ an
identification (induced by an automorphism of the root system)
$\Delta'_1 \simeq \Delta''_1$ is canonical if and only if the
corresponding identification $\Delta'_2 \simeq \Delta''_2$ is
canonical.

\vskip3mm

\noindent {\it Proof of Theorem \ref{T:equal-rank} for the remaining
types.} As above, fix systems of simple roots $\Delta_i \subset
\Phi(G_i , T_i)$ for $i = 1, 2$ so that $\pi^*(\Delta_2) =
\Delta_1.$ We need to show, for each $v \in V^K,$ that a root
$\alpha_1 \in \Delta_1$ corresponds to a distinguished vertex in the
Tits index of $G_1/K_v$ if and only if  $\alpha_2 :={\pi^*}^{-1}(\alpha_1)\in
\Delta_2$ corresponds to a distinguished
vertex in the Tits index of $G_2/K_v.$ This is obvious if both $G_1$
and $G_2$ are quasi-split over $K_v$ as then all the vertices in the
Tits indices of $G_1/K_v$ and $G_2/K_v$ are distinguished. So, it
remains to consider the case where $v \in \mathcal{V}.$ Let
$S^v_{i}$ be the maximal $K_v$-split subtorus of $T_i.$ Since
$\mathrm{rk}_{K_v}\,T_i = \mathrm{rk}_{K_v}\, G_i,$ we see that
$S^v_{i}$ is actually a maximal $K_v$-split torus of $G_i$ for $i =
1, 2,$ and besides, $\pi$ induces an isogeny between $S^v_{1}$ and
$S^v_{2}.$ Pick coherent orderings on $X(S^v_{1}) \otimes_{\Z} \R$
and $X(T_1) \otimes_{\Z} \R$ and on $X(S^v_{2}) \otimes_{\Z} \R$
and $X(T_2) \otimes_{\Z} \R$ that correspond to each other under
$\pi^*$, and let $\Delta^v_i \subset \Phi(G_i , T_i)$ for $i = 1, 2$
be the system of simple roots that corresponds to this (new)
ordering on $X(T_i) \otimes_{\Z} \R;$ clearly, $\pi^*(\Delta^v_2) =
\Delta^v_1.$ Furthermore, let $\alpha^v_i \in \Delta^v_i$ be the
root corresponding to $\alpha_i$ under the canonical identification
$\Delta_i \simeq \Delta^v_i;$ it follows from the above remarks that
$\pi^*(\alpha^v_2) = \alpha^v_1.$ On the other hand, $\alpha_i$
corresponds to a distinguished vertex in the Tits index of $G_i/K_v$
if and only if $\alpha^v_i$ restricts to $T^v_{is}$ nontrivially,
and the required fact follows. \hfill $\Box$

\vskip3mm

\noindent {\it Proof of Theorem} \ref{T:equal-rank}$'$ {\it for the
remaining types.} Let $T^0_1$ be a maximal $K_1$-torus of $G_1$ that
contains a maximal $K_1$-split torus $S^0_{1}.$ As in the
definition of the Tits index of $G_1/K_1,$ we fix coherent ordering
on $X(S^0_{1}) \otimes_{\Z} \R$ and $X(T^0_1) \otimes_{\Z} \R,$ and
let $\Delta^0_1$ denote the system of simple roots in $\Phi(G_1 ,
T^0_1)$ corresponding to this ordering on $X(T^0_1) \otimes_{\Z}
\R.$ Now, pick $g_1 \in G_1$ so that $T_1 = i_{g_1}(T^0_1),$ and let
$$
\Delta_1 = i^*_{g_1}(\Delta^0_1) \subset \Phi(G_1 , T_1).
$$
Furthermore, let $\Delta_2 \subset \Phi(G_2 , T_2)$ be the system of
simple roots such that $\pi^*(\Delta_2) = \Delta_1.$ It follows from
the above discussion that it is enough to prove the following:

\vskip2mm

\noindent $(*)$ \parbox[t]{13.2cm}{\it Let $\alpha^0_1 \in \Delta^0_1$
be distinguished in the Tits index of $G_1/K_1,$ and let $\alpha_1 =
i^*_{g_1}(\alpha^0_1) \in\Delta_1.$ Then $\alpha_2 :={\pi^*}^{-1}(\alpha_1)\in
\Delta_2$ corresponds to a
distinguished vertex of $G_2/{K_2}_v$ for all $v \in V^{K_2}.$}

\vskip2mm

\noindent Since $G_2$ is quasi-split over ${K_2}_v$ for $v \notin
\mathcal{V}_2,$ it is enough to prove $(*)$ assuming that $v \in
\mathcal{V}_2.$ By our construction, $v$ is an extension to $K_2$ of
some $v_0 \in \mathcal{V}_1.$ Since $\mathrm{rk}_{{K_1}_{v_0}}\,T_1
= \mathrm{rk}_{{K_1}_{v_0}}\, G_1,$ the maximal ${K_1}_{v_0}$-split
subtorus $S^{v}_{1}$ of $T_1$ is a maximal ${K_1}_{v_0}$-split
torus of $G_1,$ so it follows from the conjugacy of maximal split
tori (cf.\,\cite{Spr}, 15.2.6) that we can find an element $h_1$ of $G_1$, which is rational over a finite extension of ${K_1}_{v_0}$, such
that
$$
T_1 = i_{h_1}(T^0_1) \ \ \text{and} \ \ S^{v}_{1} \supset
i_{h_1}(S^0_{1}).
$$
We claim that furthermore that to prove $(*)$ it suffices to find a
different ordering on $X(T^0_1) \otimes_{\Z} \R$ (depending on $v$)
that induces the same ordering on $X(S^0_{1}) \otimes_{\Z} \R$
(this ordering on $X(T^0_1) \otimes_{\Z} \R$ will be  referred to as the 
{\it new}{ ordering}, while the ordering fixed earlier will be called the {\it
old}{ ordering}) such that if $\Delta^{0v}_1 \subset \Phi(G_1 , T^0_1)$ is the
system of simple root corresponding to the new ordering, $i^* \colon
\Delta^0_1 \simeq \Delta^{0v}_1$ is the canonical identification,
$\alpha^{0v}_1 = i^*(\alpha^0_1),$ and $\Delta^v_1 \subset \Phi(G_1
, T_1)$ and $\alpha^v_1 \in \Delta^v_1$ are such that
\begin{equation}\tag{A.1}\label{E:A1}
i^*_{h_1}(\Delta^v_1) = \Delta^{0v}_1 \ \ \text{and} \ \
i^*_{h_1}(\alpha^v_1) = \alpha^{0v}_1
\end{equation}
then for $\Delta^v_2 \subset \Phi(G_2 , T_2)$ such that
$\pi^*(\Delta^v_2) = \Delta^v_1,$  the root $\alpha^v_2 \in
\Delta^v_2$ such that $\pi^*(\alpha^v_2) = \alpha^v_1$ corresponds
to a distinguished vertex in the Tits index of $G_2/{K_2}_v.$
Indeed, the identification $\Delta_1 \simeq \Delta^v_1$ given by
$i^*_{h_1} \circ i^* \circ (i^*_{g_1})^{-1}$ is canonical and takes
$\alpha_1$ to $\alpha^v_1.$ It follows that the canonical
identification of $\Delta_2$ with $\Delta^v_2$ takes $\alpha_2$ to
$\alpha^v_2,$ so the fact that $\alpha^v_2$ corresponds to a
distinguished vertex in the Tits index of $G_2/{K_2}_v$ implies that
the same is true for $\alpha_2,$ as required. What is crucial for
the rest of the argument is that due to the invariance of the Tits
index, the root $\alpha^{0v}_1$ is distinguished in the Tits index
of $G_1/K_1,$ i.e. its restriction to $S^0_{1}$ is nontrivial.

To construct a new ordering on $X(T^0_1) \otimes_{\Z} \R$ with the
required properties, we let $T^{0v}_2$ denote a maximal
${K_2}_v$-torus of $G_2$ that contains a maximal ${K_2}_v$-split
torus $S^{0v}_{2}.$ There exists $h_2 \in G_2({K_2}_v)$ such that
\begin{equation}\tag{A.2}\label{E:A2}
T_2 = i_{h_2}(T^{0v}_2) \ \ \text{and} \ \ S^v_{2} \subset
i_{h_2}(S^{0v}_{2}),
\end{equation}
where $S^v_{2}$ is the maximal ${K_2}_v$-subtorus of $T_2.$ Since
$\pi$ is defined over $K_2,$ it follows from (\ref{E:A1}) and
(\ref{E:A2}) that $\varphi := i^{-1}_{h_2} \circ \pi \circ i_{h_1}$
has the property $$S^0_{1} \subset \varphi^{-1}(S^{0v}_{2}) =:
\mathcal{S}.$$ 
Lift the old ordering on $X(S^0_{1}) \otimes_{\mathbb{Z}} \mathbb{R}$ first to a coherent ordering 
on $X(\mathcal{S})
\otimes_{\mathbb{Z}} \mathbb{R}$, and then lift the latter to a coherent ordering on 
$X(T^0) \otimes_{\mathbb{Z}} \mathbb{R}$.
 We claim that this
ordering on $X(T^0) \otimes_{\Z} \R$ can be taken for a new
ordering. Indeed, as above, let $\Delta^{0v}_1 \subset \Phi(G_1 ,
T^0_1)$ be the system of simple roots corresponding to the new
ordering, and let $\alpha^{0v}_1 \in \Delta^{0v}_1$ be the root
corresponding to $\alpha^0_1 \in \Delta^0_1$ under the canonical
identification $\Delta^0_1 \simeq \Delta^{0v}_1;$ as we already
mentioned, $\alpha^{0v}_1$ restricts to $S^0_{1}$ nontrivially. By
construction, the system of simple roots $\Delta^{0v}_2 \subset
\Phi(G_2 , T^{0v}_2)$ such that $\varphi^*(\Delta^{0v}_2) =
\Delta^{0v}_1$ corresponds to a coherent choice of orderings on
$X(S^{0v}_{2}) \otimes_{\Z} \R$ and $\alpha^{0v}_2 \in
\Delta^{0v}_2$ such that $\varphi^*(\alpha^{0v}_2) = \alpha^{0v}_1$
restricts to $S^{0v}_{2}$ nontrivially, i.e. is a distinguished
vertex in the Tits index of $G_2/{K_2}_v.$ On the other hand, in the
above notations we have
$$
i^*_{h_1}(\alpha^v_1) = \alpha^{0v}_1, \ \ \pi^*(\alpha^v_2) =
\alpha^v_1 \ \ \text{and} \ \ i^*_{h_2}(\alpha^v_2) = \alpha^{0v}_2.
$$
Thus, $\alpha^v_2 \in \Delta^v_2$ corresponds to a distinguished
vertex in the Tits index of $G_2/{K_2}_v,$ as required. \hfill
$\Box$

\vskip1cm

\bibliographystyle{amsplain}

\bibliographystyle{amsplain}

\end{document}